\documentclass[12pt,final]{amsart}
\usepackage{hyperref}

\usepackage{xspace,amssymb,epsfig,euscript,enumerate,amsmath,nicefrac}
\DeclareMathAlphabet{\mathbbm}{U}{bbm}{m}{n}
\usepackage{tikz,etex,xspace,afterpage}
\usepackage{epstopdf,ifpdf,todonotes, enumitem}

\usepackage{stackengine}

\stackMath
\usepackage[notref,notcite]{showkeys}
\usepackage[margin=2.5cm,footskip=25pt,headheight=20pt]{geometry}

\usepackage{fancyhdr,mathrsfs}

  \usepackage{pdfsync}

\begin{document}
\newtheoremstyle{all}
  {11pt}
  {11pt}
  {\slshape}
  {}
  {\bfseries}
  {}
  {.5em}
  {}


\theoremstyle{all}

\newtheorem{itheorem}{Theorem}
\newtheorem{theorem}{Theorem}[section]
\newtheorem*{theoremfourfive}{Theorem 4.5}
\newtheorem*{proposition*}{Proposition}
\newtheorem{proposition}[theorem]{Proposition}
\newtheorem{corollary}[theorem]{Corollary}
\newtheorem{lemma}[theorem]{Lemma}
\newtheorem{assumption}[theorem]{Assumption}
\newtheorem{definition}[theorem]{Definition}
\newtheorem{ques}[theorem]{Question}
\newtheorem{conj}[theorem]{Conjecture}

\theoremstyle{remark}
\newtheorem{remark}[theorem]{Remark}
\newtheorem{example}{Example}[section]
\renewcommand{\theexample}{{\arabic{section}.\roman{example}}}
\newcommand{\nc}{\newcommand}
\newcommand{\renc}{\renewcommand}
\numberwithin{equation}{section}
\renc{\theequation}{\arabic{section}.\arabic{equation}}

\newcounter{subeqn}
\renewcommand{\thesubeqn}{\theequation\alph{subeqn}}
\newcommand{\subeqn}{%
  \refstepcounter{subeqn}
  \tag{\thesubeqn}
}\makeatletter
\@addtoreset{subeqn}{equation}
\newcommand{\newseq}{%
  \refstepcounter{equation}
}
  \nc{\kac}{\kappa^C}
\nc{\alg}{T}
\nc{\Lco}{L_{\la}}
\nc{\qD}{q^{\nicefrac 1D}}
\nc{\ocL}{M_{\la}}
\nc{\excise}[1]{}
\nc{\Dbe}{D^{\uparrow}}
\nc{\Dfg}{D^{\mathsf{fg}}}
\nc{\KLRwei}{\EuScript{W}}
\nc{\coset}{\EuScript{W}}
\nc{\zero}{o}
\nc{\defr}{\operatorname{def}}
\nc{\op}{\operatorname{op}}
\nc{\Sym}{\operatorname{Sym}}
\nc{\Symt}{S}
\nc{\hatD}{\widehat{\Delta}}
\nc{\tr}{\operatorname{tr}}
\newcommand{\Mirkovic}{Mirkovi\'c\xspace}
\nc{\tla}{\mathsf{t}_\la}
\nc{\llrr}{\langle\la,\rho\rangle}
\nc{\lllr}{\langle\la,\la\rangle}
\nc{\K}{\mathbbm{k}}
\nc{\Stosic}{Sto{\v{s}}i{\'c}\xspace}
\nc{\cd}{\mathcal{D}}
\nc{\cT}{\mathcal{T}}
\nc{\vd}{\mathbb{D}}
\nc{\lift}{\gamma}
\nc{\cox}{h}
\nc{\Aut}{\operatorname{Aut}}
\nc{\R}{\mathbb{R}}
\renc{\wr}{\operatorname{wr}}
  \nc{\Lam}[3]{\La^{#1}_{#2,#3}}
  \nc{\Lab}[2]{\La^{#1}_{#2}}
  \nc{\Lamvwy}{\Lam\Bv\Bw\By}
  \nc{\Labwv}{\Lab\Bw\Bv}
  \nc{\nak}[3]{\mathcal{N}(#1,#2,#3)}
  \nc{\hw}{highest weight\xspace}
  \nc{\al}{\alpha}

\newcommand{\dgmod}{\operatorname{-dg-mod}}
  \nc{\be}{\beta}
  \nc{\bM}{\mathbf{m}}
  \nc{\Bu}{\mathbf{u}}

  \nc{\bkh}{\backslash}
  \nc{\Bi}{\mathbf{i}}
  \nc{\Bm}{\mathbf{m}}
  \nc{\Bj}{\mathbf{j}}
 \nc{\Bk}{\mathbf{k}}
\newcommand{\bS}{\mathbb{S}}
\newcommand{\bT}{\mathbb{T}}
\newcommand{\bt}{\mathbbm{t}}

\nc{\bd}{\mathbf{d}}
\nc{\second}{\tau}
\nc{\D}{\mathcal{D}}
\nc{\mmod}{\operatorname{-mod}}  
\newcommand{\red}{\mathfrak{r}}

\nc{\RAA}{R^\A_A}
  \nc{\Bv}{\mathbf{v}}
  \nc{\Bw}{\mathbf{w}}
\nc{\Id}{\operatorname{Id}}
\nc{\Cth}{S_h}
\nc{\Cft}{S_1}
\def\MHM{{\operatorname{MHM}}}

\newcommand{\cD}{\mathcal{D}}
\newcommand{\LCP}{\operatorname{LCP}}
  \nc{\By}{\mathbf{y}}
\nc{\eE}{\EuScript{E}}
  \nc{\Bz}{\mathbf{z}}
  \nc{\coker}{\mathrm{coker}\,}
  \nc{\C}{\mathbb{C}}
\nc{\ab}{{\operatorname{ab}}}
\nc{\wall}{\mathbbm{w}}
  \nc{\ch}{\mathrm{ch}}
  \nc{\de}{\delta}
  \nc{\ep}{\epsilon}
  \nc{\Rep}[2]{\mathsf{Rep}_{#1}^{#2}}
  \nc{\Ev}[2]{E_{#1}^{#2}}
  \nc{\fr}[1]{\mathfrak{#1}}
  \nc{\fp}{\fr p}
  \nc{\fq}{\fr q}
  \nc{\fl}{\fr l}
  \nc{\fgl}{\fr{gl}}
\nc{\rad}{\operatorname{rad}}
\nc{\Ind}{\operatorname{Ind}}
\newcommand{\sS}{\mathsf{S}} 
\newcommand{\sJ}{\mathsf{J}} 
\newcommand{\sT}{\mathsf{T}}
\newcommand{\cP}{\mathcal{P}}
\newcommand{\cellb}{a}
  \nc{\GL}{\mathrm{GL}}
\newcommand{\arxiv}[1]{\href{http://arxiv.org/abs/#1}{\tt arXiv:\nolinkurl{#1}}}
  \nc{\Hom}{\mathrm{Hom}}
  \nc{\im}{\mathrm{im}\,}
  \nc{\La}{\Lambda}
  \nc{\la}{\lambda}
  \nc{\mult}{b^{\mu}_{\la_0}\!}
  \nc{\mc}[1]{\mathcal{#1}}
  \nc{\om}{\omega}
\nc{\gl}{\mathfrak{gl}}
  \nc{\cF}{\mathcal{F}}
 \nc{\cC}{\mathcal{C}}
  \nc{\Mor}{\mathsf{Mor}}
  \nc{\HOM}{\operatorname{HOM}}
  \nc{\Ob}{\mathsf{Ob}}
  \nc{\Vect}{\mathsf{Vect}}
\nc{\gVect}{\mathsf{gVect}}
  \nc{\modu}{\mathsf{-mod}}
\nc{\pmodu}{\mathsf{-pmod}}
  \nc{\qvw}[1]{\La(#1 \Bv,\Bw)}
  \nc{\van}[1]{\nu_{#1}}
  \nc{\Rperp}{R^\vee(X_0)^{\perp}}
  \nc{\si}{\sigma}
\nc{\sgns}{{\boldsymbol{\sigma}}}
  \nc{\croot}[1]{\al^\vee_{#1}}
\nc{\di}{\mathbf{d}}
  \nc{\SL}[1]{\mathrm{SL}_{#1}}
  \nc{\Th}{\theta}
  \nc{\vp}{\varphi}
  \nc{\wt}{\mathrm{wt}}
\nc{\te}{\tilde{e}}
\nc{\tf}{\tilde{f}}
\nc{\hwo}{\mathbb{V}}
\nc{\soc}{\operatorname{soc}}
\nc{\cosoc}{\operatorname{cosoc}}
 \nc{\Q}{\mathbb{Q}}
\nc{\LPC}{\mathsf{LPC}}
  \nc{\Z}{\mathbb{Z}}
  \nc{\Znn}{\Z_{\geq 0}}
  \nc{\ver}{\EuScript{V}}
  \nc{\Res}{\operatorname{Res}}
  \nc{\simples}{X}
  \nc{\edge}{\EuScript{E}}
  \nc{\Spec}{\mathrm{Spec}}
  \nc{\tie}{\EuScript{T}}
  \nc{\ml}[1]{\mathbb{D}^{#1}}
  \nc{\fQ}{\mathfrak{Q}}
        \nc{\fg}{\mathfrak{g}}
        \nc{\ft}{\mathfrak{t}}
  \nc{\Uq}{U_q(\fg)}
        \nc{\bom}{\boldsymbol{\omega}}
\nc{\bla}{{\underline{\boldsymbol{\la}}}}
\nc{\bmu}{{\underline{\boldsymbol{\mu}}}}
\nc{\bal}{{\boldsymbol{\al}}}
\nc{\bet}{{\boldsymbol{\eta}}}
\nc{\rola}{X}
\nc{\wela}{Y}
\nc{\fM}{\mathfrak{M}}
\nc{\fX}{\mathfrak{X}}
\nc{\fH}{\mathfrak{H}}
\nc{\fE}{\mathfrak{E}}
\nc{\fF}{\mathfrak{F}}
\nc{\fI}{\mathfrak{I}}
\nc{\qui}[2]{\fM_{#1}^{#2}}
\nc{\cL}{\mathcal{L}}
\nc{\ca}[2]{\fQ_{#1}^{#2}}
\nc{\cat}{\mathcal{V}}
\nc{\cata}{\mathfrak{V}}
\nc{\catf}{\mathscr{V}}
\nc{\hl}{\mathcal{X}}
\nc{\hld}{\EuScript{X}}
\nc{\hldbK}{\EuScript{X}^{\bla}_{\bar{\mathbb{K}}}}
\nc{\Iwahori}{\EuScript{I}}
\nc{\WC}{\EuScript{C}}

\nc{\pil}{{\boldsymbol{\pi}}^L}
\nc{\pir}{{\boldsymbol{\pi}}^R}
\nc{\cO}{\mathcal{O}}
\nc{\Ko}{\text{\Denarius}}
\nc{\Ei}{\fE_i}
\nc{\Fi}{\fF_i}
\nc{\fil}{\mathcal{H}}
\nc{\brr}[2]{\beta^R_{#1,#2}}
\nc{\brl}[2]{\beta^L_{#1,#2}}
\nc{\so}[2]{\EuScript{Q}^{#1}_{#2}}
\nc{\EW}{\mathbf{W}}
\nc{\rma}[2]{\mathbf{R}_{#1,#2}}
\nc{\Dif}{\EuScript{D}}
\nc{\MDif}{\EuScript{E}}
\renc{\mod}{\mathsf{mod}}
\nc{\modg}{\mathsf{mod}^g}
\nc{\fmod}{\mathsf{mod}^{fd}}
\nc{\id}{\operatorname{id}}
\nc{\compat}{\EuScript{K}}
\nc{\DR}{\mathbf{DR}}
\nc{\End}{\operatorname{End}}
\nc{\Fun}{\operatorname{Fun}}
\nc{\Ext}{\operatorname{Ext}}
\nc{\tw}{\tau}
\nc{\A}{\EuScript{A}}
\nc{\Loc}{\mathsf{Loc}}
\nc{\eF}{\EuScript{F}}
\nc{\LAA}{\Loc^{\A}_{A}}
\nc{\perv}{\mathsf{Perv}}
\nc{\gfq}[2]{B_{#1}^{#2}}
\nc{\qgf}[1]{A_{#1}}
\nc{\qgr}{\qgf\rho}
\nc{\tqgf}{\tilde A}
\nc{\Tr}{\operatorname{Tr}}
\nc{\Tor}{\operatorname{Tor}}
\nc{\cQ}{\mathcal{Q}}
\nc{\st}[1]{\Delta(#1)}
\nc{\cst}[1]{\nabla(#1)}
\nc{\ei}{\mathbf{e}_i}
\nc{\Be}{\mathbf{e}}
\nc{\Hck}{\mathfrak{H}}
\renc{\P}{\mathbb{P}}
\nc{\bbB}{\mathbb{B}}
\nc{\ssy}{\mathsf{y}}
\nc{\cI}{\mathcal{I}}
\nc{\cG}{\mathcal{G}}
\nc{\cH}{\mathcal{H}}
\nc{\coe}{\mathfrak{K}}
\nc{\pr}{\operatorname{pr}}
\nc{\bra}{\mathfrak{B}}
\nc{\rcl}{\rho^\vee(\la)}
\nc{\tU}{\mathcal{U}}
\nc{\wgmod}{\operatorname{-wgmod}}
\nc{\dU}{{\stackon[8pt]{\tU}{\cdot}}}
\nc{\dT}{{\stackon[8pt]{\cT}{\cdot}}}

\nc{\RHom}{\mathrm{RHom}}
\nc{\tcO}{\tilde{\cO}}
\nc{\Yon}{\mathscr{Y}}
\nc{\sI}{{\mathsf{I}}}
\nc{\sptc}{\ft_{1,\Z}}
\nc{\spt}{\ft_{1,\R}}
\nc{\Bpsi}{u}
\nc{\acham}{\eta}
\nc{\hyper}{\mathsf{H}}
\nc{\AF}{\EuScript{F}\ell}
\nc{\VB}{\EuScript{X}}
\nc{\OHiggs}{\cO_{\operatorname{Higgs}}}
\nc{\OCoulomb}{\cO_{\operatorname{Coulomb}}}
\nc{\tOHiggs}{\tilde\cO_{\operatorname{Higgs}}}
\nc{\tOCoulomb}{\tilde\cO_{\operatorname{Coulomb}}}
\nc{\indx}{\mathcal{I}}
\nc{\redu}{K}
\nc{\Ba}{\mathbf{a}}
\nc{\Bb}{\mathbf{b}}
\nc{\Lotimes}{\overset{L}{\otimes}}
\nc{\AC}{C}
\nc{\ideal}{\mathscr{I}}
\nc{\ACs}{\mathscr{C}}
\nc{\Stein}{\mathscr{X}}
\newcommand{\cOg}{\mathcal{O}_{\!\operatorname{g}}}
\newcommand{\tcOg}{\mathcal{\tilde O}_{\!\operatorname{g}}}
\newcommand{\dOg}{D_{\cOg}}
\newcommand{\preO}{p\cOg}
\newcommand{\dpreO}{D_{p\cOg}}
\nc{\No}{H}
\nc{\pairs}{\mathcal{SP}}
\nc{\To}{Q}
\nc{\tNo}{\tilde{H}}
\nc{\tTo}{\tilde{Q}}
\nc{\flav}{\phi}
\nc{\tF}{\tilde{F}}
\nc{\auto}{\alpha}
\nc{\zetal}{\zeta}
\nc{\zetap}{\beta}
\nc{\DO}{\mathsf{U}}
\nc{\Wei}{\EuScript{W}}
\nc{\bQ}{\mathbf{Q}}
\nc{\lp}{\nicefrac{\ell}{p}}
\newcommand{\CC}{\mathbb{C}}
\newcommand{\RR}{\mathbb{R}}
\newcommand{\ZZ}{\mathbb{Z}}
\newcommand{\cM}{\mathcal{M}}
\newcommand{\cN}{\mathcal{N}}
\newcommand{\sK}{\mathscr{K}}
\newcommand{\sL}{\mathscr{L}}
\newcommand{\sM}{\mathscr{M}}
\newcommand{\sP}{\mathscr{P}}
\newcommand{\sU}{\mathscr{U}}
\newcommand{\T}{\mathbf{T}}
\newcommand{\U}{\mathbf{U}}
\newcommand{\diag}{\mathrm{diag}}
\newcommand{\Frac}{\mathrm{Frac}}
\nc{\tWei}{\widetilde{\EuScript{W}}}
\nc{\tl}{\mathsf{t}_{(\ell)}}
\nc{\tn}{\mathsf{t}_{(n)}}
\nc{\cg}{g}
\nc{\KZ}{\mathsf{KZ}}
\setcounter{tocdepth}{2}
\newcommand{\thetitle}{Rational Cherednik algebras of $G(\ell,p,n)$\\ from the Coulomb perspective}
\newcommand{\theshorttitle}{Rational Cherednik algebras of $G(\ell,p,n)$ from the Coulomb perspective}

\renc{\theitheorem}{\Alph{itheorem}}
\newcommand{\btodo}{\todo[inline,color=blue!20]}
\newcommand{\etodo}{\todo[inline,color=red!20]}

\baselineskip=1.1\baselineskip

 \usetikzlibrary{decorations.pathreplacing,backgrounds,decorations.markings,shapes.geometric,fit}
 \tikzset{tstar/.style={fill=white,draw,star,star points=5,star point ratio=0.45,inner
		sep=3pt}}
\tikzset{ucircle/.style={fill=black,circle,inner sep=2pt}}

\tikzset{wei/.style={draw=red,double=red!40!white,double distance=1.5pt,thin}}
\tikzset{awei/.style={draw=blue,double=blue!40!white,double distance=1.5pt,thin}}
\tikzset{bdot/.style={fill,circle,color=blue,inner sep=3pt,outer
    sep=0}}
\tikzset{dir/.style={postaction={decorate,decoration={markings,
    mark=at position .8 with {\arrow[scale=1.3]{>}}}}}}
\tikzset{rdir/.style={postaction={decorate,decoration={markings,
    mark=at position .8 with {\arrow[scale=1.3]{<}}}}}}
\tikzset{edir/.style={postaction={decorate,decoration={markings,
    mark=at position .2 with {\arrow[scale=1.3]{<}}}}}}\begin{center}
\noindent {\large  \bf \thetitle}
\bigskip

  \begin{tabular}{c@{\hspace{10mm}}c}
     {\sc\large Elise LePage}&{\sc\large Ben Webster}\\
 \it Department of Physics,&   \it Department of Pure Mathematics,\\ 
   \it  University of California, Berkeley&  \it 
   University of Waterloo \&\\
 \it Berkeley, CA, USA& \it Perimeter Institute for Mathematical Physics \\
 & \it Waterloo, ON, Canada\\
{\tt elepage@berkeley.edu }&{\tt ben.webster@uwaterloo.ca}
 \end{tabular}
\end{center}
\bigskip
{\small
\begin{quote}
\noindent {\em Abstract.}
We prove a number of results on the structure and representation theory of the rational Cherednik algebra of the imprimitive reflection group $G(\ell,p,n)$.   In particular, we:
(1) show a relationship to the Coulomb branch construction of Braverman, Finkelberg and Nakajima, and 3-dimensional quantum field theory; (2) show that the spherical Cherednik algebra carries the structure of a principal Galois order; (3) construct a graded lift of category $\mathcal{O}$ and the larger category of Dunkl-Opdam modules, whose simple modules have the properties of a dual canonical basis and (4) give the first classification of simple Dunkl-Opdam modules for the rational Cherednik algebra of the imprimitive reflection group $G(\ell,p,n)$.
\end{quote}
}
\section{Introduction}

Our goal in this paper is to extend a number of results in the theory of rational Cherednik algebras $\mathsf{H}_1$ for the group $G(\ell, 1, n)$ of monomial $n\times n$ matrices with $\ell$th root of unity entries to the same algebras $\mathsf{H}_p$  attached to the imprimitive subgroups $G(\ell,p,n)$ for some $p\mid \ell$.  These results include:
\begin{enumerate}
    \item Work of Kodera and Nakajima \cite{KoNa}, expanded on by Braverman, Etingof \cite{BEF} and the second author \cite{Webalt}, interpreting the spherical subalgebra of the Cherednik algebra $\mathsf{H}_1$ as the Coulomb branch of a $\mathcal{N}=4$ supersymmetric gauge theory.  We find that the spherical subalgebra of $\mathsf{H}_p$  appears in the same theory after one incorporates certain non-trivial line defects (Theorem \ref{thm:lines-cherednik}).
    \item Work of the second author \cite[Prop 4.2]{WebGT} shows that the spherical subalgebra $\mathsf{H}_1^{\operatorname{sph}}$ (and more generally, any BFN Coulomb branch) is a principal Galois order (in the sense of \cite[Def. 2.24]{Hartwig}).  We show that $\mathsf{H}_p^{\operatorname{sph}}$ has the same property, which is unsurprising given the previous results, but is not a direct consequence of them.  
    \item Work of the second author \cite{WebRou,Webalt} and parallel results of Rouquier, Shan, Varagnolo and Vasserot \cite{RSVV, SVV} showing that category $\cO$ for these Cherednik algebras have a graded lift with a number of desirable properties, including Koszulity and a positivity property called ``mixedness'' which allows us to interpret decomposition numbers in these categories as the coefficients of a canonical basis.  All of these same properties hold in $G(\ell,p,n)$, (Corollary \ref{cor:smash-Koszul}) though we do not have as clean an interpretation of the Grothendieck group in this case as the Fock space provided for $G(\ell, 1, n)$ 
    \item Work of the second author (which generalizes work of Suzuki \cite{SuzukiCherednik} from the cyclotomic Cherednik algebra of $S_n$) classifying the Dunkl-Opdam modules for the Cherednik algebra.  Straightforward application of Clifford theory allows us to extend this to $G(\ell,p,n)$ (Theorem \ref{thm:DO-classification}).  
\end{enumerate}
These results show that the Cherednik algebras for all non-exceptional complex reflection groups have a common package of results regulating their representation theory.  It remains an interesting question whether these extend in any systematic way to the exceptional cases, though certainly very different techniques would be required.

From a converse perspective, the construction we give to extract an algebra from a 3d quantum field theory (based on the existence of a certain cyclic symmetry of the matter content) seems to be novel in the physics literature and it seems worth exploring whether this can be used to net other interesting non-commutative algebras.  
\subsection*{Acknowledgements}

We thank Stephen Griffeths and Arun Ram for many useful discussions on these subjects.

B. W. is supported by an NSERC Discovery Grant. This research was supported in part by Perimeter Institute for Theoretical Physics. Research at Perimeter Institute is supported by the Government of Canada through the Department of Innovation, Science and Economic Development Canada and by the Province of Ontario through the Ministry of Research, Innovation and Science.
\section{Background}
\subsection{Algebras with a cyclic group action}

Let $p$ be a (possible composite) integer and $\K$ be a field such that the polynomial $x^p-1$ is split and separable; i.e. the group $\mu_p(\K)$ is a cyclic group of order $p$.  Note, this means that $p$ is invertible in $\K$.
Consider a $\K$-algebra $A$ equipped with an automorphism $\auto$ of order $p$.  Let $C$ denote the subgroup of algebra automorphisms generated by $\auto$ and for any $A$-module $M$, let $M^{\auto}$ be the same underlying vector space with its $A$-action twisted by $\auto$: that is $a\cdot m=a^{\auto}m$. Since $x^p-1$ splits, we have that $A$ decomposes into isotypic components under the action of $C$, that is, the eigenspaces of $\auto$: \[A_\rho=\{a\in A | a^{\auto}=\rho a\}\]  with eigenvalue $\rho$ necessarily in $\mu_p(\K)$. The subspace $A_\rho$ is the image of the projection $\pi_\rho =\frac{1}p(1+\rho^{-1}\auto+\cdots \rho^{1-p} \auto^{p-1}) $ acting on $A$.
\begin{definition}
We call an action of $C$ on $A$ {\bf Morita} if for every $\rho\in \mu_p(\K)$, we have an equality $A_{\rho}A_{\rho^{-1}} =A^C$ for all $\rho$.
\end{definition}

Let $\tilde{A}=A\#C$ be the smash product algebra: this is isomorphic to $A\otimes_{\K}\K[C]$ as a vector space, with the commutation rule $\auto a \auto^{-1}=a^{\auto}$.  
\begin{lemma}\label{lem:Morita}
The algebras $A^C$ and $\tilde{A}$ are Morita equivalent via the bimodule $A$ if and only if the action of $C$ is Morita.
\end{lemma}
\begin{proof}
By \cite[A.9]{RamRam}, this Morita equivalence holds if and only if $\pi_1$ generates $\tilde{A}$ as a 2-sided ideal.  In this case, we can write $1_{\tilde{A}}=\sum_{i=1}^m a_i\pi_1b_i$.  Multiplying this equation on both sides by the idempotent $\pi_\rho$, we see that:
\[\pi_\rho =\sum_{i=1}^m \pi_{\rho} a_i\pi_1  b_i\pi_{\rho}.\] Projecting $a_i$ and $b_i$ to the appropriate eigenspaces, we obtain $a_i'\in A_\rho, b'_i\in A_{\rho^{-1}}$ 
such that \[ \pi_{\rho} a_i\pi_1 = \pi_{\rho} a'_i=a_i'\pi_1\qquad \pi_1  b_i\pi_{\rho}=\pi_1  b_i'=  b_i'\pi_1\] 
so we have that $1_A=\sum_{i=1}^m a_i'b_i'$, and as desired, $A_{\rho}A_{\rho^{-1}} =A^C.$

On the other hand, if $A_{\rho}A_{\rho^{-1}} =A^C,$ then we can run this proof backwards: we have $1=\sum c_id_i$ for $c_i\in A_\rho, d_i\in A_{\rho^{-1}}$, so 
\[ \pi_{\rho} = \sum \pi_{\rho} c_id_i\pi_{\rho} =\sum c_i\pi_{1} d_i.\]
Thus $\pi_{\rho}\in \tilde{A}\pi_1\tilde{A}$ for all $\rho$.   Since $1_{\tilde{A}}=\sum_{\rho\in \mu_p(\C)} \pi_\rho$, this shows that $\tilde{A}=\tilde{A}\pi_1\tilde{A}$, proving the Morita equivalence.
\end{proof}

The relationship between the representation theory of $A, A^C$ and $\tilde{A}$ is a version of Clifford theory, studied in \cite{RamRam}.  Note that in the case of a Morita action, the Morita equivalence intertwines the pairs of functors \[\Ind_{A^C}^A\colon A^C\mmod \to A\mmod \qquad \Res_{A}^{\tilde{A}}\colon \tilde{A}\mmod \to A\mmod\]
\[\Ind_{A}^{\tilde{A}}\colon \tilde{A}\mmod \to A\mmod \qquad \Res_{A^C}^{A} \colon A\mmod \to A^C\mmod\]
by the commutative diagram
  \[\tikz[->,thick]{
\matrix[row sep=15mm,column sep=30mm,ampersand replacement=\&]{
\&\node (a) {$A^C\mmod $}; \&\\
\node (b)
{$A\mmod $}; \& \& 
\node (c)
{$A\mmod $};\\
 \& \node (d) {$\tilde{A}\mmod$}; \&  \\
};
\draw (a) -- (b) node [midway,above left]{$\Ind_{A^C}^A$}; 
\draw (c) -- (a) node[above right,midway]{$\Res_{A^C}^{A}$} ; 
\draw (d) -- (b) node[ below left,midway]{$\Res_{A}^{\tilde{A}}$}; 
\draw (c) -- (d) node[below right,midway]{$\Ind_{A}^{\tilde{A}}$}; 
\draw[<->] (a) --(d) node[left, midway]{$\tilde{A}\pi_1\otimes_{A^C}-$} node[right, midway]{$(-)^C$};
}\]
Fix a Serre subcategory $\mathcal{C}$ of $A\mmod$ closed under twisting by $\auto$, and $\mathcal{C}^C$ and $\tilde{\mathcal{C}}$ the equivalent subcategories of $A^C\mmod$ and $\tilde{A}\mmod$ which lie in $\mathcal{C}$ after induction/restriction to $A$.  Let $\simples$ be the set of isomorphism classes of simple modules in this category with $V_x$ for $x\in \simples$ a distinguished representative.  Let $\auto\colon \simples \to \simples$ denote the induced map; let $p_x$ be the smallest integer such that $\auto^{p_x}(x)=x.$ A {\bf scaffolding} of this category is a system of isomorphisms $q_x\colon V_x^{\auto}\to V_{\auto{x}}$ such that $q_{\auto^{p-1}(x)}\cdots q_{\auto(x)}q_x=1$.  Note that if we only think of the underlying vector spaces, $q_x$ is an automorphism, but $a\cdot q_x(v)=q_x(a^{\auto}v)$.  

Of course, $Q_x=q_{\auto^{p_x-1}(x)}\cdots q_{\auto(x)}q_x\colon V_x^{\auto^{p_x}}\to V_x$ induces an isomorphism of $A$-modules, which we are at pains to emphasize is not the identity, though the scaffolding condition assures that $Q_x^{p/p_x}=1_{V_x}$. 

Of course, we can think of this as an action of the stabilizer $C_x$, and together with $A$, it forms an action of the smash product $A\# C_x$; for a character $\nu\colon C_x\to \mu_{p/p_x}(\K)$, we can twist this action by the character to form $p/p_x$ different module structures $V_{x}^{(\nu)}$, where $\alpha^{p_x}$ acts by $\nu^{-1}(\alpha^{p_x})Q_x$.  We can thus consider the induced module $\tilde{A}\otimes_{A\otimes C_x}V_{x}^{(\nu)}$; we can view these as the summands of $\Ind^{\tilde{A}}_{A}V_x$ obtained by diagonalizing the action of $\auto^{p_x}\otimes Q_x\colon \Ind^{\tilde{A}}_{A}V_x\to \Ind^{\tilde{A}}_{A}V_x$, which induces an action of $C_x$ on this module, and in fact an isomorphism $\End(\Ind^{\tilde{A}}_{A}V_x)\cong \C[C_x].$

Furthermore, since $Q_x$
 intertwines the original module structure and its twist by a power of $\alpha$, it commutes with the action of $A^C$.  Thus, each isotypic component of $V_x$ under the action of $Q_x$ is an $A^C$ module; these modules only depend on the orbit of $x\in \simples$ since $q_x$ gives an isomorphism between them.  In particular, $Q_{\alpha(x)}=q_xQ_xq_{x}^{-1}$.  
\begin{theorem}\label{thm:clifford}
Assume that $A$ and $C$ define a Morita action, and that $\mathcal{C}$ is a Serre subcategory  of $A\mmod$ with scaffolding $q_x$.  For each  $x\in \simples$ and each character $\nu\colon C_x\to \mu_{p/p_x}(\K)$, the corresponding isotypic component $V_{x,\nu}$ of $V_x$ under the action of $Q_x$ is a simple $A^C$-module (in particular, non-zero), and every simple $\mathcal{C}^C$ is of this form for a unique $C$-orbit $[x]\subset \simples$ and character $\nu$.  Equivalently, the simple $\tilde{A}$ modules in $\tilde{\mathcal{C}}$ are of the form $\tilde{A}\otimes_{A\# C_x}V_{x}^{(\nu)}$.
\end{theorem}
\begin{proof}
This is effectively a restatement of \cite[Th. A.6  \& 13]{RamRam}; the only place where we strengthen their results is that since taking $C$-invariants induces a Morita equivalence $\tilde{A}\mmod\cong A^C\mmod$, we must have $V_{x,\nu}=(\tilde{A}\otimes_{A\# C_x}V_{x}^{(\nu)})^C$ is non-zero.  
\end{proof}

We note an important consequence of this fact (which does not depend on the action being Morita):
\begin{lemma}\label{lem:semi-simple}
If $A$ is finite dimensional and semi-simple, then $\tilde{A}$ is as well.
\end{lemma}
\begin{proof}
Note that we have that $\Ind_{A}^{\tilde{A}}(V_x)\cong \oplus_{\nu} \tilde{A}\otimes_{A\# C_x}V_{x}^{(\nu)}$ is semi-simple.  Furthermore, if $V_x$ is projective as an $A$-module, $\Ind_{A}^{\tilde{A}}(V_x)$ must be projective as a $\tilde{A}$-module.  This shows that every simple $\tilde{A}$-module is projective, so $\tilde{A}$ is semi-simple.
\end{proof}
\excise{
\subsection{Cellular algebras with a cyclic group action}
Recall the definition of a graded cellular algebra, as given in \cite[Def. 2.1]{HM}:
\begin{definition}\label{def:cellular}
  A {\bf graded cellular $\K$-algebra} is an associative graded unital $\K$-algebra
  $A$, free of finite rank, 
  together with a {\bf cell datum} $(\cP, M, \cellb, \deg)$ such
  that \begin{itemize} 
  \item[(0)] $\deg\colon M(\xi)\to \Z$ is a map such that the element $\cellb_{\sS,\sT}$ is homogeneous of degree $\deg \sT+\deg\sT$;
  \item[(1)] $\cP$ is a partially ordered set and
    $M(\xi)$ is a finite set for each $\xi \in \cP$; \item[(2)]
    $\cellb:\dot\bigsqcup_{\xi \in \cP} M(\xi) \times M(\xi) \rightarrow A,
    (\sT,\sS) \mapsto \cellb^\xi_{\sT,\sS}$ is an injective map whose image is a
    basis for $A$;
    \item[(4)] if
    $\xi \in \cP$ and $\sS, \sT \in M(\xi)$ then for any $x \in
    A$ we have that $$ x \cellb_{\sS,\sT}^\xi \equiv \sum_{\sS'
      \in M(\xi)} r_x(\sS',\sS) \cellb_{\sS',\sT}^\xi \pmod{A(>
      \xi)}$$ where the scalar $r_x(\sS',\sS)$ is independent of
    $\sT$ and $A(> \mu)$ denotes the subspace of $A$ generated by
    $\{\cellb_{\sS'',\sT''}^q\mid q > \xi, \sS'',\sT'' \in
    M(q)\}$;
  \item[(3)] the map $*:A \rightarrow A$ sending $(\cellb^\xi_{\sT,\sS})^* = \cellb_{\sS,\sT}^\xi$ is an algebra
    anti-automorphism . \end{itemize} The $\K$-basis consisting of the vectors $\cellb^\xi_{\sS,\sT}$
  is then a {\bf cellular basis} of $A $.
\end{definition}
Recall that if $A$ is an algebra with cellular basis, there is a
natural {\bf cell representation} $\Delta_\xi$ of $A$ for each $\xi\in \cP$ which is
freely generated over $\K$ by symbols $c_{\sT}$ for each $\sT\in
M(\xi)$, with the action rule $xc_{\sT}=\sum_{\sS\in M(\xi)}r_x(\sS,\sT)c_{\sS}$.  This is graded with the degree $\deg(c_{\sT})=\deg \sT$.  Every cell module is equipped with a unique  $*$-Hermitian bilinear form $\langle -,-\rangle$ such that \[\cellb_{\sS',\sT}^\xi\cellb_{\sS,\sT'}^\xi=\langle c_{\sT},c_{\sS}\rangle_{\xi} \cellb_{\sS',\sT'}^\xi \pmod{A(>
      \xi)},\] or equivalently
      \[\cellb_{\sS',\sT}^\xi c_{\sS}=\langle c_{\sT},c_{\sS}\rangle_{\xi}c_{\sS'} \]
      Let $\cP^+=\{\xi\in \cP| \langle -,- \rangle_{\xi}\neq 0\},$ and for $\xi \in\cP^+$, let $L(\xi)$ be the quotient of $\Delta_{\xi}$ by the radical of $\langle -, - \rangle$.  This is a naturally a graded module.
\begin{theorem}[\mbox{\cite[Th. 3.4]{GLcell}, \cite[Prop. 2.18]{HM}}]\hfill
\begin{enumerate}
    \item The quotients $L(\xi)$ for $\xi \in\cP^+$ are a complete irredundant list of the simple $A$-modules, up to grading shift.
    \item With the induced grading, $L(\xi)$ is isomorphic to its contragredient $L(\xi)^\circledast$. 
\end{enumerate}  
\end{theorem}
  


Now assume that $\auto$ is a order $p$ automorphism of $A$ which is compatible with the cellular structure in the sense that $\cP$ and $\dot\bigsqcup_{\xi \in \cP} M(\xi)$ carry actions of $\auto$ such that $\auto(\cellb^\xi_{\sS,\sT})=\cellb^{\auto \xi}_{\auto\sS,\auto\sT}$.  

There is a unique coarsest partial order such that the basis remains cellular, and we note that $\xi$ and $\auto^k\xi$ are never comparable in this order unless $\xi=\auto^k\xi$.  

\begin{lemma}
The map $q_{\xi}\colon \Delta_{\xi}^\auto\to \Delta_{\auto\xi}$ sending $c_{\sS}\mapsto c_{\auto\sS}$ is an isomorphism of $A$-modules which satisfies  $\langle c_{\sT},c_{\sS}\rangle=\langle c_{\auto\sT},c_{\auto\sS}\rangle$.  This induces an isomorphism between the quotient by the radical of this form,  giving a scaffolding for the set of simple $A$-modules.
\end{lemma}
As usual, we let $p_{\xi}$ be the smallest positive integer such that $\auto^{p_{\xi}}\xi=\xi$.  
\begin{lemma} The simple $A^{\auto}$-modules are given by the eigenspaces on $L(\xi)$ for $\xi\in \cP^\circ$ of the automorphism induced by
\[Q_\xi c_{\sT}=c_{\auto^{p_{\xi}}\sT}.\] In particular, the action of $\auto$ is Morita if and only if the action of $\auto$ on $\cup_{k=1}^{p} M(\auto^k\xi)$ is faithful for all $\xi$.
\end{lemma}}

\subsection{Koszul algebras with a cyclic group action}
Now, specialize to the case where $A$ is a positively graded algebra with $A_0$ semi-simple; we call such an algebra {\bf mixed}.  As usual, we call a complex of graded projectives $\cdots \to P_k\to P_{k-1}\to \cdots \to P_0$ over $A$ {\bf linear} if $P_k$ is generated in degree $k$.  Note that linear projective complexes form an graded abelian category (with grading shift given by simultaneous grading and homological shift to maintain linearity).
\begin{definition}
We call $A$ Koszul if $A_0$ possesses a linear projective resolution, and the Koszul dual $A^!$ of $A$ is the endomorphism algebra of this resolution, that is $\Ext_A(A_0,A_0)$.
\end{definition}
Assume that we have a graded automorphism $\alpha$ of $A$, and consider the graded algebras $A^C$ and $\tilde{A}$ (where we have $C$ in degree 0).  Note that $\tilde{A}_0$ is semi-simple by Lemma \ref{lem:semi-simple}, so  $\tilde{A}$ satisfies the same conditions.  
\begin{theorem}\label{thm:smash-Koszul}
If $A$ is Koszul, then $\tilde{A}$ is Koszul as well, with $\tilde{A}^!\cong A^!\# C$.
\end{theorem}
\begin{proof}
The functor $\Ind_{A}^{\tilde{A}}$ is exact and sends projectives to projectives.  Thus, it induces a functor on linear complexes of projectives that sends a resolution of  $A_0$ to one of $\tilde{A}_0$.  This shows that $\tilde{A}_0$ has a linear resolution, as desired.

Furthermore, $C$ acts on $A_0$ as a left module by right multiplication.  This gives an action of $A^!\# C$ on the linear resolution.  Furthermore, the fact that \[\Hom_{\tilde{A}}(\Ind_{A}^{\tilde{A}}(P^\bullet),\Ind_{A}^{\tilde{A}}(P^\bullet))\cong \Hom_{A}(P^\bullet,\Res^{\tilde{A}}_{A}\Ind_{A}^{\tilde{A}}(P^\bullet))\cong \Hom_{A}(P^\bullet,(P^\bullet)^{\oplus p})\] shows that this is precisely the endomorphism ring.
\end{proof}

We wish to strengthen this result a bit to the case of a quasi-hereditary algebra.  Assume that $A$ is quasi-hereditary with a partial order $\leq$ on $\simples$ which is invariant under the action of $\auto$.  
Recall that $A$ is called {\bf standard Koszul} if every standard module $\Delta_x$ has a linear projective resolution.
\begin{theorem}\label{thm:standard-Koszul}
The algebras $A^C$ and $\tilde{A}$ are quasi-hereditary with the induced partial order on $\tilde{\simples}$.  If $A$ is standard Koszul, so are $A^C$ and $\tilde{A}$.
\end{theorem}
\begin{proof}
Given any standard module, we clearly have $\Delta_x^\auto\cong \Delta_{\auto\cdot x}$; since $\Delta_x$ has only scalar endomorphisms, there is a unique system of these isomorphisms compatible with the scaffolding on simples.  We will prove that the isotypic components $\Delta_{x,\nu}$ for the action of $C_x$ on $\Delta_x$ are the standard modules for the algebra $A^C$ (and thus their images under Morita equivalence for $\tilde{A}$).  To show this, we need only confirm that the algebra $A^C$ is filtered by these standard modules.  By the quasi-hereditary assumption, the algebra $A$ is filtered with subquotients being sums of standard modules; since no two elements of a single order are comparable, we can choose this filtration to be $\auto$-equivariant.   As a $C$-module, each subquotient is a sum of terms given by a tensor product of  $M_x\otimes \Delta_x$ where a $M_x$ is a  $C_x$-representation. The intersection of this filtration with $A^C$ gives the invariants $(M_x\otimes \Delta_x)^{C_x}$, which is given by the isotypic components of $M_x$ tensored with the dual isotypic component of $\Delta_X$.  Thus, $A^C$ is quasi-hereditary.

Now, assume that $A$ is standard Koszul.  By \cite[Thm. 3]{ADL}, $A^!$ is also standard Koszul.  By the first part of this result, $\tilde{A}$ and $\tilde{A}^!$ are both quasi-hereditary and by Theorem \ref{thm:smash-Koszul}, they are Koszul dual to each other.  Applying \cite[Thm. 3]{ADL} again, the fact that $\tilde{A}^!$ is quasi-hereditary with the reverse order to $\tilde{A}$ implies that $\tilde{A}$ is standard Koszul.
\end{proof}
\excise{
Now, let us consider the effect of adding this cyclic group action to Ext algebras.
 Let $S'_1,S'_2$ be two simple $\tilde{A}$-modules.
Let $S_1$ and $S_2$ be simple $A$-modules such that $S_i'=\Ind^{\tilde{A}}_{A\# C_{i}}S_i$, with scaffolding that matches the action of $C_i$ on $S_i$.   That is, $(S'_i)^C$ is the invariants of $Q_i$ acting on $S_i$.

Let $p_1=p_{S_1},p_2=p_{S_2}$ and let $p'=\operatorname{lcm}(p_1,p_2)$ and $C_{p'}=\langle \auto^{p'}\rangle$. Note that this is the number of distinct pairs of modules of the form $(S_1^{\auto^k},S_2^{\auto^k})$ up to isomorphism.  Then we have order $p/p'$ morphisms of $S_1^{\auto^{p'}}\to S_1$ and $S_2^{\auto^{p'}}\to S_2$ respectively given by $Q_{S_1}^{p'/p_1}$ and $Q_{S_2}^{p'/p_2}$.  These induces an action of $C_{p'}$ on $\Ext^n(S_1,S_2)$ via the identification with $\Ext^n(S_1^{\auto^{p'}},S_2^{\auto^{p'}})$ and then $f\mapsto Q_{S_2}^{p'/p_2}fQ_{S_1}^{-p'/p_1}$.  

\begin{theorem}
\[\Ext^n_{\tilde{A}}(S'_1,S'_2)\cong \bigoplus_{k=0}^{\operatorname{gcm}(p_1,p_2)-1}\Ext^n_{A}(S_1,S_2^{\auto^k})^{C_{p'}}\]
\end{theorem}
\begin{proof}
First, note that Shapiro's Lemma shows that \[\Ext^n_{\tilde{A}}(\Ind^{\tilde{A}}_{A}S_1,\Ind^{\tilde{A}}_{A}S_2)\cong  \Ext^n_{A}(S_1,L_2)\]
where $L_2=\oplus_{k=0}^{p-1}S_2^{\auto^k}$.  

The left-hand side has an action of $C_1=C_{S_1}$ and $C_2=S_{S_2}$ induced by their action on the inductions.  If we let $\sigma \colon L_2^{\alpha}\to L_2$ be the shift map $(s_1,\dots, s_p) \mapsto (s_2,\dots, s_p,s_1)$, then the  action of $(\auto^{kp_1},\auto^{mp_2})$ is given by 
\[(\auto^{kp_1},\auto^{mp_2})\cdot \eta=Q_{2}^{-m}\sigma^{mp_2-kp_1} \eta Q_{1}^{k} \]
We're interested in the invariants of this action.  The claimed answer embeds in the obvious way.
\end{proof}
}

\section{Cherednik algebras}\label{sec:cherednik}
\subsection{The standard presentation}\label{ss:standard}
There is an infinite family of complex reflection groups $G(\ell,p,n)$ labeled by positive integers $\ell,p$, and $n$ with the requirement that $p$ divides $\ell$. The group $G(\ell,1,n)$ consists of $n \times n$ matrices with $\ell$th roots of unity as the nonzero entries where there is exactly one nonzero entry in each row and column. 
\begin{definition}
    The group $G(\ell,p,n)$ is the subgroup of $G(\ell,1,n)$ containing matrices where the product of the nonzero entries is an $\ell/p$th root of unity. 
\end{definition}

Let $\K$ be a field of characteristic coprime to $\ell$ and $n$ and let $\zetal$ be $\ell$th root of unity.  The most common choice of interest to us is $\K=\C$ and $\zetal = e^{2\pi i/ \ell}$.  Consider the diagonal matrix $t_i = \diag(1,\dots,\zetal,\dots,1)$ with $\zetal$ as the $i$th entry along the diagonal. Let $(i,j)$ be the permutation matrix interchanging $i$ and $j$. The reflections in $G(\ell, p, n)$ are
\begin{align*}
r_{ij}^{(k)}&=t_i^k t_j^{-k} (i,j)    &&\text{for $1 \leq i \leq j \leq n$ and $0 \leq k \leq \ell-1$}\\
t_i^{kp}& &&\text{for $1 \leq i \leq n$ and $0 \leq k \leq \ell/p - 1$.}
\end{align*}
In the case $p=1$, the group $G(\ell, 1, n)$ is generated by $(i,i+1)$ and $t_i$. For general $p$, the group $G(\ell,p,n)$ is generated by the elements $t_i^p, t_i t_{i+1}^{-1}$, and $(i,i+1)$.

To define the Cherednik algebra, we first give a standard definition, following \cite{Gri}. Let $\mathfrak{h}=\K^n$ with the usual action of the matrix group $G(\ell,p,n)$ and $\mathfrak{h}^*$ the dual space, which by the usual inner product, we can think of as $\K^n$ with the transpose inverse action of this matrix group. 
We let $V = \mathfrak{h}^* \oplus \mathfrak{h}$. This is a symplectic representation of $G(\ell,p,n)$, with each reflection acting by a symplectic reflection, that is, a symplectomorphism $g\colon V \to V$ such that $\ker(g-1)$ is a codimension 2 subspace. 
Let $T$ be the set of reflections in $G(\ell,p,n)$. For $s \in T$, define $\alpha_s \in \mathfrak{h}^*$ and $\alpha_s^\vee \in \mathfrak{h}$ by
\begin{align}
	sy = y - \langle y, \alpha_s^\vee \rangle \alpha_s.
\end{align}
Fixing a basis $x_1,\dots,x_n$ for $\mathfrak{h}$ and the dual basis $y_1,\dots,y_n$ for $\mathfrak{h}^*$, we have
\begin{equation}
    \alpha_s = \zetal^{-\ell-k} x_i \qquad 
	\alpha_s^\vee = (\zetal^{k+1}-\zetal) y_i
\end{equation}
for $s = t_i^k$, and
\begin{equation}
    \alpha_s = y_i - \zetal^k y_j \qquad \alpha_s^\vee = x_i - \zetal^{-k} x_j
\end{equation}
for $s = t_i^k t_j^{-k} (i,j)$. These vectors also satisfy
\begin{align}
	s^{-1} x = x - \langle \alpha_s, x \rangle \alpha_s^\vee.
\end{align}
Let $TV = \bigoplus_{k=0}^\infty V^{\otimes k}$ be the tensor algebra over $V$. Let $W \subset GL(V)$ be a finite subgroup. For $f \in TV$, denote its image under $w \in W$ by $f^w$. Let $\Bbbk W$ be the group algebra of $W$ over a field $\Bbbk$. The semidirect product $TV \rtimes W$ is $TV \otimes \Bbbk W$ with multiplication 
\begin{align}
	(f_1 w_1)(f_2 w_2) = f_1 f_2^{w_1} w_1 w_2
\end{align}
for $f_1, f_2 \in TV$ and $w_1, w_2 \in W$. 
\begin{definition}
    The rational Cherednik algebra corresponding to a complex reflection group $W$ is the quotient of $TV \rtimes W$ by the relations
\begin{equation}
    [x, x'] = 0\qquad 
	 [y, y'] = 0 
\end{equation}	 
\begin{equation}
     [x, y] = \hbar \langle y, x \rangle - \sum_{s \in T} c_s \langle \alpha_s, x \rangle \langle y, \alpha_s^\vee \rangle s
\end{equation}
for $x, x' \in \mathfrak{h}$, $y, y' \in \mathfrak{h}^*$, and fixed parameters $\hbar$ and $c_s$ satisfying $c_{wsw^{-1}} = c_s$ for all $w \in W$.

We let $\mathsf{H}_{p}$ denote the Cherednik algebra for $G(\ell,p,n)$, with $\ell,n$ left implicit.
\end{definition}
 The restriction on $c_s$ associates a variable to each conjugacy class of $W$. When $n\geq 3$, the reflections of the form $s=r_{ij}^{(k)}$ are a single conjugacy class, so we give all of them the parameter $c_s = k$. The powers $t_i^{kp}$ for $k=1,\dots,\ell/p-1$ are representatives of all other conjugacy classes, and let $c_k$ be the parameter for $t_i^k$. Note that for $G(\ell,p,n)$, we have $c_k = 0$ unless $p$ divides $k$. The parameters $c_k$ will ultimately correspond to the masses of the hypermultiplet scalars.

In the basis we chose for $\mathfrak{h}$ and $\mathfrak{h}^*$, a reflection $t_i$ acts by
\begin{align}
	t_i x_j &= \zetal^{\delta_{i,j}} x_j t_i \label{eqn:refln1}\\
	t_i y_j &= \zetal^{-\delta_{i,j}} y_j t_i. \label{eqn:refln2}
\end{align}
Expanding the commutation relations in this basis gives
\begin{align}
	&[x_i, y_i] = \hbar + k \sum_{i \neq j} \sum_{k=0}^{\ell-1}  t_i^k t_j^{-k} s_{ij} + \sum_{k=1}^{\ell-1} c_k (1-\zetal^{-k}) t_i^k \\
	&[x_i, y_j] = -k \sum_{k=0}^{\ell-1}  t_i^k t_j^{-k} s_{ij} \qquad \mathrm{for} \ i \neq j. \label{eqn:com2}
\end{align}
Following \cite{Webalt}, define polynomials
\begin{equation}
	p(u) = \sum_{k=1}^{\ell-1} c_k u^k = \sum_{k=1}^{\ell-1} \sum_{r=0}^{\ell-1} \zetal^{-rk} h_r u^k, 
\end{equation}
for the fixed variables $h_0,\dots,h_{\ell-1}$, which are the Fourier transforms of the variables $c_1,\dots,c_{\ell-1}$.  The condition $c_k=0$ for $k$ not divisible by $p$ implies that $p(u)=q(u^p)$ for $q(v)=\sum_{m=1}^{\ell/p-1}c_{mp}v^m$.   Explicitly, we have
\begin{equation}
	h_r = p(\zetal^r)=\sum_{s=1}^{\ell-1} c_s \zetal^{rs}.
\end{equation}
Since $c_k = 0$ if $p$ does not divide $k$, the exchange of translation and multiplication under Fourier tranform implies that 
$h_r = h_{r+\ell/p}$.
Using these polynomials, we can simplify the commutation relation
\begin{align} \label{eqn:com1}
	[x_i, y_i] = \hbar + k \sum_{i \neq j} \sum_{k=0}^{\ell-1}  t_i^k t_j^{-k} s_{ij} + p(t_i) - p(\zetal^{-1} t_i).
\end{align}
We can then regard the Cherednik algebra as generated by the variables $t_i$,$x_i$, and $y_i$ and obeying the relations (\ref{eqn:refln1}), (\ref{eqn:refln2}), (\ref{eqn:com2}), and (\ref{eqn:com1}). 

 We also define the variables $s_m = h_m + m\hbar$, which are only independent for $0 \leq m \leq \ell/p - 1$. They obey $s_m + \ell \hbar/p = s_{m + \ell/p}$.  
\begin{definition}\label{def:p-cyclic}
    If the parameters of our Cherednik algebra satisfy one of the equivalent conditions:
    \begin{enumerate}
        \item $c_k=0$ if $p\nmid k$,
        \item $h_r = h_{r+\ell/p}$ for all $r$,
        \item $s_m + \ell \hbar/p = s_{m + \ell/p}$ for all $m$,
    \end{enumerate}
    then we say our parameters are {\bf $p$-cyclic}.
\end{definition}
As noted in \cite[(4.14)]{Gri}, if parameters of the Cherednik algebra $\mathsf{H}_1$ are $p$-cyclic, then we can realize $\mathsf{H}_{p}$ as a subalgebra in $\mathsf{H}_{1}$, and in fact as the fixed points of an automorphism $\auto$ of order $p$  defined by 
\begin{equation}
    \auto(x_i)=x_i\qquad \auto(y_i)=y_i \qquad \auto(r_{ij}^{k})=r_{ij}^{k}\qquad \auto(t_i)=\zetap t_i
\end{equation}
where $\zetap=\zetal^{\ell/p}.$  Let $C_p$ be the cyclic group of automorphisms generated by $\auto$.
\begin{lemma}
The action of $\auto$ on $\mathsf{H}_{1}$ is Morita. Thus, the algebras $\mathsf{H}_{p}$ and $\tilde{\mathsf{H}}_{1}=\mathsf{H}_{1}\# C_p$ are Morita equivalent.
\end{lemma}
\begin{proof}
    The elements $1,t_1,t_1^2,\dots, t_1^{p-1}$ are all units and represent the $p$ different isotypic components of $C_p$.  This shows that the action is Morita and by Lemma \ref{lem:Morita}, that we have the desired Morita equivalence.
\end{proof}


\subsection{An alternate presentation}\label{ss:alternate}
An alternate presentation of the rational Cherednik algebra is given in \cite{Webalt} for $G(\ell,1,n)$, and a very closely related presentation is given by Griffeth in \cite[\S 3]{GrifU}. This presentation matches better with the physical description of the Coulomb branch. We extend this presentation to $G(\ell,p,n)$.

Recall that $G(\ell,p,n)$ is generated by $t_i^p$, $t_i t_{i+1}^{-1}$, and $(i,i+1) \in S_n$. Let $\tilde{H}$ be the free $\CC[\hbar]$ algebra generated by $\CC[\hbar] G(\ell,p,n)$ and the symbols $\sigma$, $\tau$, and $u_i$ for $i=1,\dots,n$. Define $u_i$ and $t_i$ for all $i \in \ZZ$ by the rules $u_i = u_{i-n} + \hbar$ and $t_i = \zetal^{-1} t_{i-n}$. Let $H$ be the quotient of $\tilde{H}$ by the relations
\begin{align*}
	u_i(j,j+1)&=(j,j+1)u_{(j,j+1)\cdot
		i}+k\ell (\delta_{i,j} - \delta_{i,j+1}) \pi_{j,j+1} &j&=1,\dots, n-1 \\
	u_it_j&=t_ju_i\\
	\sigma (j,j-1)&=(j+1,j)\sigma& j&=2,\dots, n-1 \\
	\tau (j,j+1)&=(j-1,j)\tau& j&=2,\dots, n-1 \\
	      \sigma^2 r_{n}&=r_1\sigma^2\subeqn\\
    \tau^2 r_1&=r_{n}\tau^2\subeqn\\
	\sigma\tau&=u_1-p(\zetal ^{-1} t_1)+\hbar\\
	\tau\sigma&=u_n-p(t_n) \\
	u_iu_j&=u_ju_i & i,j&\in \ZZ \\
	u_i\sigma&=\sigma u_{i-1}& i&\in \ZZ \\
	u_i\tau&=\tau u_{i+1}& i&\in \ZZ \\
	t_i\sigma&=\sigma t_{i-1}& i&\in \ZZ \\
	t_i\tau&=\tau t_{i+1}& i&\in \ZZ \\
	\tau (1,2) \sigma &= \sigma (n-1,n) \tau +k \left(\sum_{p=0}^{\ell-1}\zetal^p t_n^pt_{1}^{-p}\right)
\end{align*}
where $\pi_{j,j+1} := \frac{1}{\ell} \sum_{p=0}^{\ell-1} t_j^pt_{j+1}^{-p}$.
\begin{theorem}
	The algebra $H$ is isomorphic to the rational Cherednik algebra for $G(\ell, p, n)$ with the maps
	\begin{align*}
	x_i \mapsto (i,\dots,1) \sigma (n,\dots,i) \qquad
	y_i \mapsto (i,\dots,n) \tau (1,\dots,i).
	\end{align*}
\end{theorem}

\begin{proof}
	For the case $G(\ell,1,n)$, this theorem is proved in \cite{Webalt}. For the general case $G(\ell,p,n)$, the same proof applies. In particular, $G(\ell, p, n)$ is a subalgebra of $G(\ell,1,n)$.
\end{proof}

As in \cite{Webalt}, we can represent the algebra $H$ graphically using string diagrams with $n$ strings on a cylinder. Many of our examples of string diagrams are taken from \cite{Webalt}. We draw the diagrams with the cylinder cut along a dashed line. Crossing the dashed line passes around the cylinder between the $n$th string position and the first string position. The generators of $H$ are the following:
\begin{equation*}
\tikz{
	\node[label=below:{$t_i^p$}] at (-4.5,0){ 
		\tikz[very thick,xscale=1.2]{
			\draw[dashed,thick] (-.7,-.5)-- (-.7,.5);
			\draw[dashed,thick] (1.7,-.5)-- (1.7,.5);
			\draw (-.5,-.5)-- (-.5,.5);
			\draw (.5,-.5)-- (.5,.5) node [midway,tstar]{};
			\node[scale=.7] at (.7,.2) {$p$};
			\draw (1.5,-.5)-- (1.5,.5);
			\node at (1,0){$\cdots$};
			\node at (0,0){$\cdots$};
		}
	};
	\node[label=below:{$t_i t_{i+1}^{-1}$}] at (0,0){ 
		\tikz[very thick,xscale=1.2]{
			\draw[dashed,thick] (-.7,-.5)-- (-.7,.5);
			\draw[dashed,thick] (1.7,-.5)-- (1.7,.5);
			\draw (-.5,-.5)-- (-.5,.5);
			\draw (.3,-.5)-- (.3,.5) node [midway,tstar]{};
			\draw(.7,-.5) --(.7,.5) node [midway,tstar]{};	
			\node[scale=.7] at (.9,.3) {$-1$};		
			\draw (1.5,-.5)-- (1.5,.5);
			\node at (1.2,0){$\cdots$};
			\node at (-.1,0){$\cdots$};
		}
	};
	\node[label=below:{$u_i$}] at (4.5,0){ 
		\tikz[very thick,xscale=1.2]{
			\draw[dashed,thick] (-.7,-.5)-- (-.7,.5);
			\draw[dashed,thick] (1.7,-.5)-- (1.7,.5);
			\draw (-.5,-.5)-- (-.5,.5);
			\draw (.5,-.5)-- (.5,.5) node [midway,fill=black,circle,inner
			sep=2pt]{};
			\draw (1.5,-.5)-- (1.5,.5);
			\node at (1,0){$\cdots$};
			\node at (0,0){$\cdots$};
		}
	};
}
\end{equation*}

\begin{equation*}
\tikz{
	\node[label=below:{$(i,i+1)$}] at (4.5,0){ 
		\tikz[very thick,xscale=1.2]{
			\draw[dashed,thick] (-.7,-.5)-- (-.7,.5);
			\draw[dashed,thick] (1.7,-.5)-- (1.7,.5);
			\draw (-.5,-.5)-- (-.5,.5);
			\draw (.1,-.5)-- (.9,.5);
			\draw (.9,-.5)-- (.1,.5);
			\draw (1.5,-.5)-- (1.5,.5);
			\node at (1,0){$\cdots$};
			\node at (0,0){$\cdots$};
		}
	};
	\node[label=below:{$\sigma$}] at (-4.5,0){ 
		\tikz[very thick,xscale=1.2]{
			\draw[dashed,thick] (-.7,-.5)-- (-.7,.5);
			\draw[dashed,thick] (1.7,-.5)-- (1.7,.5);
			\draw (-.5,-.5)-- (-.1,.5);
			\draw (.3,-.5)-- (.7,.5);
			\draw (1.1 ,-.5)-- (1.5,.5);
			\node at (.9,0){$\cdots$};
			\draw (-.7,0)-- (-.5,.5);
			\draw (1.5 ,-.5)-- (1.7,0);
			\node at (.1,0){$\cdots$};
		}
	};
	\node[label=below:{$\tau$}] at (0,0){ 
		\tikz[very thick,xscale=1.2]{
			\draw[dashed,thick] (-.7,-.5)-- (-.7,.5);
			\draw[dashed,thick] (1.7,-.5)-- (1.7,.5);
			\draw (-.5,.5)-- (-.1,-.5);
			\draw (.3,.5)-- (.7,-.5);
			\draw (1.1 ,.5)-- (1.5,-.5);
			\node at (.9,0){$\cdots$};
			\draw (-.7,0)-- (-.5,-.5);
			\draw (1.5 ,.5)-- (1.7,0);
			\node at (.1,0){$\cdots$};
		}
	};
}
\end{equation*}
\excise{
In the isomorphism with the chiral ring, the $u_i$, represented by dots, correspond to the complex scalar fields of the vectormultiplet, and any diagram where the strings wrap around the cylinder corresponds to a monopole operator.
}
The string diagrams are multiplied by stacking one on top of the other so that the ends of the strings match, and the resulting diagram is read top to bottom. For example, the relations $t_1 \sigma = \zetal \sigma t_1$ and $u_1 \sigma = \sigma (u_n -\hbar)$ are captured by the diagrams:

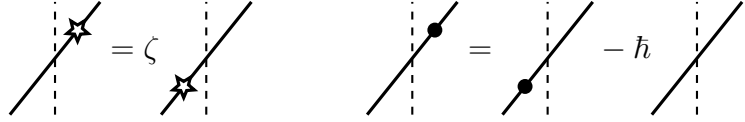
\begin{equation*}
\tikz[very thick,scale=1.5,baseline]{  
	\draw[dashed,thick] (0,-.5)-- (0,.5);
	\draw (-.4,-.5) --  node [pos=.75,tstar]{} (.4,.5);
}= \zetal\tikz[very thick,scale=1.5,baseline]{  
	\draw[dashed,thick] (0,-.5)-- (0,.5);
	\draw (-.4,-.5) --  node [pos=.25,tstar]{} (.4,.5);
} \hspace{1.5cm}
\tikz[very thick,scale=1.5,baseline]{  
	\draw[dashed,thick] (0,-.5)-- (0,.5);
	\draw (-.4,-.5) --  node [pos=.75,ucircle]{} (.4,.5);
}= \tikz[very thick,scale=1.5,baseline]{  
	\draw[dashed,thick] (0,-.5)-- (0,.5);
	\draw (-.4,-.5) --  node [pos=.25,ucircle]{} (.4,.5);
}-\hbar \tikz[very thick,scale=1.5,baseline]{  
	\draw[dashed,thick] (0,-.5)-- (0,.5);
	\draw (-.4,-.5) --  (.4,.5);
}
\end{equation*}

Using the isomorphism to the standard presentation of the Cherednik algebra, we can express $x_i$ and $y_i$ as string diagrams:
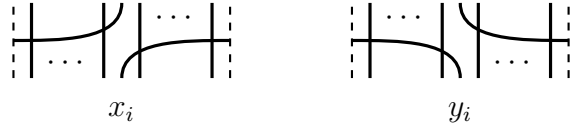
\begin{equation*}
\tikz{
	\node[label=below:{$ x_i $}] at (-4.5,0){ 
		\tikz[very thick,xscale=1.2]{
			\draw[dashed,thick] (-.7,-.5)-- (-.7,.5);
			\draw[dashed,thick] (1.7,-.5)-- (1.7,.5);
			\draw (-.5,-.5)-- (-.5,.5);
			\draw (.3,-.5)-- (.3,.5);
			\draw (.7,-.5)-- (.7,.5);
			\draw (1.5 ,-.5)-- (1.5,.5);
			\node at (1.1,.3){$\cdots$};
			\draw (-.7,0) to[out=0,in=-90] (.5,.5);
			\draw (.5 ,-.5)to[out=90,in=180] (1.7,0);
			\node at (-.1,-.3){$\cdots$};
		}
	};
	\node[label=below:{$ y_i $}] at (0,0){ 
		\tikz[very thick,xscale=1.2]{
			\draw[dashed,thick] (-.7,-.5)-- (-.7,.5);
			\draw[dashed,thick] (1.7,-.5)-- (1.7,.5);
			\draw (-.5,.5)-- (-.5,-.5);
			\draw (.7,.5)-- (.7,-.5);
			\draw (.3,.5)-- (.3,-.5);
			\draw (1.5 ,.5)-- (1.5,-.5);
			\node at (1.1,-.3){$\cdots$};
			\draw (-.7,0)to[out=0,in=90] (.5,-.5);
			\draw (.5 ,.5) to[out=-90,in=180] (1.7,0);
			\node at (-.1,.3){$\cdots$};
		}
	};
}
\end{equation*}

In \cite{Webalt}, a polynomial representation of the Cherednik algebra for $G(\ell,1,n)$ is defined over the ring $\sP = \CC[U_1,\dots,U_n; T_1,\dots,T_n]/\langle T_i^\ell -1 \rangle$ where $U_i$ and $T_i$ are defined for all $i \in \ZZ$ by $U_i = U_{i-n}+\hbar$ and $T_i = \zetal^{-1} T_{i-n}$. We want to define an analogous representation here. However, since $G(\ell,p,n)$ is a subgroup of $G(\ell,1,n)$, it will split $\sP$ into $p$ isomorphic submodules, one for each coset of $G(\ell,p,n)$ in $G(\ell,1,n)$. To only consider one such submodule, we define the polynomial ring as $\sP = \CC[U_1,\dots,U_n; T_1^p,\dots,T_n^p, T_1 T_2^{-1},\dots, T_{n-1} T_n^{-1}]/\langle T_i^\ell - 1\rangle$ for $p \neq 1$.

We have exactly the same action of $H$ on $\sP$ as \cite{Webalt}, but we reproduce it here for convenience. We have
\begin{align}
	&u_i \cdot f(\mathbf{U};\mathbf{T}) =U_if(\mathbf{U};\mathbf{T}) \\
	&t_i\cdot f(\mathbf{U};\mathbf{T}) = T_i f(\mathbf{U};\mathbf{T}) \\
	&(i,i+1)\cdot f(\mathbf{U};\mathbf{T}) = f^{(i,i+1)}+k\ell \frac{
		f^{(i,i+1)}-f}{U_{i+1}-U_i}\pi_{i,i+1}\\
	&\sigma \cdot f(\mathbf{U};\mathbf{T}) = ( U_1-p(\zetal^{-1}T_1)+\hbar ) f(U_2,\dots,U_{n+1};T_2,\dots,T_{n+1})\\
	&\tau \cdot f(\mathbf{U};\mathbf{T}) = f(U_0,\dots,U_{n-1};T_0,\dots,T_{n-1})
\end{align}
where $\pi_{i,i+1} = \frac{1}{\ell}\sum_{k=0}^{\ell-1} T_i^k T_{i+1}^{-k}$ and $f^{(i,i+1)}$ denotes $f(\U;\T)$ with $U_i$ and $U_{i+1}$ exchanged and $T_i$ and $T_{i+1}$ exchanged. More generally, we will use $f^a$ to denote the image of $f(\U;\T)$ under the linear action of $a \in G(\ell,p,n)$, which consists of permutation and shifts of the variables.

\subsection{Spherical and partially spherical subalgebras}\label{ss:subalgebras}
Let $A$ be the subgroup of $G(\ell, p, n)$ generated by elements of the form $t_i^p$ and $t_i t_{i+1}^{-1}$. 
\begin{definition}
	Let $W=G(\ell,p,n)$ and 
	\begin{equation*}
		e = \frac{1}{|W|} \sum_{g \in W} g \qquad
		e' = \frac{1}{|A|} \sum_{g \in A} g
	\end{equation*}
	The {\bf spherical subalgebra} is defined to be $\mathsf{H}_p^{\operatorname{sph}}=e\mathsf{H}_pe$ and the {\bf partially spherical subalgebra} is defined as $\mathsf{H}_p^{\operatorname{psph}}=e'\mathsf{H}_pe'$.
\end{definition}

The idempotents $e$ and $e'$ project to elements of $H$ that are invariant under $W$ and $A$, respectively. Thus, the spherical subalgebra consists of elements invariant under $W$, and the partially spherical subalgebra consists of elements invariant under $A$.

In order to clarify the connection between $\mathsf{H}_p$ and $\mathsf{H}_1$, it is useful to consider the Morita equivalent smash product $\tilde{\mathsf{H}}_1$.
\begin{lemma}
Let $\tilde{W}=G(\ell,1,n)$, and let $\tilde{A}$ be its subgroup generated by $t_i$'s.  Consider the idempotents
	\begin{equation*}
		\tilde{e} = \frac{1}{|\tilde{W}|} \sum_{g \in \tilde{W}} g \qquad\qquad 
		\tilde{e}' = \frac{1}{|\tilde{A}|} \sum_{g \in \tilde{A}} g
	\end{equation*}
	The inclusion $\mathsf{H}_p\hookrightarrow \tilde{\mathsf{H}}_1$ induces isomorphisms:
	\[e\mathsf{H}_pe\cong \tilde{e}\tilde{\mathsf{H}}_1\tilde{e}\qquad\qquad  e'\mathsf{H}_pe'\cong\tilde{e}'\tilde{\mathsf{H}}_1\tilde{e}'.\]
\end{lemma}
Note that $\tilde{\mathsf{H}}_1$ is $\Z/p\Z$ graded by the power of $\auto$ which appears, and this induces a $\Z/p\Z$ grading on $e\mathsf{H}_pe$.  This agrees with the grading induced by the action of $\tilde{A}/A$ on $e\mathsf{H}_pe$ by conjugation.  



\begin{proposition} \label{prop:psc_gen1}
	The partially spherical subalgebra $\mathsf{H}_p^{\operatorname{sph}}$ is generated by the following elements:
	\newseq
	\begin{equation*}\subeqn
		 u_i\qquad  (i,i+1)\qquad x_i^\ell\qquad y_i^\ell\label{eq:ps-gens1}
	\end{equation*}
	\begin{equation*}\subeqn
	    (x_1 \dots x_n)^{\ell/p} \qquad (y_1 \dots y_n)^{\ell/p} \qquad (x_{i_1} \dots x_{i_j} )^{k\ell/p}(y_{i_{j+1}} \dots y_{i_n})^{\ell-k\ell/p}\label{eq:ps-gens2}
	\end{equation*}
	where $i_1,\dots,i_n$ is any permutation of $1,\dots,n$ and $k=1,\dots,p-1$. The multiplication by $e'$ is implicit in these generators.  
\end{proposition}
The elements of (\ref{eq:ps-gens1}) all have degree 0 in the $\Z/p\Z$ grading, whereas the elements of (\ref{eq:ps-gens2}) have degrees 
	\begin{equation*}
	    \deg\big((x_1 \dots x_n)^{\ell/p}\big)=1 \qquad \deg\big((y_1 \dots y_n)^{\ell/p}\big)=-1 
	\end{equation*}  
	\begin{equation*}
	    \deg\big((x_{i_1} \dots x_{i_j} )^{k\ell/p}(y_{i_{j+1}} \dots y_{i_n})^{\ell-k\ell/p}\big)=k
	\end{equation*}
\begin{proof}
	Clearly, $e' u_i$ and $e' (i,i+1)$ are invariant under $A$. To understand the rest of the generators, we consider all possible string diagrams such that moving $t_i^p$ or $t_i t_{i+1}^{-1}$ past the diagram does not pick up extra factors of $\zetal$.
	
	To see why this condition is necessary, consider an element $h \in H$ and the associated string diagram. We just need to establish that $e'he'$ is nonzero. Suppose $h t_i^p = \zetal^k t_i^p h$ for some $k$ such that $\zetal^k \neq 1$. Then, we have
	\begin{align*}
		e' h e' = e' h t_i^p e' = e' \zetal^k t_i^p h e' = \zetal^k e' h e',
	\end{align*}
	implying $e' h e' = 0$. The same thing would happen if moving $t_i t_{i+1}^{-1}$ past $h$ picked up a power of $\zetal$.
	
	By definition, moving $t_i^p$ or $t_i t_{i+1}^{-1}$ past the seam of the cylinder will pick up powers of $\zetal$, but if this power is a multiple of $\ell$, the extra factor will cancel. We find that the allowed diagrams are combinations of diagrams where
	\begin{itemize}
		\item one string wraps around $\ell$ times,
		\item every string wraps around $\ell/p$ times in the same direction, or
		\item some strings wrap around $k\ell/p$ times in one direction and all the other strings wrap around $\ell - k\ell/p$ times in the other direction for some $k = 1,\dots,p-1$.
	\end{itemize}
	For the first case, suppose the $i$th string wraps around $\ell$ times. Moving $t_i$ past the diagram will take it around the cylinder $\ell$ times, so it will pick up a factor of $\zetal^{\pm \ell} = 1$ depending on what direction it goes. For $j\neq i$, moving $t_j$ past the diagram does not take it around the cylinder at all. Therefore, moving $t_j^p$ or $t_j t_{j+1}^{-1}$ (for any $j$) past this diagram will not pick up an overall power of $\zetal$, so this diagram corresponds to a nonzero element of $\mathsf{H}_p^{\operatorname{psph}}$. The corresponding generators that take the $i$th string around $\ell$ times are $x_i^\ell$ and $y_i^\ell$.
	
	For the second case, moving $t_i$ past the diagram picks up a factor of $\zetal^{\pm \ell/p}$ depending on the direction. Therefore, $t_i^p$ picks up a factor of $\zetal^{\pm \ell} = 1$ and $t_i t_{i+1}^{-1}$ picks up a factor of $\zetal^{\pm \ell/p} \zetal^{\mp \ell/p} = 1$. The corresponding generators are $(x_1 \dots x_n)^{\ell/p}$ and $(y_1 \dots y_n)^{\ell/p}$.
	
	For the last case, some of the $t_i$ will pick up a factor of $\zetal^{\pm k\ell/p}$ and the others will pick up a factor of $\zetal^{\mp (\ell - k\ell/p)} = \zetal^{\pm k\ell/p}$. Just as in the previous case, these diagrams correspond to nonzero generators. These generators are $(x_{i_1} \dots x_{i_j} )^{k\ell/p}(y_{i_{j+1}} \dots y_{i_n})^{\ell-k\ell/p}$.
\end{proof}

This set of generators turns out to not be the best choice for our purposes. We instead consider a slightly different, but nonetheless equivalent, set of generators. Because we have the $S_n$ generators, we only need $x_1^\ell$ and $y_n^\ell$ in place of $x_i^\ell$ and $y_i^\ell$. Similarly, we only need $(x_i x_{i-1} \dots x_1 )^{k\ell/p} (y_{i+1} y_{i+2} \dots y_n)^{\ell-k\ell/p}$ for $1<i<n$ in place of the last generator.

In order to avoid extra terms coming the $S_n$ action in what follows, we remove extra permutations from our generators. Considered as string diagrams, this is the statement that we want to remove extra permutations from the top and bottom of the diagrams so that our generators have the minimum numbers of string crossings while still passing around the cylinder the appropriate number of times. This requirement gives us the modified set of generators. The following diagrams show some comparisons of the two sets of generators:

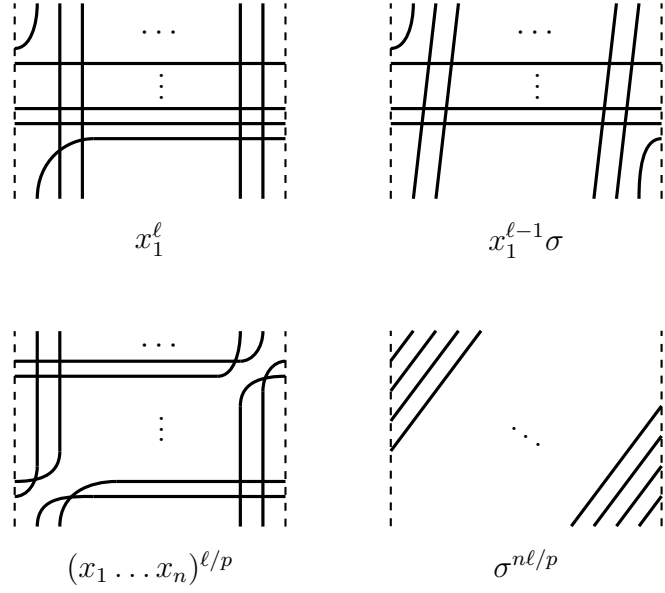
\begin{equation*}
\tikz{
	\node[label=below:{$ x_1^{\ell} $}] at (-5,0){ 
		\tikz[very thick,xscale=1.5,yscale=2]{
			\draw[dashed,thick] (-.7,-.8)-- (-.7,.5);
			\draw[dashed,thick] (1.7,-.8)-- (1.7,.5);
			\draw (-.3,-.8)-- (-.3,.5);
			\draw (-.1,-.8)-- (-.1,.5);
			\draw (1.3,-.8)-- (1.3,.5);
			\draw (1.5 ,-.8)-- (1.5,.5);
			\node at (.6,.3){$\cdots$};
			\node at (.6,0){$\vdots$};
			\draw (-.7,.1)-- (1.7,.1);
			\draw (-.7,-.3)-- (1.7,-.3);
			\draw (-.7,-.2)-- (1.7,-.2);
			\draw (-.7,.2) to[out=0,in=-90] (-.5,.5);
			\draw (-.5,-.8)to[out=90,in=180] (0,-.4);
			\draw (0,-.4)-- (1.7,-.4);
		}
	};
	\node[label=below:{$ x_1^{\ell-1} \sigma $}] at (0,0){ 
		\tikz[very thick,xscale=1.5,yscale=2]{
			\draw[dashed,thick] (-.7,-.8)-- (-.7,.5);
			\draw[dashed,thick] (1.7,-.8)-- (1.7,.5);
			\draw (-.5,-.8)-- (-.3,.5);
			\draw (-.3,-.8)-- (-.1,.5);
			\draw (1.1,-.8)-- (1.3,.5);
			\draw (1.3 ,-.8)-- (1.5,.5);
			\node at (.6,.3){$\cdots$};
			\node at (.6,0){$\vdots$};
			\draw (-.7,.1)-- (1.7,.1);
			\draw (-.7,-.3)-- (1.7,-.3);
			\draw (-.7,-.2)-- (1.7,-.2);
			\draw (-.7,.2) to[out=0,in=-90] (-.5,.5);
			\draw (1.5,-.8)to[out=90,in=180] (1.7,-.4);
		}
	};
}
\end{equation*}

\begin{equation*}
\tikz{
	\node[label=below:{$ (x_1 \dots x_n)^{\ell/p} $}] at (-5,0){ 
		\tikz[very thick,xscale=1.5,yscale=2]{
			\draw[dashed,thick] (-.7,-.8)-- (-.7,.5);
			\draw[dashed,thick] (1.7,-.8)-- (1.7,.5);
			\node at (.6,.4){$\cdots$};
			\node at (.6,-.1){$\vdots$};
			\draw (-.7,-.6) to[out=0,in=-90] (-.5,-.4);
			\draw (-.5,-.4)-- (-.5,.5);
			\draw (-.5,-.8)to[out=90,in=180] (0,-.6);
			\draw (0,-.6)-- (1.7,-.6);
			\draw (-.7,-.5) to[out=0,in=-90] (-.3,-.3);
			\draw (-.3,-.3)-- (-.3,.5);
			\draw (-.3,-.8)to[out=90,in=180] (.2,-.5);
			\draw (.2,-.5)-- (1.7,-.5);
			\draw (1.1,.2) to[out=0,in=-90] (1.3,.5);
			\draw (-.7,.2)-- (1.1,.2);
			\draw (1.3,0)to[out=90,in=180] (1.7,.2);
			\draw (1.3,-.8)-- (1.3,0);
			\draw (1.3,.3) to[out=0,in=-90] (1.5,.5);
			\draw (-.7,.3)-- (1.3,.3);
			\draw (1.5,.1)to[out=90,in=180] (1.7,.3);
			\draw (1.5,-.8)-- (1.5,.1);
		}
	};
	\node[label=below:{$ \sigma^{n\ell/p} $}] at (0,0){ 
		\tikz[very thick,xscale=1.5,yscale=2]{
			\draw[dashed,thick] (-.7,-.8)-- (-.7,.5);
			\draw[dashed,thick] (1.7,-.8)-- (1.7,.5);
			\draw (-.7,.1)-- (-.3,.5);
			\draw (-.7,-.1)-- (-.1,.5);
			\draw (1.1,-.8)-- (1.7,-.2);
			\draw (1.3 ,-.8)-- (1.7,-.4);
			\node at (.5,-.15){$\ddots$};
			\draw (-.7,.3)-- (-.5,.5);
			\draw (1.5,-.8)-- (1.7,-.6);
			\draw (.9,-.8)-- (1.7,0);
			\draw (-.7,-.3)-- (.1,.5);
		}
	};
}
\end{equation*}
The original generators are on the left, and the modified generators with extra permutations removed are on the right. Note that in the above diagrams, we have enlarged the cylinder by $\ell/p$ so that passing once around the enlarged cylinder corresponds to passing $\ell/p$ times around the original cylinder.

\begin{proposition} \label{prop:psc_gen2}
	The partially spherical subalgebra is generated by the following elements of degree $0\in \Z/p\Z$ 
	\begin{equation*}
	    u_i\qquad r_i\qquad x_1^{\ell-1} \sigma\qquad y_n^{\ell-1} \tau
	\end{equation*}
	and the elements
\begin{equation*}
   \sigma^{n\ell/p} \qquad \tau^{n\ell/p} \qquad\left( x_1^{k\ell/p-1} \sigma \right)^i \left( y_{n-i}^{\ell-k\ell/p-1} (n-i,\dots,n) \tau \right)^{n-i}
\end{equation*} 
which have degree $1,-1$ and $i$ respectively.  
\end{proposition}

\begin{proof}
	These generators are equivalent to the ones given in Proposition \ref{prop:psc_gen1}, up to permutations. For example, we have 
	\begin{align*}
		x_1^{\ell-1} \sigma = x_1^{\ell-1} \sigma (n,\dots,1) (1,\dots,n) = x_1^\ell (1,\dots,n).
	\end{align*}
	The remainder of the calculations are the same in spirit but considerably longer. They are given in Appendix A of \cite{LePageThesis}.
\end{proof}

To gain a bit more insight into the last generator, it is useful to note that
\begin{align*}
	&\left( x_1^{k\ell/p-1} \sigma \right)^i \left( y_{n-i}^{\ell-k\ell/p-1} (n-i,\dots,n) \tau \right)^{n-i} \\
	&= \left( x_1^{k\ell/p-1} \sigma \right)^i \left( (1,\dots,i) (2,\dots,i+1) \dots (n-i,\dots,n) \right) \left( y_{n}^{\ell-k\ell/p-1} \tau \right)^{n-i}.
\end{align*}
Essentially, we want to use $x_1^{k\ell/p-1} \sigma$ or $ y_{n}^{\ell-k\ell/p-1} \tau$ to take each strand around, and the permutations in the middle ensure that the strands are matched correctly so that each one wraps around one direction or the other.

Now that we have established the generators of $\mathsf{H}_p^{\operatorname{psph}}$, we need their action on the polynomial representation. Since $\mathsf{H}_p^{\operatorname{psph}} \cdot \sP \subseteq e'\sP$, we only care about the action on $e'\sP$.

\begin{proposition}\label{prop:action}
	The partially spherical subalgebra $\mathsf{H}_p^{\operatorname{psph}}$ acts on the polynomial representation $e'\sP$ in the following way:
	\begin{align}
		& x_1^{\ell-1} \sigma \cdot f(\U) = \prod_{m=0}^{\ell-1} (U_1+\ell\hbar-s_m) f(U_2,\dots,U_n,U_1+\ell\hbar) \label{action1} \\
		& y_n^{\ell-1} \tau \cdot f(\U) = f(U_n - \ell\hbar, U_1, \dots, U_{n-1}) \label{action2} \\
		& \sigma^{n\ell/p} \cdot f(\U) = \prod_{i=1}^n \prod_{m=0}^{\ell/p-1} (U_i + \ell\hbar/p - s_m) f(U_1 + \ell\hbar/p, \dots, U_n + \ell\hbar/p) \label{action3} \\
		& \tau^{n\ell/p} \cdot f(\U) = f(U_1 - \ell\hbar/p,\dots, U_n - \ell\hbar/p) \label{action4} \\
		& \left( x_1^{k\ell/p-1} \sigma \right)^i \left( y_{n-i}^{\ell-k\ell/p-1} (n-i,\dots,n) \tau \right)^{n-i} \cdot f(\U) \nonumber \\
		& \hspace{2cm} = \prod_{j=1}^i \prod_{m=0}^{k\ell/p-1} (U_j + k\ell\hbar/p - s_m) f'(\U) \label{action5}.
	\end{align}
where $$f'(\U) =  f(U_{i+1} + k\ell\hbar/p - \ell\hbar, \dots, U_{n} + k\ell\hbar/p - \ell\hbar, U_{1} + k\ell\hbar/p, \dots, U_i + k\ell\hbar/p).$$
\end{proposition}


In order to verify this proposition, we have to introduce additional projectors. We define the projector
\begin{equation}
	e_{\eta} = \frac{1}{|A|}\sum_{g \in A} \eta(g^{-1}) g
\end{equation}
where $A$ is the same subgroup of $G(\ell,p,n)$ as before and $\eta$ is a character of $A$. The characters that we need are $\eta_i$, which satisfy $\eta_i(t_j) = \zetal^{\delta_{ij}}$, and products of $\eta_i$, possibly for multiple values of $i$. Let $E_{\eta_i} = e_{\eta_i} \cdot 1 \in \sP$, which is essentially $e_{\eta_i}$ with every $t_j$ replaced by $T_j$. Similarly, let $E' = e' \cdot 1$.

In the following calculations, we have to leave $e' \cdot \sP$, and these additional projectors keep track of which component of $\sP$ we are projecting to, allowing us to easily see which parts of expressions are projected out and which parts survive. 

\begin{proof}
	All the relations of Proposition \ref{prop:action} are verified by direct calculation. We show the calculations to verify (\ref{action1}), (\ref{action3}), and (\ref{action5}), as the other two calculations are extremely similar.
	
	We need a few preliminary results first. For $k \neq 0$, we have
	\begin{align*}
		x_i & \cdot f(\U;\T) E_{\eta_i^k} \\
		&= (i,\dots,1) \sigma (n,\dots,i) \cdot f(\U;\T) E_{\eta_i^k} \\
		&= (i,\dots,1) \sigma \cdot \left( f(\U;\T)^{(n,\dots,i)} + (1-E_{\eta_n^k}) a(f) \right) E_{\eta_n^k} \\
		&= (i,\dots,1) \sigma \cdot f(\U;\T)^{(n,\dots,i)} E_{\eta_n^k} \\
		&= (i,\dots,1) \cdot (U_1 - p(\zetal^{-1} T_1) + \hbar) f(\U;\T)^{\sigma(n,\dots,i)} E_{\eta_1^{k+1}} \\
		&= (U_i - p(\zetal^{-1} T_i) + \hbar) \left( f(\U;\T)^{(i,\dots,1)\sigma(n,\dots,i)} + (1-E_{\eta_1^{k+1}}) a'(f) \right) E_{\eta_i^{k+1}} \\
		&= (U_i - p(\zetal^{-1} T_i) + \hbar) f(U_1, \dots, U_i + \hbar, \dots, U_n;T_1, \dots, \zetal^{-1} T_i, \dots, T_n) E_{\eta_i^{k+1}}
	\end{align*}
	where $a(f)$ and $a'(f)$ are correction terms that are projected out. This result still holds if $E_{\eta_i^k}$ is replaced by an arbitrary projector $E_\eta$ provided $E_\eta$ contains a nontrivial power of $\eta_i$. This requirement ensures that the correction term vanishes. In other words, we have
	\begin{align*}
		x_i &\cdot f(\U;\T) E_{\eta} \\
		&= (U_i - p(\zetal^{-1} T_i) + \hbar) f(U_1, \dots, U_i + \hbar, \dots, U_n;T_1, \dots, \zetal^{-1} T_i, \dots, T_n) E_{\eta_i \eta}
	\end{align*}
	for $\eta(t_i) \neq 1$. Similarly, we find
	\begin{align*}
		y_i \cdot f(\U;\T) E_{\eta} &= f(U_1, \dots, U_i - \hbar, \dots, U_n;T_1, \dots, \zetal^{1} T_i, \dots, T_n) E_{\eta_i^{-1} \eta}
	\end{align*}
	for $\eta(t_i) \neq 1$.
	
	Using the first result, we have
	\begin{align*}
		e' x_1^{\ell-1} \sigma \cdot f(\U;\T) E'
		&= e' x_1^{\ell-1} \cdot (U_1 - p(\zetal^{-1} T_1) + \hbar)) \\
		& \hspace{2cm} f(U_2,\dots, U_n, U_1 + \hbar; T_2, \dots, T_n, \zetal^{-1} T_1) E_{\eta_1} \\
		&= e' x_1^{\ell-2} \cdot (U_1 - p(\zetal^{-1} T_1) + \hbar) (U_1 - p(\zetal^{-2} T_1) + 2\hbar) \\
		& \hspace{2cm} f(U_2,\dots,U_n,U_1+2\hbar; T_2, \dots, T_n, \zetal^{-2} T_1) E_{\eta_1^2} \\
		&= e' \cdot \prod_{m=1}^{\ell} (U_1 - p(\zetal^{-m} T_1) + m\hbar) \\
		& \hspace{2cm} f(U_2,\dots,U_n,U_1+\ell\hbar; T_2, \dots, T_n, \zetal^{-\ell} T_1) E_{\eta_1^\ell} \\
		&= \prod_{m=1}^{\ell} (U_1 - p(\zetal^{-m}) + m\hbar) f(U_2,\dots,U_n,U_1+\ell\hbar) E'.
	\end{align*}
	We can rewrite
	\begin{align*}
		\prod_{m=1}^{\ell} (U_1 - p(\zetal^{-m}) + m\hbar)
		&= \prod_{m=1}^{\ell} (U_1 + \ell\hbar - p(\zetal^{\ell-m}) - (\ell-m)\hbar)  \\
		&= \prod_{m=1}^{\ell} (U_1 + \ell\hbar - s_{\ell-m})  \\
		&= \prod_{m=0}^{\ell-1} (U_1 + \ell\hbar - s_{m}),
	\end{align*}
	giving us (\ref{action1}).
	
	To verify (\ref{action3}), observe that
	\begin{align*}
		\sigma^n &\cdot f(\U;\T) E' \\
		&= \sigma^{n-1} \cdot (U_1 - p(\zetal^{-1} T_1) + \hbar)) \\
		& \hspace{1.5cm} f(U_2,\dots, U_n, U_1 + \hbar; T_2, \dots, T_n, \zetal^{-1} T_1) E_{\eta_1} \\
		&= \sigma^{n-2} \cdot (U_1 - p(\zetal^{-1} T_1) + \hbar)) (U_2 - p(\zetal^{-1} T_2) + \hbar)) \\
		& \hspace{1.5cm} f(U_3,\dots, U_n, U_1 + \hbar, U_2 + \hbar; T_3, \dots, T_n, \zetal^{-1} T_1, \zetal^{-1} T_2) E_{\eta_1 \eta_2} \\
		&= \prod_{i=1}^n (U_i - p(\zetal^{-1} T_i) + \hbar)) \\
		& \hspace{1.5cm} f(U_1 + \hbar,\dots, U_n + \hbar; \zetal^{-1} T_1,\dots, \zetal^{-1} T_n) E_{\eta_1 \dots \eta_n}.
	\end{align*}
	Then we have	
	\begin{align*}
		e' \sigma^{n\ell/p} &\cdot f(\U;\T) E' \\
		&= e' \sigma^{n(\ell/p-1)} \cdot \prod_{i=1}^n (U_i - p(\zetal^{-1} T_i) + \hbar)) \\
		& \hspace{1.5cm} f(U_1 + \hbar,\dots, U_n + \hbar; \zetal^{-1} T_1,\dots, \zetal^{-1} T_n) E_{\eta_1 \dots \eta_n} \\
		&= e' \sigma^{n(\ell/p-2)} \cdot \prod_{i=1}^n (U_i - p(\zetal^{-1} T_i) + \hbar)) \prod_{i=1}^n (U_i - p(\zetal^{-2} T_i) + 2\hbar))\\
		& \hspace{1.5cm} f(U_1 + 2\hbar,\dots, U_n + 2\hbar; \zetal^{-2} T_1,\dots, \zetal^{-2} T_n) E_{\eta_1^2 \dots \eta_n^2} \\
		&= e' \cdot \prod_{m=1}^{\ell/p} \prod_{i=1}^n (U_i - p(\zetal^{-m} T_i) + m\hbar)) \\
		& \hspace{1.5cm} f(U_1 + \ell\hbar/p,\dots, U_n + \ell\hbar/p; \zetal^{-\ell/p} T_1,\dots, \zetal^{-\ell/p} T_n) E_{\eta_1^{\ell/p} \dots \eta_n^{\ell/p}} \\
		&=  \prod_{m=1}^{\ell/p} \prod_{i=1}^n (U_i - p(\zetal^{-m} T_i) + m\hbar)) f(U_1 + \ell\hbar/p,\dots, U_n + \ell\hbar/p) E',
	\end{align*}
	which gives (\ref{action3}) after reindexing the product.
	
	The calculation to verify (\ref{action5}) is slightly more involved than the others. The most important observation is that the correction terms from the $S_n$ action are all projected out. We will first show that the correction terms vanish, then we will calculate the change to $f(\U)$.
	
	First, note that
	\begin{align*}
		y_{n-i}^{\ell-k\ell/p-1} &(n-i,\dots,n)  \tau \cdot f(\U;\T) E_{\eta_{n-i-j}^{k\ell/p} \dots \eta_{n-i}^{k\ell/p}} \\
		&= y_{n-i}^{\ell-k\ell/p-1} (n-i,\dots,n) \cdot f(\U;\T)^\tau E_{\eta_n^{-1} \eta_{n-i-j-1}^{k\ell/p} \dots \eta_{n-i-1}^{k\ell/p}} \\
		&= y_{n-i}^{\ell-k\ell/p-1} \cdot \Big( f(\U;\T)^{(n-i,\dots,n)\tau} \\
		& \hspace{1.5cm} + (1-E_{\eta_{n-i}^{-1} \eta_{n-i-j-1}^{k\ell/p} \dots \eta_{n-i-1}^{k\ell/p}}) a(f) \Big) E_{\eta_{n-i}^{-1} \eta_{n-i-j-1}^{k\ell/p} \dots \eta_{n-i-1}^{k\ell/p}} \\
		&= y_{n-i}^{\ell-k\ell/p-1} \cdot f(\U;\T)^{(n-i,\dots,n)\tau} E_{\eta_{n-i-j-1}^{k\ell/p} \dots \eta_{n-i-1}^{k\ell/p} \eta_{n-i}^{-1}} \\
		&= f(\U;\T)^{y_{n-i}^{\ell-k\ell/p-1} (n-i,\dots,n) \tau} E_{\eta_{n-i-j-1}^{k\ell/p} \dots \eta_{n-i-1}^{k\ell/p} \eta_{n-i}^{-(\ell-k\ell/p)}} \\
		&= f(\U;\T)^{y_{n-i}^{\ell-k\ell/p-1} (n-i,\dots,n) \tau} E_{\eta_{n-i-j-1}^{k\ell/p} \dots \eta_{n-i}^{k\ell/p}},
	\end{align*}
	which gives us
	\begin{align*}
		\left( y_{n-i}^{\ell-k\ell/p-1} (n-i,\dots,n) \tau \right)^{n-i} \cdot f(\U;\T) E' = \tilde f(\U;\T) E_{\eta^{k\ell/p}_1 \dots \eta^{k\ell/p}_{n-i}}
	\end{align*}
	where
	\begin{align*}
		\tilde{f}(\U;\T) &= f(\U;\T)^{\left( y_{n-i}^{\ell-k\ell/p-1} (n-i,\dots,n) \tau \right)^{n-i}} \\
		&= f(U_{1} + k\ell\hbar/p - \ell\hbar,\dots, U_{n-i} + k\ell\hbar/p - \ell\hbar, U_{n-i+1}, \dots U_n; \\
		& \hspace{4cm} \zetal^{k\ell/p} T_{1}, \dots,\zetal^{k\ell/p} T_{n-i}, T_{n-i+1}, \dots T_n).
	\end{align*}
	
	We then do the same thing with the $x_1^{k\ell/p-1} \sigma$ terms, but these terms pick up extra multiplicative factors. We have
	\begin{align*}
		x_{1}^{k\ell/p-1}  \sigma &\cdot \prod_{j'=1}^{j-1} \prod_{m=1}^{k\ell/p} (U_{j'} - p(\zetal^{-m} T_{j'}) + m\hbar) f(\U;\T) E_{\eta^{k\ell/p}_1 \dots \eta^{k\ell/p}_{n-i+j-1}} \\
		&= x_{1}^{k\ell/p-1} \cdot (U_1 - p(\zetal^{-1} T_1) + \hbar) \\
		& \hspace{1.5cm} \prod_{j'=2}^j \prod_{m=1}^{k\ell/p} (U_{j'} - p(\zetal^{-m} T_{j'}) + m\hbar) f(\U;\T)^{\sigma} E_{\eta_1 \eta^{k\ell/p}_2 \dots \eta^{k\ell/p}_{n-i+j}} \\
		&=\prod_{j'=1}^j \prod_{m=1}^{k\ell/p} (U_{j'} - p(\zetal^{-m} T_{j'}) + m\hbar) f(\U;\T)^{x_{1}^{k\ell/p-1}\sigma} E_{\eta_1^{k\ell/p} \dots \eta^{k\ell/p}_j \dots \eta^{k\ell/p}_{n-i+j}}.
	\end{align*}
	We then find
	\begin{align*}
		& e' \left( x_{1}^{k\ell/p-1} \sigma \right)^i \cdot \tilde f(\U;\T) E_{\eta^{k\ell/p}_1 \dots \eta^{k\ell/p}_{n-i}} \\
		& \hspace{3cm}= e' \cdot \prod_{j=1}^i \prod_{m=1}^{k\ell/p} (U_j - p(\zetal^{-m} T_j) + m\hbar) \tilde{\tilde{f}}(\U;\T) E_{\eta_1^{k\ell/p} \dots \eta^{k\ell/p}_{n}} \\
		& \hspace{3cm}= \prod_{j=1}^i \prod_{m=1}^{k\ell/p} (U_j - p(\zetal^{-m}) + m\hbar) \tilde{\tilde{f}}(\U) E' \\
		& \hspace{3cm}= \prod_{j=1}^i \prod_{m=0}^{k\ell/p-1} (U_j + k\ell\hbar/p - s_{m}) \tilde{\tilde{f}}(\U) E'
	\end{align*}
	where
	\begin{align*}
		&\tilde{\tilde{f}}(\U) = \tilde{f}(\U)^{\left( x_{1}^{k\ell/p-1} \sigma \right)^i} \\
		&= f(U_{i+1} + k\ell\hbar/p - \ell\hbar, \dots, U_{n} + k\ell\hbar/p - \ell\hbar, U_{1} + k\ell\hbar/p, \dots, U_i + k\ell\hbar/p).
	\end{align*}
	This completes the calculation.
\end{proof}

\subsection{Principal Galois orders}\label{sec:galois}
In this section, we recall the relevant definitions and results from \cite{Hartwig} that we will use to prove that the partially spherical and spherical subalgebras are principal Galois orders. Then, we proceed with the proofs.

First, we set some notation. Let $\Lambda$ be an integrally closed domain, $G$ a finite subgroup of $\Aut(\Lambda)$, and $\sM$ a submonoid of $\Aut(\Lambda)$ satisfying (A1)-(A3) of \cite{Hartwig}. Let $L = \Frac \, \Lambda$ be the field of fractions of $\Lambda$ and $\sL = L \star \sM$ be a skew monoid ring, the free left $L$-module with multiplication given by $a_1 \mu_1 \cdot a_2 \mu_2 = (a_1 \mu_1(a_2)) \mu_1 \mu_2$ for $a_i \in L$, $\mu_i \in \sM$. Let $\Gamma = \Lambda^G$, $K = L^G$, and $\sK = \sL^G$ be the subrings invariant under $G$.

\begin{definition}[\cite{FOgalois}]
	A $\Gamma$-subring $\sU \subseteq \sK$ is a {\bf Galois $\Gamma$-ring} in $\sK$ if $$\sU K = \sK = K \sU.$$
	A Galois $\Gamma$-ring $\sU$ in $\sK$ is a {\bf left/right Galois $\Gamma$-order} in $\sK$ if for any finite-dimensional left/right $K$-subspace $W \subseteq \sK$, the intersection $W \cap \sU$ is a finitely generated left (right) $\Gamma$-module. A Galois $\Gamma$-ring $\sU$ in $\sK$ is a {\bf Galois $\Gamma$-order} in $\sK$ if $\sU$ is a left and right Galois $\Gamma$-order in $\sK$.
\end{definition}

Based solely on this definition, we would have to establish whether $W \cap \sU$ is finitely generated for all choices of $W \subseteq \sK$, which could be very difficult. However, \cite{Hartwig} provides another way to verify whether a Galois $\Gamma$-ring is a Galois $\Gamma$-order.  Note that  the field $K$ has a natural left $\mathscr{K}$-module structure   via evaluation (as in \cite[Def. 2.18]{Hartwig}

\begin{theorem}[\cite{Hartwig}]\label{cor:PGO}
	If $\sU$ is a Galois $\Gamma$-ring in $\sK$ and $\Gamma\subset K$ is a submodule under the restricted $\sU$-module structure, then $\sU$ is a Galois $\Gamma$-order.
\end{theorem}
We call Galois orders satisfying the hypotheses of this theorem {\bf principal.}

For the case of the partially spherical subalgebra, let $\Lambda = \CC[U_1,\dots,U_n]$ and $G=1$, so $\Lambda = \Gamma$. Let $T = \langle \mu_1^{\pm\ell}, \dots, \mu_n^{\pm\ell}, (\mu_1\dots \mu_n)^{\pm\ell/p} \rangle \subset \ZZ^n$ be the lattice of translations acting on $\Lambda$ by $\mu_i^{\pm 1} (U_j) = U_j \pm \hbar\delta_{ij}$ and $s_{i,i+1} \in S_n$ act on $\Lambda$ by $s_{i,i+1} (U_j) = U_{(i,i+1)\cdot j}$. Let $\sM = T \rtimes S_n$. Then $L = K = \CC(U_1, \dots, U_n)$ and $\sL = \sK = \CC(U_1, \dots, U_n) \star \left( T \rtimes S_n \right)$. Showing that $e'He' \cdot \CC[U_1,\dots,U_n] \subseteq \sK$ is a Galois ring amounts to showing that $e'He' \CC(U_1, \dots, U_n) = \CC(U_1, \dots, U_n) \star \left( T \rtimes S_n \right)$, which we do by showing that both rings have the same action on $\CC(U_1,\dots,U_n)$.

\begin{proposition}\label{prop:PSSA=GGR}
	The partially spherical subalgebra of the Cherednik algebra associated with $G(\ell,p,n)$ is a Galois $\CC[U_1,\dots,U_n]$-ring in $\CC(U_1, \dots, U_n) \star \left( T \rtimes S_n \right)$.
\end{proposition}

\begin{proof}
	We will first show that the generators of $\CC(U_1, \dots, U_n) \star \left( T \rtimes S_n \right)$ can be written in terms of elements of $e'He' \CC(U_1, \dots, U_n)$. This amounts to finding elements of $e'He'  \CC(U_1, \dots, U_n)$ that act on $\CC(U_1, \dots, U_n)$ in the same way that the generators of $T \rtimes S_n$ act. These generators are $s_{i,i+1}$, $\mu_j^{\pm\ell}$, and $(\mu_1 \dots \mu_n)^{\pm\ell/p}$ for $i=1,\dots,n-1$ and $j=1,\dots,n$. For the $S_n$ generators, we have
	\begin{align}
		s_{i,i+1} \left( f(\U) \right) = \left(\frac{U_{i+1} - U_i}{U_{i+1} - U_i + k\ell} (i,i+1) + \frac{k\ell}{U_{i+1} - U_i} \right) \cdot f(\U),
	\end{align}
	which is easily verified using relation $(2.15)$ from \cite{Webalt} under the projection $e'$. The remaining generators can be written
	\begin{align}
		\mu_i^\ell \left( f(\U) \right) &= \frac{1}{\prod_{m=0}^{\ell-1} (U_i + \ell\hbar - s_m)} s_{i,\dots,1} x_1^{\ell-1} \sigma s_{n,\dots,i} \cdot f(\U) \\
		\mu_i^{-\ell} \left( f(\U) \right) &= s_{i,\dots,n} y_n^{\ell-1} \tau s_{1,\dots,i} \cdot f(\U) \\
		(\mu_1 \dots \mu_n)^{\ell/p} \left( f(\U) \right) &= \frac{1}{\prod_{i=1}^n \prod_{m=0}^{\ell/p-1} (U_i + \ell\hbar/p - s_m)} \sigma^{n\ell/p} \cdot f(\U) \\
		(\mu_1 \dots \mu_n)^{-\ell/p} \left( f(\U) \right) &= \tau ^{n\ell/p} \cdot f(\U),
	\end{align}
	using the notation $s_{j,\dots,j'} = s_{j,j-1} s_{j-1,j-2} \dots s_{j'+2,j'+1} s_{j'+1,j'}$ for $j>j'$ and an analogous definition for $j<j'$. These too can be verified by a short computation using Proposition \ref{prop:action}. Thus, we have shown that $\CC(U_1, \dots, U_n) \star \left( T \rtimes S_n \right) \subseteq e'He' \CC(U_1, \dots, U_n)$.
	
	Next, we'll show that the generators of $e'He' \CC(U_1, \dots, U_n)$ can be written as elements of $\CC(U_1, \dots, U_n) \star \left( T \rtimes S_n \right)$ by writing the action of the generators of $e'He'$ on $\CC(U_1, \dots, U_n)$ in terms of the action of elements of $\CC(U_1, \dots, U_n) \star \left( T \rtimes S_n \right)$. We have
	\begin{align}
		& (i,i+1) \cdot f(\U) = \left( \left( 1 - \frac{k\ell}{U_{i+1}-U_i} \right) s_{i,i+1} - \frac{k\ell}{U_{i+1}-U_i} \right)\left( f(\U) \right) \\
		& x_1^{\ell-1} \sigma \cdot f(\U) = \prod_{m=0}^{\ell-1} (U_1 + \ell\hbar - s_m) \left( \mu_1^\ell s_{1,2} s_{2,3} \dots s_{n-1,n} \right) \left( f(\U) \right) \\
		& y_n^{\ell-1} \tau \cdot f(\U) = \left( \mu_n^{-\ell} s_{n-1,n} s_{n-2,n-1} \dots s_{1,2} \right) \left( f(\U) \right) \\
		& \sigma^{n\ell/p} \cdot f(\U) = \prod_{i=1}^n \prod_{m=0}^{\ell/p-1} (U_i + \ell\hbar/p - s_m) (\mu_1 \dots \mu_n)^{\ell/p} \left( f(\U) \right) \\
		& \tau ^{n\ell/p} \cdot f(\U) = (\mu_1 \dots \mu_n)^{-\ell/p} \left( f(\U) \right) \\
		& \left( x_1^{k\ell/p-1} \sigma \right)^i \left( y_{n-i}^{\ell-k\ell/p-1} (n-i,\dots,n) \tau \right)^{n-i} \cdot f(\U) \\
		& \hspace{2cm} = \prod_{j=1}^i \prod_{m=0}^{k\ell/p-1} (U_j + k\ell\hbar/p - s_m) (\mu_1 \dots \mu_n)^{k\ell/p} \mu_{i+1}^{-\ell} \dots \mu_n^{-\ell} \tilde{s}_i (f(\U)) \nonumber
	\end{align}
	where $\tilde{s}_i$ is the permutation that sends $U_1,\dots,U_n$ to $U_{i+1},\dots,U_n,U_1,\dots,U_i$.	This shows that $\CC(U_1, \dots, U_n) \star \left( T \rtimes S_n \right) = e'He' \CC(U_1, \dots, U_n)$, so that $e'He'$ satisfies the definition of a Galois $\CC[U_1,\dots,U_n]$-ring.
\end{proof}

\begin{proposition}
	The partially spherical subalgebra is a principal Galois $\CC[U_1,\dots,U_n]$-order in $\CC(U_1, \dots, U_n) \star \left( T \rtimes S_n \right)$.
\end{proposition}

\begin{proof}
	We have $e'He' \cdot \CC[U_1,\dots,U_n] \subset \CC[U_1,\dots,U_n]$ because $H$ preserves $\CC[U_1,\dots,U_n; T_1,\dots,T_n]$,. By this fact and Proposition \ref{prop:PSSA=GGR}, the conditions of Theorem \ref{cor:PGO} are satisfied and the result follows.
\end{proof}

We will now show analogous results for the spherical subalgebra. First, let $G' = S_n$ and $\sM' = T$ so that $\Gamma' = \CC[U_1,\dots,U_n]^{S_n}$, $K' = \CC(U_1,\dots,U_n)^{S_n}$ and $\sK' = \left(\CC(U_1, \dots, U_n) \star T \right)^{S_n}$. 

\begin{proposition} \label{prop:SSA=GGR}
	The spherical subalgebra is a Galois $\CC[U_1,\dots,U_n]^{S_n}$-ring in $\left(\CC(U_1, \dots, U_n) \star T \right)^{S_n}$.
\end{proposition}

\begin{proof}
	By Proposition \ref{prop:PSSA=GGR}, we have $$e'He' \CC(U_1,\dots,U_n) = \CC(U_1,\dots,U_n) \star (T \rtimes S_n) = \CC(U_1,\dots,U_n) e'He'.$$ By symmetrizing over $S_n$, we obtain $$eHe \CC(U_1,\dots,U_n)^{S_n} = \left( \CC(U_1,\dots,U_n) \star T \right)^{S_n} = \CC(U_1,\dots,U_n)^{S_n} eHe,$$ completing the proof.
\end{proof}

\begin{proposition}
	The spherical subalgebra of the Cherednik algebra is a principal Galois $\CC[U_1,\dots,U_n]^{S_n}$-order in $\left(\CC(U_1, \dots, U_n) \star T \right)^{S_n}$.
\end{proposition}

\begin{proof}
	The result is immediate from the fact that $eHe \cdot \CC[U_1,\dots,U_n]^{S_n} \subset \CC[U_1,\dots,U_n]^{S_n}$ and Proposition \ref{prop:SSA=GGR}.
\end{proof}

\section{The representation theory of \texorpdfstring{$\mathsf{H}_p$}{Hp}}
Throughout this section, we assume that $\K$ is characteristic 0.  Thus, $\K$ is a $\Q$-vector space, and we can define a $\Q$-linear map $\Upsilon\colon \K\to \R$ such that $\Upsilon(1)= 1$.  The characteristic example the reader should have in mind is $\K=\C$ and $\Upsilon$ is the map of taking real part. 

\subsection{Weighted KLR algebras}

We first review the definition and basic facts about certain algebras that appear in the representation theory of $\mathsf{H}_1$.   These algebras were defined by the second author in \cite{WebwKLR}, and their connection to $\mathsf{H}_1$ is discussed in \cite{WebBKnote, WebRou, Webalt}; one of our goals in the present manuscript is to draw out this connection to include $G(\ell, p,n)$.  

Fix a field $\K$ of characteristic coprime to the order of $G(\ell,1,n)$ and consider the quotient the quotient of abelian groups $\K/\Z $.  Fix a subset $D\subset \K/\Z$, scalars $k\in \K$ and $\cg\in \R$ and $\ell$-tuples $(\vartheta_1,\dots, \vartheta_\ell)\in \R^\ell$ and $(d_1,\dots, d_n)\in D^\ell$.  As usual, we'll make $D$ into a quiver by adding an edge $\bar{a}\to \overline{a+k}$ whenever both lie in $D$; we extend this to a Crawley-Boevey quiver by adding a new vertex and an edge from this vertex to $d_i$ for each $i$.  In the convention for dimension vectors familiar from the work of Nakajima, this means that $w_d$ for $d\in D$ is the number of $i$ where this element appears as $d_i$.
\begin{definition}
  We let a  {\bf weighted
KLR diagram} be a collection of curves in
  $\R\times [0,1]$ with each curve mapping diffeomorphically to
  $[0,1]$ via the projection to the $y$-axis.  Each curve is allowed
  to carry any number of dots, and has a label that lies in $D$. We draw:
  \begin{itemize}
  \item a dashed line $\cg$ units to the right of each strand,
    which we call a {\bf ghost},
\item red
lines at $x=\theta_k$  labeled
with the fundamental weight for the node  $d_k\in D$.  
  \end{itemize} 
We now require that there are no
  triple points or tangencies involving any combination of strands,
  ghosts or red lines and no dots lie on crossings.  We consider these diagrams equivalent if they
  are related by an isotopy that avoids these tangencies, double
  points and dots on crossings.
\end{definition} 
The intersection of such a diagram with $y=0$ or $y=1$ gives a {\bf
  loading}, that is, a labeling of a finite subset of $\R$
with vertices of the quiver $D$.   

Let $\mu_\ell(\K)$ be the group of $\ell$th roots of unity in $\K$ as an abelian group under multiplication.
  Let $\Sigma\colon \K\times \mu_\ell(\K)\to \K/
\Z$ be the homomorphism $\Sigma(a,z)=\frac{\bar{a}}{\ell}+\gamma(z)$.
  For every pair of $n$-tuples $\Ba\in \K^n$ and $\Bz\in \mu_{\ell}(\K)^n$ with
  $\Sigma(a_i,z_i)\in D$ for all $i$, we can define
  a loading $e(\Ba,\Bz)$ as follows: we label the real number
  $\Upsilon(\frac{a_i}{\ell})+i\epsilon$ with the element $\Sigma(a_i,z_i)\in D$.
\begin{definition}
  Consider the algebra $R_D$ spanned by weighted KLR diagrams whose
  top and bottom both give loadings of the form $e(\Ba,\Bz)$ with
  $\Sigma(a_i,z_i)\in D$ modulo the local relations

\newseq
\begin{equation*}\subeqn\label{dots-1}
    \begin{tikzpicture}[scale=.45,baseline]
      \draw[very thick](-4,0) +(-1,-1) -- +(1,1) node[below,at start]
      {$d$}; \draw[very thick](-4,0) +(1,-1) -- +(-1,1) node[below,at
      start] {$d'$}; \fill (-4.5,.5) circle (5pt);
      \node at (-2,0){=}; \draw[very thick](0,0) +(-1,-1) -- +(1,1)
      node[below,at start] {$d$}; \draw[very thick](0,0) +(1,-1) --
      +(-1,1) node[below,at start] {$d'$}; \fill (.5,-.5) circle (5pt);
      \node at (4,0){for $d\neq d'$};
    \end{tikzpicture}\end{equation*}
\begin{equation*}\label{dots-2}\subeqn
    \begin{tikzpicture}[scale=.45,baseline]
      \draw[very thick](-4,0) +(-1,-1) -- +(1,1) node[below,at start]
      {$d$}; \draw[very thick](-4,0) +(1,-1) -- +(-1,1) node[below,at
      start] {$d$}; \fill (-4.5,.5) circle (5pt);
      \node at (-2,0){=}; \draw[very thick](0,0) +(-1,-1) -- +(1,1)
      node[below,at start] {$d$}; \draw[very thick](0,0) +(1,-1) --
      +(-1,1) node[below,at start] {$d$}; \fill (.5,-.5) circle (5pt);
      \node at (2,0){+}; \draw[very thick](4,0) +(-1,-1) -- +(-1,1)
      node[below,at start] {$d$}; \draw[very thick](4,0) +(0,-1) --
      +(0,1) node[below,at start] {$d$};
    \end{tikzpicture}\qquad 
    \begin{tikzpicture}[scale=.45,baseline]
      \draw[very thick](-4,0) +(-1,-1) -- +(1,1) node[below,at start]
      {$d$}; \draw[very thick](-4,0) +(1,-1) -- +(-1,1) node[below,at
      start] {$d$}; \fill (-4.5,-.5) circle (5pt);
      \node at (-2,0){=}; \draw[very thick](0,0) +(-1,-1) -- +(1,1)
      node[below,at start] {$d$}; \draw[very thick](0,0) +(1,-1) --
      +(-1,1) node[below,at start] {$d$}; \fill (.5,.5) circle (5pt);
      \node at (2,0){+}; \draw[very thick](4,0) +(-1,-1) -- +(-1,1)
      node[below,at start] {$d$}; \draw[very thick](4,0) +(0,-1) --
      +(0,1) node[below,at start] {$d$};
    \end{tikzpicture}
  \end{equation*}
\begin{equation*}\label{strand-bigon}\subeqn
    \begin{tikzpicture}[very thick,scale=.8,baseline]
      \draw (-2.8,0) +(0,-1) .. controls (-1.2,0) ..  +(0,1)
      node[below,at start]{$d$}; \draw (-1.2,0) +(0,-1) .. controls
      (-2.8,0) ..  +(0,1) node[below,at start]{$d$}; \node at (-.5,0)
      {=}; \node at (0.4,0) {$0$};
\node at (1.5,.05) {and};
    \end{tikzpicture}
\hspace{.4cm}
    \begin{tikzpicture}[very thick,scale=.8 ,baseline]

      \draw (-2.8,0) +(0,-1) .. controls (-1.2,0) ..  +(0,1)
      node[below,at start]{$d$}; \draw (-1.2,0) +(0,-1) .. controls
      (-2.8,0) ..  +(0,1) node[below,at start]{$d'$}; \node at (-.5,0)
      {=};

\draw (1.8,0) +(0,-1) -- +(0,1) node[below,at start]{$d'$};
      \draw (1,0) +(0,-1) -- +(0,1) node[below,at start]{$d$}; 
    \end{tikzpicture}
  \end{equation*} 
\begin{equation*}\label{ghost-bigon1}\subeqn
\begin{tikzpicture}[very thick,xscale=1.6 ,yscale=.8,baseline]
 \draw (1,-1) to[in=-90,out=90]  node[below, at start]{$d$} (1.5,0) to[in=-90,out=90] (1,1)
;
  \draw[dashed] (1.5,-1) to[in=-90,out=90] (1,0) to[in=-90,out=90] (1.5,1);
  \draw (2.5,-1) to[in=-90,out=90]  node[below, at start]{$d'$} (2,0) to[in=-90,out=90] (2.5,1);
\node at (3,0) {=};
  \draw (3.7,-1) -- (3.7,1) node[below, at start]{$d$}
 ;
  \draw[dashed] (4.2,-1) to (4.2,1);
  \draw (5.2,-1) -- (5.2,1) node[below, at start]{$d'$};  \node at (6.2,0){for $d+k\neq d'$};
\end{tikzpicture}
\end{equation*} 
\begin{equation*}\label{ghost-bigon1a}\subeqn
\begin{tikzpicture}[very thick,xscale=1.6 ,yscale=.8,baseline]
 \draw (1.5,-1) to[in=-90,out=90]  node[below, at start]{$d$} (1,0) to[in=-90,out=90] (1.5,1)
;
  \draw[dashed] (1,-1) to[in=-90,out=90] (1.5,0) to[in=-90,out=90] (1,1);
  \draw (2,-1) to[in=-90,out=90]  node[below, at start]{$d'$} (2.5,0) to[in=-90,out=90] (2,1);
\node at (3,0) {=};
  \draw (4.2,-1) -- (4.2,1) node[below, at start]{$d$}
 ;
  \draw[dashed] (3.7,-1) to (3.7,1);
  \draw (4.7,-1) -- (4.7,1) node[below, at start]{$d'$};  \node at (6.2,0){for $d+k\neq d'$};
\end{tikzpicture}
\end{equation*} \begin{equation*}\label{ghost-bigon2}\subeqn
\begin{tikzpicture}[very thick,xscale=1.4,baseline=25pt]
 \draw (1,0) to[in=-90,out=90]  node[below, at start]{$d$} (1.5,1) to[in=-90,out=90] (1,2)
;
  \draw[dashed] (1.5,0) to[in=-90,out=90] (1,1) to[in=-90,out=90] (1.5,2);
  \draw (2.5,0) to[in=-90,out=90]  node[below, at start]{$d+k$} (2,1) to[in=-90,out=90] (2.5,2);
\node at (3,1) {=};
  \draw (3.7,0) -- (3.7,2) node[below, at start]{$d$}
 ;
  \draw[dashed] (4.2,0) to (4.2,2);
  \draw (5.2,0) -- (5.2,2) node[below, at start]{$d+k$} node[midway,fill,inner sep=2.5pt,circle]{};
\node at (5.75,1) {$-$};

  \draw (6.2,0) -- (6.2,2) node[below, at start]{$d$} node[midway,fill,inner sep=2.5pt,circle]{};
  \draw[dashed] (6.7,0)-- (6.7,2);
  \draw (7.7,0) -- (7.7,2) node[below, at start]{$d+k$};
\end{tikzpicture}
\end{equation*}
 \begin{equation*}\label{ghost-bigon2a}\subeqn
\begin{tikzpicture}[very thick,xscale=1.4,baseline=25pt]
 \draw (1.5,0) to[in=-90,out=90]  node[below, at start]{$d$} (1,1) to[in=-90,out=90] (1.5,2)
;
  \draw[dashed] (1,0) to[in=-90,out=90] (1.5,1) to[in=-90,out=90] (1,2);
  \draw (2,0) to[in=-90,out=90]  node[below, at start]{$d+k$} (2.5,1) to[in=-90,out=90] (2,2);
\node at (3,1) {=};
  \draw (4.2,0) -- (4.2,2) node[below, at start]{$d$}
 ;
  \draw[dashed] (3.7,0) to (3.7,2);
  \draw (4.7,0) -- (4.7,2) node[below, at start]{$d+k$} node[midway,fill,inner sep=2.5pt,circle]{};
\node at (5.25,1) {$-$};

  \draw (6.2,0) -- (6.2,2) node[below, at start]{$d$} node[midway,fill,inner sep=2.5pt,circle]{};
  \draw[dashed] (5.7,0)-- (5.7,2);
  \draw (6.7,0) -- (6.7,2) node[below, at start]{$d+k$};
\end{tikzpicture}
\end{equation*}
 \begin{equation*}\subeqn\label{triple-boring}
    \begin{tikzpicture}[very thick,scale=1 ,baseline]
      \draw (-3,0) +(1,-1) -- +(-1,1) node[below,at start]{$m$}; \draw
      (-3,0) +(-1,-1) -- +(1,1) node[below,at start]{$d$}; \draw
      (-3,0) +(0,-1) .. controls (-4,0) ..  +(0,1) node[below,at
      start]{$d'$}; \node at (-1,0) {=}; \draw (1,0) +(1,-1) -- +(-1,1)
      node[below,at start]{$m$}; \draw (1,0) +(-1,-1) -- +(1,1)
      node[below,at start]{$d$}; \draw (1,0) +(0,-1) .. controls
      (2,0) ..  +(0,1) node[below,at start]{$d'$};
    \end{tikzpicture}
  \end{equation*}
\begin{equation*}\subeqn \label{eq:triple-point1}
    \begin{tikzpicture}[very thick,xscale=1.6,yscale=.8,baseline]
      \draw[dashed] (-3,0) +(.4,-1) -- +(-.4,1);
 \draw[dashed]      (-3,0) +(-.4,-1) -- +(.4,1); 
    \draw (-1.5,0) +(.4,-1) -- +(-.4,1) node[below,at start]{$d+k$}; \draw
      (-1.5,0) +(-.4,-1) -- +(.4,1) node[below,at start]{$d+k$}; 
 \draw (-3,0) +(0,-1) .. controls (-3.5,0) ..  +(0,1) node[below,at
      start]{$d$};\node at (-.75,0) {=};  \draw[dashed] (0,0) +(.4,-1) -- +(-.4,1);
 \draw[dashed]      (0,0) +(-.4,-1) -- +(.4,1); 
    \draw (1.5,0) +(.4,-1) -- +(-.4,1) node[below,at start]{$d+k$}; \draw
      (1.5,0) +(-.4,-1) -- +(.4,1) node[below,at start]{$d+k$}; 
 \draw (0,0) +(0,-1) .. controls (.5,0) ..  +(0,1) node[below,at
      start]{$d$};
\node at (2.25,0)
      {$-$};   
     \draw (4.5,0)
      +(.4,-1) -- +(.4,1) node[below,at start]{$d+k$}; \draw (4.5,0)
      +(-.4,-1) -- +(-.4,1) node[below,at start]{$d+k$}; 
 \draw[dashed] (3,0)
      +(.4,-1) -- +(.4,1); \draw[dashed] (3,0)
      +(-.4,-1) -- +(-.4,1); 
\draw (3,0)
      +(0,-1) -- +(0,1) node[below,at start]{$d$};
    \end{tikzpicture}
  \end{equation*}
\begin{equation*}\subeqn\label{eq:triple-point2}
    \begin{tikzpicture}[very thick,xscale=1.6,yscale=.8,baseline]
\draw[dashed] (-3,0) +(0,-1) .. controls (-3.5,0) ..  +(0,1) ;  
  \draw (-3,0) +(.4,-1) -- +(-.4,1) node[below,at start]{$d$}; \draw
      (-3,0) +(-.4,-1) -- +(.4,1) node[below,at start]{$d$}; 
 \draw (-1.5,0) +(0,-1) .. controls (-2,0) ..  +(0,1) node[below,at
      start]{$d+k$};\node at (-.75,0) {=};  
    \draw (0,0) +(.4,-1) -- +(-.4,1) node[below,at start]{$d$}; \draw
      (0,0) +(-.4,-1) -- +(.4,1) node[below,at start]{$d$}; 
 \draw[dashed] (0,0) +(0,-1) .. controls (.5,0) ..  +(0,1);
 \draw (1.5,0) +(0,-1) .. controls (2,0) ..  +(0,1) node[below,at
      start]{$d+k$};
\node at (2.25,0)
      {$+$};   
     \draw (3,0)
      +(.4,-1) -- +(.4,1) node[below,at start]{$d$}; \draw (3,0)
      +(-.4,-1) -- +(-.4,1) node[below,at start]{$d$}; 
\draw[dashed] (3,0)
      +(0,-1) -- +(0,1);\draw (4.5,0)
      +(0,-1) -- +(0,1) node[below,at start]{$i+k$};
    \end{tikzpicture}.
  \end{equation*}
  \begin{equation*}\subeqn\label{cost}
  \begin{tikzpicture}[very thick,baseline,scale=.8,xscale=.7]
    \draw (-2.8,0)  +(0,-1) .. controls (0,0) ..  +(0,1) node[below,at start]{$d$};
       \draw[wei] (-1.4,0)  +(0,-1) -- +(0,1) node[below,at start]{$d$};
           \node at (0,0) {=};
    \draw[wei] (2.8,0)  +(0,-1) -- +(0,1) node[below,at start]{$d$};
       \draw (1.2,0)  +(0,-1) -- +(0,1) node[below,at start]{$d$} node[circle,fill=black, midway, inner sep=2pt]{};
  \end{tikzpicture}\qquad\qquad
  \begin{tikzpicture}[very thick,baseline,scale=.8,xscale=.7]
          \draw[wei] (6.8,0)  +(0,-1) -- +(0,1) node[below,at start]{$d'$};
  \draw (5.2,0)  +(0,-1) .. controls (8,0) ..  +(0,1) node[below,at start]{$d$};
           \node at (8,0) {=};
    \draw (9.2,0)  +(0,-1) -- +(0,1) node[below,at start]{$d$};
       \draw[wei] (10.8,0)  +(0,-1) -- +(0,1) node[below,at start]{$d'$};
  \end{tikzpicture}
\end{equation*} 
  \begin{equation*}\subeqn
    \begin{tikzpicture}[very thick,scale=.8,baseline]
      \draw (-3,0)  +(1,-1) -- +(-1,1) node[at start,below]{$d$};
      \draw (-3,0) +(-1,-1) -- +(1,1)node [at start,below]{$d'$};
      \draw[wei] (-3,0)  +(-.3,-1) -- node[at start,below]{$m$}+(-.3,1);
      \node at (-1,0) {=};
      \draw (1,0)  +(1,-1) -- +(-1,1) node[at start,below]{$d$};
      \draw (1,0) +(-1,-1) -- +(1,1) node [at start,below]{$d'$};
      \draw[wei] (1,0) +(.3,-1)-- node[at start,below]{$m$} +(.3,1);   
\node at (2.8,0) {$+ $};
      \draw (6.5,0)  +(1,-1) -- +(1,1)  node[at start,below]{$d$};
      \draw (6.5,0) +(-1,-1) -- +(-1,1) node [at start,below]{$d'$};
      \draw[wei] (6.5,0) +(0,-1) -- node[at start,below]{$m$} +(0,1);
\node at (3.8,-.2){$\delta_{d,d',m} $}  ;
 \end{tikzpicture}
  \end{equation*}
\begin{equation*}\subeqn\label{dumb}
    \begin{tikzpicture}[very thick,scale=.8,baseline]
      \draw[wei] (-3,0)  +(0,-1) -- +(0,1);
      \draw (-3,0)  +(-.5,-1) .. controls (-4,-.3) ..  +(.5,1);
      \draw (-3,0) +(-1,-1) -- +(1,1);
      \node at (-1,0) {=};
      \draw[wei] (1,0)  +(0,-1) -- +(0,1);
  \draw (1,0)  +(-.5,-1) .. controls (2,.3) ..  +(.5,1);
      \draw (1,0) +(-1,-1) -- +(1,1);    \end{tikzpicture}\qquad \qquad
    \begin{tikzpicture}[very thick,scale=.8,baseline]
  \draw(-3,0) +(-1,-1) -- +(1,1);
  \draw[wei](-3,0) +(0,-1) -- +(0,1);
\fill (-3.5,-.5) circle (3pt);
\node at (-1,0) {=};
 \draw(1,0) +(-1,-1) -- +(1,1);
  \draw[wei](1,0) +(0,-1) -- +(0,1);
\fill (1.5,.5) circle (3pt);
    \end{tikzpicture}
  \end{equation*}
For the relations (\ref{dumb}), we also include their mirror images.

This algebra is graded with \[
\deg\tikz[baseline,very thick,scale=1.5]{\draw (.2,.3) --
  (-.2,-.1) node[at end,below, scale=.8]{$d$}; \draw
  (.2,-.1) -- (-.2,.3) node[at start,below,scale=.8]{$d'$};} =-2\delta_{d,d'} \qquad  \deg\tikz[baseline,very thick,scale=1.5]{\draw[densely dashed] 
  (-.2,-.1)-- (.2,.3) node[at start,below, scale=.8]{$d$}; \draw
  (.2,-.1) -- (-.2,.3) node[at start,below,scale=.8]{$d'$};} =\delta_{d',d-k} \qquad \deg\tikz[baseline,very thick,scale=1.5]{\draw (.2,.3) --
  (-.2,-.1) node[at end,below, scale=.8]{$d$}; \draw [densely dashed]
  (.2,-.1) -- (-.2,.3) node[at start,below,scale=.8]{$d'$};}
=\delta_{d',d+k}\]
\[
\deg\tikz[baseline,very thick,scale=1.5]{\draw
  (0,.3) -- (0,-.1) node[at end,below,scale=.8]{$d$}
  node[midway,circle,fill=black,inner
  sep=2pt]{};}=2 \qquad \deg\tikz[baseline,very thick,scale=1.5]{\draw[wei] 
  (0,-.1)-- (0,.3) node[at start,below, scale=.8]{$d$}; \draw
  (.2,-.1) -- (-.2,.3) node[at start,below,scale=.8]{$d'$};} =\delta_{d',d} \qquad \deg\tikz[baseline,very thick,scale=1.5]{\draw (.2,.3) --
  (-.2,-.1) node[at end,below, scale=.8]{$d$}; \draw [wei]
  (0,-.1) -- (0,.3) node[at start,below,scale=.8]{$d'$};}
=\delta_{d',d}\] and we'll also consider the completion $\widehat{R}_D$ of this
  algebra with respect to its grading.  
\end{definition}


  It will often be technically more convenient for us to think of
  $R_D$ or $\widehat{R}_D$ as a category whose objects are loadings
  and whose morphisms are elements of $R_D$ matching the source
  loading at the bottom and target loading at the top; this is the
  standard trick for considering a ring with set of idempotents
  summing to the identity as a category, discussed in \cite[\S 3.1]{Lauintro}.

  \begin{remark}\label{rem:equiv}
    As we've defined it, the algebra $R_D$ is infinite rank as a
    module over $\K[y_1,\dots, y_n]$, since we consider the
    $x$-values of the strands at the top and bottom of the diagram as
    fixed.  However, if two loadings are related by an isotopy
    (i.e. the straight line diagram relating them has no crossings),
    they are equivalent objects in the category $R_D$.  This is {\bf
      equivalence} of loadings, as discussed in
    \cite[Def. 2.9]{WebwKLR}.  As in \cite[Def. 2.13]{WebwKLR}, we
    usually take ``weighted KLR algebra'' to mean the algebra Morita
    equivalent to $R_D$ where we keep only one loading from each
    equivalence class.  
    
    Note that being finite-dimensional isn't invariant under this equivalence, so it is better to consider {\bf locally finite-dimensional} modules, those with $e(\Ba,\Bz)M$ is finite dimensional over $\K$ for all $(\Ba,\Bz)$.
  \end{remark}

\begin{definition}\label{def:wgmod}
Let $R_D\wgmod$ be the category of finitely generated $R_D$-modules $M$ which are weakly gradable and locally finite-dimensional, that is, $M$ has a finite filtration with gradable subquotients and  $e(\Ba,\Bz)M$ is finite dimensional over $\K$ for all $(\Ba,\Bz)$.  

These are precisely the modules that extend to finitely generated modules over the completion $\widehat{R}_D$ with the discrete topology.  
\end{definition}

  A structure on the category of $R_D$-modules that will be important for us is duality of modules.  There is an anti-involution $\psi$ of $R_D$ defined by reflecting diagrams through the line $y=\frac{1}{2}$ and multiplying by $-1$ raised to the number of crossings of strands with the same label (note the similarity to the anti-involution $\psi$ defined for the original KLR algebra in \cite[\S 2.1]{KLI}).  We can use $\psi$ to switch between left and right modules.
  \begin{definition}
  The {\bf contragredient} $M^\circledast$ of a left $R_D$-module $M$ is the elements of the $\K$-vector space dual which are zero on the image of almost all idempotents $e(\Ba,\Bz)$, with the  left module structure induced by $\psi$.  
  \end{definition}
  Of course, the contragredient of a graded module has a natural dual grading.  
  Let us collect a couple of useful observations about $R_D$ which are implicit in earlier literature, but which are useful to have here:
  \begin{lemma}
  If $\K$ is characteristic 0, every simple module over $\widehat{R}_D$ has a unique grading for which is self-dual as a graded module.   
  \end{lemma}
  \begin{proof}
  This is implicit in \cite[Thm. 3.22]{Webalt}; since this theorem is stated for Dunkl-Opdam modules, we must translate a bit.  Any simple over $\widehat{R}_D$ has a corresponding simple Dunkl-Opdam module.  
  By the classification in \cite[Thm. 3.22]{Webalt}, this arises from a multi-segment $\bQ$ and $\ell$-multipartition $\xi$, and these can be read off from the collection of pairs $(\Ba,\Bz)$ with $\Wei_{(\Ba,\Bz)}(S)$ which are non-zero. Since this is preserved by duality, every simple is self-dual.  Furthermore, one can construct $(\Ba,\Bz)$ where this weight space is one dimensional: the Dunkl-Opdam operators act on the representation $V_\xi$ with multiplicity free spectrum, and while this is not always true for $L(\bQ)$, the invariants under the Young subgroup of $S_{|\bQ|}$ corresponding to the decomposition into segments has a 1 dimensional space of invariants uniquely characterized by its weight under the Dunkl-Opdam elements $u_i$ (which correspond to the polynomial generators in the degenerate affine Hecke algebra). 
  
  Thus, if we grade $S$ so that $e(\Ba,\Bz)S$ for this choice of $(\Ba,\Bz)$  is in degree 0, then $S$ will be graded self-dual.
  \end{proof}
  \begin{remark}
  While the proof of this result used the characteristic 0 hypothesis (needed for \cite[Thm. 3.22]{Webalt}), in fact if one translates the proof of \cite[Thm. 3.22]{Webalt} into purely diagrammatic terms, it works in characteristic $p$ as well.
  \end{remark}

Assume that we have an order $p$ permutation $\sigma\in S_{\ell}$ such that $\theta_{\sigma(k)}=\theta_{k}$ and $d_{\sigma(k)}=d_{k}+\frac{1}{p}$ for all $k$.  In this case, we have an order $p$ permutation $\mathsf{a}$ of the set of weighted KLR diagrams induced by shifting all labels by $d\mapsto d+\frac{1}{p}$.  Our conditions on the $\sigma$ above assure that the set of red lines is unchanged, so all the relations (\ref{dots-1}--\ref{dumb}) and the grading are preserved.  
\begin{lemma}\label{lem:KLR-Morita}
The induced map $\mathsf{a}\colon R_D\to R_D$ generates a Morita cyclic group action of $R_D$.
\end{lemma}
\begin{proof}
The compatibility discussed above shows that it is a ring automorphism, and the fact that $C_p$ acts freely on $\K/\Z$ by $d\mapsto d+\frac{1}{p}$ shows that this action is faithful.

In fact, the straight-line idempotents in $R_D$ carry a free action by $C_p$.  Thus, every isotypic component for $C_p$ contains a sum of these idempotents where $e(\Ba,\Bz)$ appears with non-zero coefficient.  Thus, $e(\Ba,\Bz)$ lies in the two-sided ideal generated by this isotypic component for all $(\Ba,\Bz)$, showing this ideal is whole algebra.  This shows that the action is Morita. 
\end{proof}

\subsection{Review of the representation theory of \texorpdfstring{$\mathsf{H}_1$}{H1}}

Here, we quickly remind the reader of the results on the representation theory of $\mathsf{H}_1$ discussed in \cite{Webalt,WebRou}.   It will probably be helpful to the reader to We analyze the representations of this algebra by means of the functors \[\Wei_{\Ba,\Bz}(M)=\{m\in M\mid
(u_i-a_i)^Nm=(t_i-z_i)^Nm=0 \text{ for } N\gg 0\}.\] 
\begin{definition}
We call a module {\bf Dunkl-Opdam} if $M\cong \bigoplus_{(\Ba,\Bz)}\Wei_{\Ba,\Bz}(M)$, that is, if the subalgebra $\DO$ acts locally finitely.
\end{definition}  


Let $\mathcal{C}$ be the category of Dunkl-Opdam modules over $\mathsf{H}_1$.  This category naturally decomposes into blocks according to the orbit of $(\Ba,\Bz)$ under the action of the extended affine Weyl group $S_n\ltimes \Z^n$.  This group acts by
permutations and translations sending \[(\Ba,\Bz)\mapsto
((a_1+m_1,\dots, a_n+m_n),(\zetal^{m_1}z_1,\dots, \zetal^{m_n}z_n)).\] Two pairs lie in the same orbit if and only if their images under $\Sigma$
agree up to permutation of the entries.  As in the previous section, we fix $D\subset \K/\Z$ and let $\tilde{D}=\Sigma^{-1}(D)\subset \K \times \mu_{\ell}(\K)$.   
Let $R_D$ be the weighted KLR algebra attached to the quiver $D$ 
and the parameters attached to the set $D$ with the parameters $\cg=\Upsilon(k)$, $\theta_i=\Upsilon (h_i/\ell)$, and $d_i=\overline{s_i/\ell}$. 

This matches the conventions of 
\cite[Def. 3.8]{Webalt}, except that we don't add small real numbers to avoid placing red lines on top of each other: in \cite{Webalt}, these positions depend on a small parameter $\epsilon$, but if we send $\epsilon$ to 0 and let red strands with $h_m=h_{m'}$ sit at the same position, the resulting algebra is isomorphic since the loadings of the form $e(\Ba,\Bz)$ have no black strands between red strands with $h_m=h_{m'}$.  This is the algebra we will call $R_D$ in this paper.

Since we always have $h_{m}=h_{m+\ell/p}$, our red strands now come in {\bf packets} corresponding to the labels $\{m, m+\frac{\ell}{p},m+2\frac{\ell}{p}, \cdots\}$ for $1\leq m\leq \frac{\ell}{p}$, which all have the same $x$-value, and whose labels in $\K/\Z$ have the form $\{\mathsf{s}_m,\mathsf{s}_m+\frac{1}{p},\mathsf{s}_m+\frac{2}{p},\cdots \}$.
\begin{theorem}\label{thm:H1-case}
There is an equivalence $\mathsf{W}\colon \mathsf{H}\mmod_D\to \widehat{R}_D\wgmod$ sending $M\mapsto
\oplus_{\Ba,\Bz} W_{\Ba,\Bz}(M)$, induced by an isomorphism $\widehat{R}_D\to \End(\bigoplus_{(\Ba,\Bz)}W_{\Ba,\Bz})$.  
\end{theorem}

In \cite{Webalt}, we give an indexing set for the Dunkl-Opdam modules of $\mathsf{H}_1$.  
\begin{definition}
A {\bf charged segment} is a $g$-tuple (for some $g\leq n$) of elements
$\mathbf{q}=(q_1,\dots, q_g)$ of
$\K/\Z$, which satisfy $q_{i+1}-q_i=k$.  We'll use {\bf lifted
  segment} to mean a similar $g$-tuple $\mathbf{a}$ in $\K$ satisfying $a_{i+1}-a_i=k\ell$, together with a scalar $z\in \mu_\ell(\K)$.
  
  As usual, a {\bf multisegment} is a multi-set of charged segments, and the size of a multisegment is the sum of the sizes of the constituent segments.
\end{definition} 
Our indexing set will be the set $\pairs_n$  of pairs $\chi=(\bQ,\xi)$ with $\mathbf{Q}$ a choice of charged multisegment  and $\xi$ an $\ell$-multipartition  whose sizes add to $n$.

Fix a negative integer  $N\ll 0$ and let $\Lambda(q_1,\dots, q_g)$ for a charged segment be the unique
lifted segment $(a_1,\dots, a_g), z$ of elements of $\K$ such that
$\Sigma(a_i,z)=q_i$,  $a_{i+1}-a_i=k\ell$, and
$\Upsilon(a_1)$ is minimized subject to  $P\leq \Upsilon(a_i)$; this
means that $P\leq \Upsilon(a_1)< P+1$ if $\Upsilon(k)\geq 0$ and
$P\leq \Upsilon(a_g)< P+1$  if  $\Upsilon(k)\leq 0$.

Let $\mathcal{D\!O}_n$ be the subalgebra of $\mathsf{H}_1$ generated by $G(\ell,1,n)$ and $\DO$; this is isomorphic to the degenerate affine Hecke algebra of $G(\ell,1,n)$.   Given $\chi=(\bQ,\xi)\in \pairs_n$, we can define an irreducible representation $K_{\chi}$ of $\mathcal{D\!O}_n$. The multisegment given by lifting each segment of $\mathbf{Q}$ induces a representation $L_{\bQ}$ of $\mathcal{D\!O}_{|\mathbf{Q}|}$ by the usual formalism of Zelevinsky, and the irrep $V_{\xi}$ of $G(\ell,1,|\xi|)$  can  be inflated to a $\mathcal{D\!O}_{|\xi|}$-module as described in \cite[Lem. 3.7]{Webalt}.  The representation $K_{\chi}$ is the (irreducible) convolution of these two modules.  In the notation of \cite{Webalt}, this is $K_{\chi}=\mathcal{D\!O}_n\otimes_{\mathcal{D\!O}_{|\mathbf{Q}|}\otimes \mathcal{D\!O}_{|\xi|}}L(\mathbf{Q})\otimes V_\xi$.

We let $\mathcal{M}(\chi)=\mathsf{H}_1\otimes_{\mathcal{D\!O}_n} K_{\chi}$, and $\Delta(\chi)$ be the quotient of $\mathcal{M}(\chi)$ by the submodule generated by the images of $\mathcal{M}(\mathbf{Q}', \xi')$ for any multi-segment $\mathcal{Q'}$ of larger size than $\mathcal{Q}$, or equal size with $c_{\xi'}<c_{\xi}$, with $c_{\xi}$ the usual $c$-function.
\begin{theorem}
For each simple Dunkl-Opdam module $S$ in $\mathsf{H}_1$, there is a unique pair $\chi\in \pairs_n$ such that $S=S(\chi)$ is the unique simple quotient of $\Delta(\chi)$.  
\end{theorem}

\subsection{The category of Dunkl-Opdam modules for \texorpdfstring{$\mathsf{H}_p$}{Hp}}
Now, we consider these results in the context of the $C_p$-action whose fixed points are $\mathsf{H}_p$.

 The associated category $\mathcal{C}^{C_p}$ of $\mathsf{H}_p$-modules which are Dunkl-Opdam after induction is simply the category where $\DO^{C_p}\subset \mathsf{H}_p$ acts locally finitely, and the equivalent category $\tilde{\mathcal{C}}$ over $\tilde{\mathsf{H}}_1=\mathsf{H}_1\# C_p$ is that where $\DO$ acts locally finitely. 

Just as above, we can think of vectors $\Ba\in \K^n$ and $\Bz\in \mu_p(\K)^n$ as giving a maximal ideal in $\DO$, and consider the generalized weight space functor $\tWei_{\Ba,\Bz}\colon \tilde{\mathsf{H}}_1\mmod\to \K\mmod$ defined by \[\tWei_{\Ba,\Bz}(M)=\{m\in M\mid
(u_i-a_i)^Nm=(t_i-z_i)^Nm=0 \text{ for } N\gg 0\}.\]  
Note that while this is defined by the same formula as the functor $\Wei_{\Ba,\Bz}\colon \mathsf{H}_1\mmod\to \K\mmod$, it is a different functor since the source category is different.  

Let $\Bz^\alpha=(\zetap^{-1}z_1,\dots, \zetap^{-1} z_n)$.  
\begin{lemma}\label{lem:tWei-action}
The action of $\auto$ induces an isomorphism of functors $\auto\colon \tWei_{\Ba,\Bz}\to \tWei_{\Ba,\Bz^{\auto}}$, and thus an action of $\auto$ on $\bigoplus_{(\Ba,\Bz)}\tWei_{(\Ba,\Bz)}$ whenever the indexing set is closed under $\auto$.
\end{lemma}
\begin{proof}
If $v\in W_{\Ba,\Bz}(M)$, then \[t_i\auto v=\zetap^{-1} \auto t_iv=\zetap^{-1} z_i(\auto v)\]
so $v\in W_{\Ba,\Bz^{\auto}}(M)$, and similarly $\auto^{-1}$ gives the inverse isomorphism.
\end{proof}

Now, let us consider a set $D\subset \K/\Z$, and assume that $D$ is closed under the action by addition of $\frac{1}p\Z/\Z$.  In this case, the set $\tilde{D}$ is closed under $\auto$, which induces a free $C_p$ action on $\tilde{D}$.  Let \[\tWei_D=\bigoplus_{(a_i,z_i)\in \tilde{D}}\tWei_{(\Ba,\Bz)}\qquad \Wei_D=\bigoplus_{(a_i,z_i)\in \tilde{D}}\Wei_{(\Ba,\Bz)}\]
and note that these are related by $\tWei_D=\Wei_D\circ \Res^{\tilde{\mathsf{H}_1}}_{\mathsf{H}_1}.$  This induces an obvious homomorphism $\End(\Wei_D)\to \End(\tWei_D)$.

From Lemma \ref{lem:tWei-action}, we have a homomorphism $C_p\to\End(\tWei_{(\Ba,\Bz)}).  $

\begin{proposition}\label{prop:smash-iso}
These homomorphisms induce an isomorphism $\End(\tWei_D)\cong \End(\Wei_D)\# C_p$. 
\end{proposition}
\begin{proof}
Letting $\mathfrak{m}_{(\Ba,\Bz)}\subset \DO$ be the maximal ideal corresponding to this point, we have an expression for this natural transformation spaces in question:
\[\Hom(\Wei_{(\Ba,\Bz)},\Wei_{(\Ba',\Bz')})\cong \varprojlim{\mathsf{H}}_1/{\mathsf{H}}_1\mathfrak{m}_{(\Ba,\Bz)}^N+\mathfrak{m}_{(\Ba',\Bz')}^N{\mathsf{H}}_1.\]
\[\Hom(\tWei_{(\Ba,\Bz)},\tWei_{(\Ba',\Bz')})\cong \varprojlim\tilde{\mathsf{H}}_1/\tilde{\mathsf{H}}_1\mathfrak{m}_{(\Ba,\Bz)}^N+\mathfrak{m}_{(\Ba',\Bz')}^N\tilde{\mathsf{H}}_1.\]  Using the fact that $\tilde{\mathsf{H}}_1\cong \bigoplus_{m=0}^{p-1} \mathsf{H}_1\auto^m$, we have that 
\[\Hom(\tWei_{(\Ba,\Bz)},\tWei_{(\Ba',\Bz')})\cong \bigoplus_{m=0}^{p-1} \Hom(\Wei_{\auto^m(\Ba,\Bz)},\Wei_{(\Ba',\Bz')})\auto^m,\] from which the result is clear.
\end{proof}
Now, we wish to study the compatibility of this isomorphism with the KLR type presentation of $\End(\Wei_D)\cong \widehat{R}_D$.  Note that $R_D$ has an automorphism $R_D\to R_D$ induced by labelling all black strands via the automorphism $\bar{a}\mapsto \bar{a}-\frac{1}{p}$ of $D$.  This is well-defined since by assumption, the labels of each packet of red strands form an orbit under this action, and so the overall set of red strands with labels and $x$-values is left unchanged.   
\begin{theorem}
This presentation of Proposition \ref{prop:smash-iso} induces an isomorphism $\End(\tWei_D)\cong \widehat{R}_D\# C_p$ which agrees with the automorphism $\mathsf{a}$ of Lemma \ref{lem:KLR-Morita}.
\end{theorem}
\begin{proof}
It's clear from the formulas of \cite[Lem. 3.12]{Webalt} that the action of $C_p$ on $\Wei_D$ is compatible with the action on $\widehat{R}_D$ by relabeling.
\end{proof}
As in Definition \ref{def:wgmod}, we can consider the category $\tilde{R}_D\wgmod$ of finitely generated, weakly gradable, and locally finite-dimensional modules over $\tilde{R}_D=R_D\#C_p$.  
\begin{corollary}
The category $\mathsf{H}_p\mmod_D$ of Dunkl-Opdam $\mathsf{H}_p$-modules with supports in $D^n$ are equivalent to the category $\tilde{R}_D\wgmod$, via the functor $\Wei\circ \Ind^{\mathsf{H}_1}_{\mathsf{H}_p}$.  
\end{corollary}

\begin{definition}
A {\bf grading} of a Dunkl-Opdam module over $\mathsf{H}_p$ is a grading of the corresponding $\tilde{R}_D$-module.   
\end{definition}
Since the simple modules over $\tilde{R}_D$ are all gradable, the same is true of all simple Dunkl-Opdam modules.

\subsection{Scaffolding and indexing sets}
Let us now turn to establishing a scaffolding for the category of Dunkl-Opdam modules.  For simplicity, we assume that $\Upsilon(k)\geq 0$.  First, note that the inclusion of the subalgebra $\mathcal{D\!O}_n\subset \mathsf{H}_1$ is compatible with the automorphism $\auto$. Thus we have an action of $C_p$ on $\mathcal{D\!O}_n$; this is obviously a Morita action, since the powers of $t_1$ give units in each isotypic component.

Thus, we first describe a scaffolding for the representations of $\mathcal{D\!O}_n$.  As usual, for some $g$, and a lifted segment $(a_1,\dots, a_g)$ and choice of $z\in \mu_{\ell}(\K)$, there is a 1-dimensional representation $L(\Ba,z)$ over $\mathcal{D\!O}_g$ where $S_g$ acts trivially, $t_i$ acts by $z$ for all $i$, and $u_i$ by $a_i$. \begin{lemma}
 There is a bijection between lifted multisegments and simple $\mathcal{D\!O}_n$-modules.   The simple $\mathcal{L}(\bQ)$ is the unique simple quotient of the induction $I(\bQ)$ of $L(\Ba^{(1)},z^{(1)})\boxtimes \cdots \boxtimes L(\Ba^{(m)},z^{(m)})$ from a parabolic subalgebra $\mathcal{D\!O}_{g_1}\otimes \cdots \otimes \mathcal{D\!O}_{g_m}\subset \mathcal{D\!O}_n$ where the lifted segments $\Ba,z$ are ordered so that $\Upsilon(a_1^{(i)})\geq \Upsilon(a_1^{(i+1)})$. 
\end{lemma} 
Since this will be important for us later, we note that this order is not unique, since we might have $\Upsilon(a_1^{(i)})= \Upsilon(a_1^{(i+1)})$, but  $\Ba^{(i)}\neq \Ba^{(i+1)}$ or $z^{(i)}\neq z^{(i+1)}$.  For purposes of scaffolding, we wish to have a canonical order, so we choose a total order on lifted segments compatible with the partial ordering by $\Upsilon$ of the first element, and order elements $\zeta'$ of $\mu_\ell(\K)$ by the smallest non-negative integer $k$ such that $\zeta'=\zetal^k$.  This induces a total order on pairs where we use these orders lexicographically, first on $\Ba$ and then on $z$.    
\begin{definition}\label{def:azQ}
Let $\Ba_{\bQ}$ and $\Bz_{\bQ}$ be the concatenation of the liftings of the multisegments $\bQ$ (with $z^{(i)}$ appearing with multiplicity $|\Ba^{(i)}|)$. This is the weight of $u_i$ and $t_i$ acting on the generator of $I(\bQ)$.
\end{definition}

If we twist by the automorphism $\auto$, then $L(\Ba^{(1)},z^{(1)})\boxtimes \cdots \boxtimes L(\Ba^{(m)},z^{(m)})$ is sent to $L(\Ba^{(1)},\zetap^{-1}z^{(1)})\boxtimes \cdots \boxtimes L(\Ba^{(m)},\zetap^{-1}z^{(m)})$.  These lifted segments are no longer ordered correctly for our total order: the segments $\Ba$ have stayed in the correct order, but if we have $\Ba^{(i)}=\Ba^{(i+1)}$ but $z^{(i)}\neq z^{(i+1)}$, these might be potentially out of order.  However, we can apply the unique shortest permutation in $\sigma_{\bQ}\in S_n$ that puts these back into order.  
\begin{lemma}
Right multiplication by $\sigma_{\bQ}$ induces an isomorphism $q_{\bQ}\colon L(\bQ)^\auto\to L(\auto\cdot \bQ)$, which gives a scaffolding.
\end{lemma}
\begin{proof}
To check that right multiplication by $\sigma_{\bQ}$ induces a map $I(\bQ)^\auto\to I(\auto\cdot \bQ)$, we simply use many times that if $z_i\neq z_{i+1}$, then $r_i=(i,i+1)$ induces a natural transformation $\Wei_{(\Ba,\Bz)}\to \Wei_{(\Ba^{r_i},\Bz^{r_i})}$ which is obviously invertible.  Since $\sigma_{\bQ}$ leaves all blocks with the same value of $z$ in the same order, every simple reflection in a word for $\sigma_{\bQ}$ acts on the generating copy of $L(\Ba^{(1)},z^{(1)})\boxtimes \cdots \boxtimes L(\Ba^{(m)},z^{(m)})\subset I(\bQ)^\auto$ sending it again to a weight vector of the correct weight.  Similarly, the generators of the Young subgroup are conjugated by $\sigma_{\bQ}$ to the generators of the Young subgroup for $\auto\cdot \bQ$.  This shows we have an isomorphism.  To show it is a scaffolding,
we need only check that $q_{\auto^{p-1}\bQ}\cdots q_{\bQ}=1$, and this is clear since the permutation $\sigma_{\bQ}$ satisfies the analogous equation $\sigma_{\auto^{p-1}\bQ}\cdots \sigma_{\bQ}=1$.
\end{proof}

In particular, if there is a permutation $\sigma$ such that $\Ba^{(\sigma(i))}=\Ba^{i}$ and $z^{(\sigma(i))}=\zetap^{-k}z^{(i)}$, then $L(\Ba^{(*)},z^{(*)})^{\alpha^k}\cong L(\Ba^{(*)},\zetap^{-k}z^{(*)})\cong L(\Ba^{(*)},z^{(*)})$, and our scaffolding gives a preferred isomorphism.  Since our action is Morita, the eigenspaces of this isomorphism give the different irreps of $\mathcal{D\!O}_n^{C_p}$.

It's worth thinking about this isomorphism in terms of the algebra $R_D$;   when a permutation $\sigma$ as above exists, the idempotents $e(\Ba_{\bQ},\Bz_{\bQ})$ and  $e(\Ba_{\bQ},\zetap^{-k}\Bz_{\bQ})$ are not the same, but the straight-line diagram $\delta_\sigma$ tracing out the permutation $\sigma$ gives an isomorphism between these idempotents.  Thus, right multiplication by $r\mapsto \auto^k(r\delta_{\sigma})$ gives the desired $R_D\# C_\chi$-module structure on  $R_De(\Ba_{\bQ},\Bz_{\bQ})$.  

It's easy to upgrade this scaffolding to one for Dunkl-Opdam modules over $\mathsf{H}_p$.  The simple module $K_{\chi}$ has an attached lifted multisegment $\bQ'$ obtained by adding all the rows of $\xi$ (correctly charged) to the lifts of the multisegment $\bQ$.   Applying the scaffolding to $K_{\chi}=L(\bQ')$, we obtain an isomorphism $q_{\bQ,\xi}\colon K_{\chi}^\auto\to K_{\auto\cdot \chi}$ where the action of $\auto$ on multisegments and multipartitions is by sending 
\begin{align}
    \label{eq:segment-action}\auto\cdot (q_1,\dots, q_g)&=\big(q_1-\frac{1}{p},\dots, q_g-\frac{1}{p}\big)\\ \auto\cdot (\xi^{(1)},\dots, ,\xi^{(\ell)})&=(\xi^{(\lp+1)},\dots, ,\xi^{(\ell)},\xi^{(1)},\dots, \xi^{(\lp)})\label{eq:YD-action}
\end{align}
It may look a little strange to have such different formulas for the action on the multisegment and the multipartition, but this is because of the how the eigenvalues of the $t_i$'s are organized in the corresponding $\mathcal{D\!O}_n$ representation.  For a multisegment, the choice of this eigenvalue $z_i$ is left implicit, and in fact depends on the choice of $N$ used to construct the corresponding lifted multisegment.  Since $\Sigma(a, \zetap^{-1}z)=\Sigma(a, z)-\frac{1}{p}$, subtracting $\frac{1}{p}$ exactly has the effect of changing the eigenvalue as desired.  On the other hand, the $\ell$ components of the multipartition $\xi$ correspond to the different roots of unity that $t_i$ can act by, and thus the rotation above is the correct operator to change this eigenvalue by $\zetap^{-1}$.

This scaffolding extends to the modules $M(\chi)$  and $\Delta(\chi)$ and thus a scaffolding morphism $S(\chi)^{\auto}\to S(\auto\cdot \chi)$ for simple Dunkl-Opdam modules over $\mathsf{H}_p$.  This is induced by considering the corresponding $R_D$-modules as a quotients of $R_De(\Ba_{\chi},\Bz_{\chi})$ where $\Ba_{\chi}=\Ba_{\bQ'}$, and $\Bz_{\chi}=\Bz_{\bQ'}$.  All of these quotient maps are compatible with the $R_D\# C_{\chi}$-module structure already described.  

Thus, applying Theorem \ref{thm:clifford}, we can classify all Dunkl-Opdam modules over $\mathsf{H}_p$, which strengthens a result of Griffeth on category $\mathcal{O}$ \cite[Th. 9.1]{GrifJack}:
\begin{theorem}\label{thm:DO-classification}
For each orbit $[\chi]$ of the action of $C_p$ on  $\pairs_n$, the $p/p_\chi$ different eigenspaces of $Q_\chi$ acting on acting on $S(\chi)$ are all (non-zero) simple $\mathsf{H}_p$-modules, and every simple Dunkl-Opdam $\mathsf{H}_p$-module is of this form for a unique $\chi$ up to the action of $C_p$ and unique $p/p_\chi$th root of unity.  
\end{theorem}

``Typically'' the action of $C_p$ will be free (for example, whenever $p$ is coprime to $n$), in which case we just have that every simple Dunkl-Opdam $\mathsf{H}_p$-module is the restriction of a simple $\mathsf{H}_1$-module, with restriction being a $p$ to $1$ map.

On the other hand, this is certainly not always the case.  Perhaps the case of greatest external interest would be that of $G(2,2,n)$, the Weyl group of type $D_n$.  In this case, we have a multisegment $\bQ$ in $\K/\Z$ and a bipartition $\xi=(\xi^{(1)},\xi^{(2)})$.  
\begin{theorem}
A simple Dunkl-Opdam module over the Cherednik algebra of $G(2,1,n)$ remains simple after restriction to that for $G(2,2,n)$ unless $\xi_1=\xi_2$ and $\bQ=\alpha\cdot \bQ=\bQ-\nicefrac{1}{2}$ (by the formula \eqref{eq:segment-action}), in which case it is the sum of two distinct simple modules (which appear in the restriction of no other simple).
\end{theorem}
So, for example, for $G(2,2,2)$, the Dunkl-Opdam modules come in the following families:
\begin{itemize}
    \item There are 2 simples in category $\cO$ which are the restriction of the simples in category $\cO$ for $G(2,1,2)$ corresponding to $\xi=((2),\emptyset)$ and $\xi=((1,1),\emptyset)$.
    \item There are 2 simples in category $\cO$ which are the distinct summands of the restriction of the simple module for $\xi=((1),(1))$ in category $\cO$ for $G(2,1,2)$.
    \item For each $a\in \K/\Z$, there is a simple corresponding to the multisegment $(a)$ and multipartition $((1),\emptyset)$; this is the restriction of a simple Dunkl-Opdam module for $G(2,1,2)$.
    \item There is a simple for each multisegment of the form $(a,a+k)$ with $a\in \K/\Z$; this is always the restriction of a simple for $G(2,1,2)$ and the only other simple isomorphic after restriction is that for the segment $(a+\nicefrac{1}{2},a+\nicefrac{1}{2}+k)$.
    \item There is a simple for each multisegment given by the segments $(a)$ and $(b)$ if $a\neq b+\nicefrac{1}{2}$;  the multisegment consisting of $(a+\nicefrac{1}{2})$ and $(b+\nicefrac{1}{2})$ gives the same simple.
    \item There are 2 simples which are the distinct summands of the restriction of the simple module for $(a)$ and $(a+\nicefrac{1}{2})$ for each $a\in \K/\Z$. 
\end{itemize}

One of the interesting consequences of this description is that it is compatible with grading.  The algebras $\tilde{R}_D$ and $R_D^{C_p}$ are naturally graded, and the graded modules over these algebras give a graded lift of the category of Dunkl-Opdam modules.  Thus, one can ask questions about graded dimensions of weight spaces and various multiplicities.  

This grading satisfies a very non-trivial property: it is Morita equivalent to a graded algebra with non-negative grading and semi-simple degree 0 part. 
\begin{proposition}
The algebra $R_D$ is $\auto$-equivariantly graded Morita equivalent to a mixed algebra $R_D'$.
\end{proposition}
\begin{proof}
This is true for $R_D$ by \cite[Cor. 4.7]{WebwKLR}.  To see the compatibility with $C_p$, let us recall the construction of this equivalence in more detail.  For each  $\chi\in \pairs_n$, we have the weight $(\Ba_{\chi},\Bz_{\chi})$ defined above.     By construction, $\Wei_{(\Ba_{\chi},\Bz_{\chi})}(S(\chi))\cong \C$; thus $R_D$ contains a homogeneous primitive idempotent $e_{\chi}$ which acts non-trivially on $\Wei_{(\Ba_{\chi},\Bz_{\chi})}(S(\chi))$.  
The module $R_De_{\chi}$ is thus an indecomposable projective $R_D$ module with head given by $S(\chi)$. Summing over $\pairs_n$, we obtain a projective submodule  $P\subset R_D$; this has a compatible action of $\auto$ arising from the isomorphism $\Wei_{(\Ba_{\chi},\Bz_{\chi}^\circ)}\cong \Wei_{(\Ba_{\auto\chi},\Bz_{\auto\chi})}$.

Every homogenous simple appears as a quotient of $P$, so $P$ is a projective generator and induces a graded Morita equivalence between $R_D$ and $R_D'=\End(P,P)$ compatible with the action of $\auto$ on both rings. 

By \cite[Cor. 4.7]{WebwKLR}, we can identify $e_{\chi}R_De_{\chi'}$ with the Ext-space between two simple perverse sheaves which are only isomorphic if $\chi=\chi'$.  Thus the grading on $e_{\chi}R_De_{\chi'}$ is strictly positive unless $\chi=\chi'$, in which case it is 1-dimensional in degree 0 and non-negative otherwise  (see \cite[Lem. 1.18]{WebCB} and \cite[Thm. 5.25]{WebRou} for more discussion of this point).  
\end{proof}
\begin{corollary}\label{cor:mixed}
The algebra $\tilde{R}_D$ is Morita equivalent to the mixed algebra $\tilde{R}_D'$.  
\end{corollary}
\begin{remark}\label{rem:characters}
We should note that this mixed property is key for the calculation of the classes in the Grothendieck group of  simple modules following the approach of \cite{WebCB}, ultimately tracing back to the algorithm described by Leclerc in \cite[\S 5.5]{Lecshuf}.  

Let us first describe in the case of $p=1$.  The dimensions of the spaces $e(\Ba,\Bz)S$ play the role of a shuffle expansion of a dual canonical basis vector, and we can compute them using the fact that the graded dimensions are bar-invariant (i.e. they are palindromic as Laurent polynomials in $q$), analogous to the relationship between $q$-characters and shuffles from 
\cite[Thm. 4.4]{KlRa}.

By the mixed property, when the standard modules $\Delta(\chi)$ are given the grading where the simple quotient is self-dual, all other composition factors are positive shifts of simples where  $|\bQ'|\leq |\bQ|$, and if $|\bQ'|= |\bQ|$ then $c_{\xi'}<c_{\xi}$.  Assuming that we can compute $e(\Ba,\Bz)\Delta(\chi)$, we can start with a minimal case (where $\bQ=\emptyset$ and $\xi$ is minimal in the $c$-function order; this corresponds to a simple Verma in category $\cO$), and then inductively compute the unique bar-invariant shuffle expansion which can be obtained from $e(\Ba,\Bz)\Delta(\chi)$ by subtracting positive shifts of the expansions we already know.  Note that we only need to compute these for the idempotents $e(\Ba_{\chi},\Bz_{\chi})$ for $\chi\in \pairs_n$.    

The same approach proceeds for general $p$, except that we need to use some additional structure: $\Delta(\chi)$ is a module over $R_D\smash C_{\chi}$, and the inductions  $\Delta(\chi,\la)$ to $\tilde{\mathsf{H}}_1$  of the different twists of the $C_{\chi}$ action are irreducible.  These have not just a $\Z[q,q^{-1}]$-valued shuffle expansion, but one where $e(\Ba_{\chi'},\Bz_{\chi'})\Delta(\chi,\la)$ carries an action of $C_{\chi'}$ by $m\mapsto \auto^k(\delta m)$ where as before, $\delta$ is an invertible straight-line diagram whose bottom matches $e(\Ba_{\chi'},\Bz_{\chi'})$ and whose top matches $e(\Ba_{\auto^{-k}\chi'},\Bz_{\auto^{-k}\chi'})$, for $\auto^k\in C_{\chi'}$.  Thus, we have apply our algorithm for finding bar-invariant characters to shuffle expansions of standards valued in the characters of the groups $C_{\chi'}$.

Finding the characters of the standards is a combinatorial exercise based on the basis of \cite[Thm. 2.8]{WebwKLR} and the cellular basis of standard modules over the steadied quotient from \cite[Thm. 4.11]{WebRou}.  This assumes that we have already calculated the $q$-characters of simple modules for the degenerate affine Hecke algebra as in \cite[\S 6]{Lecshuf}, but we can instead replace $\Delta(\chi)$ with modules where we use standard modules for the degenerate affine Hecke algebra instead of simples, which it is simpler to calculate the character of.  We leave the details of this calculation to the reader (or perhaps the authors' future selves).
\end{remark}

\subsection{Category \texorpdfstring{$\cO$}{O}}

In \cite[Th. 9.1]{GrifJack}, Griffeth classifies the simple objects of category $\cO$.  In our schema above, these are the simple Dunkl-Opdam modules with $\bQ=\emptyset$, that is, corresponding to a orbit of the action of $C_p$ on the space of $\ell$-multipartitions, and a character of the stabilizer (that is, a $p/p_{\xi}$th root of unity).

However, our understanding of category $\cO$ for $\mathsf{H}_1$, which we denote $\cO_1$, allows us to understand the category $\cO$ for $\mathsf{H}_p$, which we denote $\mathcal{O}_p$.  In particular, recall that by \cite[Thm. 3.15]{Webalt}, the functor $\Wei_D$ defines an equivalence between $\cO_1$, and the category of modules over a steadied quotient $R_D(+)$ of the ring $R_D$. 
\begin{proposition}
The action of $\auto$ descends to $R_D(+)$, and there is an induced equivalence of categories $\Wei\circ \Ind^{\mathsf{H}_1}_{\mathsf{H}_p}\colon \mathcal{O}_p\to \widetilde{R_D(+)}$.  
\end{proposition}
\begin{proof}
First, note that a $\mathsf{H}_p$-module $M$ lies in $\mathcal{O}_p$ if and only if its induction $\Ind_{\mathsf{H}_p}^{\mathsf{H}_1}M$ lies in $\mathcal{O}_p$.  Thus, by \cite[Thm. 3.15]{Webalt}, $M$ lies in $\mathcal{O}_p$ if and only if the action of $R_D$ on $\Wei\circ \Ind^{\mathsf{H}_1}_{\mathsf{H}_p}(M)$ factors through $R_D(+)$.  This is equivalent to the action of $R_D\# C_p$ on $\Wei\circ \Ind^{\mathsf{H}_1}_{\mathsf{H}_p}(M)$ factoring through $\widetilde{R_D(+)}$.
\end{proof}\begin{corollary}\label{cor:smash-Koszul}
The algebra $\widetilde{R_D(+)}$, and thus a graded lift of category $\cO$ for $\mathsf{H}_p$, is mixed (up to Morita equivalence) and standard Koszul.   
\end{corollary}
This result was proven for dihedral groups in \cite{RolloJenkins}.
\begin{proof}
By \cite[Th. C]{WebRou}, the algebra $R_D(+)$ is standard Koszul, and after a Morita equivalence is mixed.  Thus by Theorem \ref{thm:standard-Koszul}, $\widetilde{R_D(+)}$ is standard Koszul as well.  
\end{proof}

Note that the Koszul dual of $R_D(+)$ is again (up to Morita equivalence) an algebra of the same type, giving a block of category $\cO$ for the Cherednik algebra of $G(\ell',1,n')$ for quite different data related by rank-level duality (see \cite{CM, WebRou} for more discussion of this combinatorics). In particular, $\ell'$ is the denominator of $k$ in least terms (and $n'$ is quite complicated), while this dual block has $k'=1/\ell$ (or some other fraction with $\ell$ as denominator; there are many equivalent blocks which all give the same category).
At the moment, the natural of this dual action is unclear to us, though it obviously must be quite different from the one that gives $G(\ell,p,n)$.  In particular, $p$ will typically not divide $\ell'$, and the $p$-cyclic condition corresponds to a restriction on the block considered (but not the possible parameters).  

\begin{remark}
As in Remark \ref{rem:characters}, we can use the mixed property to compute decomposition numbers.  We have not been able to find a slick packaging of this conjecture in terms of Fock spaces, as Rouquier did in the $G(\ell,1,n)$ case \cite[\S 6.5]{RouqSchur}.  The mixed property gives the suitable raw material for such a proof as in \cite[Cor. 5.27]{WebRou}, but we can have not found a convenient description of the Grothendieck group.  
\end{remark}

\subsection{The Knizhnik-Zamolodchikov functor}
Just as in \cite[Thm. 3.18]{Webalt}, we can describe the KZ functor on category $\cO$ in terms of this equivalence.  Let $\mathcal{H}_p$ be the Hecke algebra of $G(\ell,p,n)$ with appropriate parameters;  note that $\mathcal{H}_p=\mathcal{H}_1^C$ for a Morita action.  In \cite{LiuClifford}, it's shown that this functor commutes with restriction; this is equivalent to saying that the projective in $\mathcal{O}_p$ representing the functor $\KZ$ satisfies $\Ind_{\mathsf{H}_p}^{\mathsf{H}_1}P_p=P_1$.  This can be rephrased as saying that $P_1$ carries an isomorphism $P_1\cong P_1^\auto$ which makes it into a $\tilde{\mathsf{H}}_1$-module, and $P_p$ is the invariants of this action.  

We will recall the description of $\KZ\colon \cO_1\to \mathcal{H}_1\mmod$ and make a small modification in order to make it compatible with the $C_p$ action.   Let $D^\circ$ be the orbit space for  $\frac{1}{\ell}\Z/\Z$ on $D$.  For each $[d]\in D^\circ$, we choose an element $\varphi([d])\in \K$ such that the orbit is given by the set 
$\Sigma(\varphi(d),\zetal^k)$ for $k\in [1,\ell]$.  Note that we have a small difference with the construction in \cite{Webalt}, where we chose a splitting of $\Sigma$ at always had 1 in the second coordinate. This will not effect the construction at all; in fact, we could choose {\it any} fixed splitting of $\Sigma$, so it is more convenient to choose one which is $C_p$-equivariant.

Choose an integer \[N\gg \max_{\substack{i\in
  [1,\ell]\\ [d]\in D^\circ}}(|\Upsilon(p(\zeta^i)|,|\Upsilon(k)|,|\Upsilon(\varphi([d]))|).\]  For each $n$-tuple
$\mathbf{d}=([d_1],\dots, [d_n])\in (D^\circ)^n$, let
\[\mathbf{a}_{\mathbf{d}}^\pm=(\varphi([d_1])\mp N,\varphi([d_2])\mp2N,\dots,
\varphi([d_n])\mp nN)\qquad \mathbf{1}=(1,\dots, 1).\] 
Then we have an isomorphism of functors \[\KZ\cong \bigoplus_{\substack{\mathbf{d}\in (D^\circ)^n\\ k=1,\dots, \ell}}
  \Wei_{\mathbf{a}_{\mathbf{d}}^-,\zetal^k\mathbf{1}}.\] The action of $R_D$ gives an action of the usual KLR algebra $\mathsf{R}_n$ on the RHS, and this intertwines with the $\mathcal{H}$-action on $\KZ$ under the Brundan-Kleshchev isomorphism \cite{BKKL}.

Note that while the individual terms of the right hand side are not well-defined on $\mathsf{H}_p$-modules, the terms $W_{\mathbf{a}_{\mathbf{d}}^-}=\bigoplus_{k=1}^\ell
  \Wei_{\mathbf{a}_{\mathbf{d}}^-,\zetal^k\mathbf{1}}$ for a fixed $\mathbf{d}$ only depend on the action of elements of $\mathsf{H}_p$, since this is a simultaneous generalized eigenspace for $u_i $ and $T_iT_j^{-1}$ for all pairs $i,j$. Thus, these define exact functors $W_{\mathbf{a}_{\mathbf{d}}^-}\colon \mathcal{O}_p\to \Vect.$
  \begin{theorem}
  The functor $\bigoplus_{\mathbf{d}\in (D^\circ)^n}
  W_{\mathbf{a}_{\mathbf{d}}^-}$ is isomorphic to $\KZ_p\colon  \mathcal{O}_p\to \Vect.$ This functor has an induced action of  $\mathsf{R}_n^C$, which matches the Hecke algebra of $G(\ell,p,n)$ under the isomorphism of Rostam \cite{Rostam}. 
  \end{theorem}
  \begin{proof}
  Let $\KZ_p'=\bigoplus_{\mathbf{d}\in (D^\circ)^n}
  W_{\mathbf{a}_{\mathbf{d}}^-}$. Since this is an exact functor, we only need to establish that it has the correct behavior on each simple module. Let $S$ be a simple in $\mathcal{O}_p$ and  $S'$ a simple in $\mathcal{O}_1$ with $S$ a summand of $\Res^{\mathsf{H}_1}_{\mathsf{H}_p}(S')$.  Note that since $\mathcal{H}_p=\mathcal{H}_1^C$ for a Morita action, $\KZ(S')$ is either 0, or a sum of $p/p_S$ summands of equal dimension.  
  
  Thus, if $\KZ_p(S)=0$, then all other summands of $\Res^{\mathsf{H}_1}_{\mathsf{H}_p}(S')$ are also killed by $\KZ_p$, so $\KZ(S)=0$.  Since  $\KZ_p'\circ \Res^{\mathsf{H}_1}_{\mathsf{H}_p}\cong \KZ$, the functor  $\KZ_p'$ kills $\Res^{\mathsf{H}_1}_{\mathsf{H}_p}(S')$, and thus $S'$ by exactness.
  
  On the other hand, if $\KZ_p(S)\neq 0$, then it is one of the  $p/p_S$ summands of $\KZ(S')$ equal dimension.  Since the $\auto$ action on $\KZ_p(S')$ permutes different weight spaces, it is free, and so it divides into $p/p_S$ summands of equal size under $C_S$, which the images of the different summands of $\Res^{\mathsf{H}_1}_{\mathsf{H}_p}(S')$.  This establishes that the number of copies of the projective cover of $S'$ in the representing projective for $\KZ_p$ and $\KZ_p'$ are the same.  The compatibility with Rostam's isomorphism follows because it is the restriction of Brundan-Kleshchev's to the $C$-fixed points, and the compatibility with restriction from \cite[Thm 5.2]{LiuClifford}.
  \end{proof}
  As with the construction in \cite{Webalt}, this construction has the distinct advantage of being non-transcendental, and thus defined over an arbitrary field of characteristic $0$.

\section{Relationship to quantum field theory}
In recent years, ideas from quantum field theory have made a key contribution to noncommutative algebra and symplectic algebraic geometry.  Before getting into details, let us discuss a bit the theories concerned.

\subsection{3d \texorpdfstring{$\cN=4$}{N=4} Gauge Theories}
A 3d $\cN=4$ gauge theory is defined by a compact Lie group $G_{\mathbb{R}}$, a quaternionic representation $\mathcal{R}$ of $G$, and a set of mass and FI parameters. We note some differences of notation between the mathematical and physical work on this subject.  In the mathematical literature, we typically show a preference for the complex structure given by the obvious embedding $\C\subset \mathbb{H}$, making $\mathcal{R}$ into a complex vector space, with a module structure over the complexification $G$ of $G_{\mathbb{R}}$.  Of course, we can choose a $G$-invariant inner product $\langle -, -\rangle$ on $\mathcal{R}$ which is quaternionic Hermitian  and $\Omega(r_1,r_2)=\langle r_1,Jr_2\rangle - i \langle r_1,Kr_2\rangle$ is a $\C$-linear symplectic structure on $\mathcal{R}$.  Thus, it is essentially equivalent to study a complex representation $\mathcal{R}$ of a complex reductive group $G$ with an invariant symplectic form.  We will also typically want to assume that $\mathcal{R}=N\oplus N^*$ for a $G$ representation $N$.  

This gives us a 3d theory whose field content consists of a vectormultiplet transforming in the adjoint representation of $G$ and hypermultiplets transforming in the representation $\mathcal{R}$. The vectormultiplet contains three real scalar fields, and the hypermultiplets each contain four real scalar fields. The real scalar fields $\phi_i$ (for $i=1,2,3$) in the vectormultiplet are each valued in the Lie algebra $\mathfrak{g}$ of $G$. The real scalars in the hypermultiplets parameterize a copy of $\mathcal{R}$.

Attached to this theory, we have a {\bf moduli space of vacua.}
Since Lorentz invariance prevents all fields except for massless scalar fields from acquiring vacuum expectation values, the moduli space of vacua is parameterized by the massless scalar fields of the theory.

The moduli space of vacua is union of components, or in the parlance of quantum field theory, ``branches.''  The two most important of these are (1) the Higgs branch $\cM_H$ is parameterized by nonzero vacuum expectation values of the hypermultiplet scalars, and (2) the Coulomb branch $\cM_C$ is parameterized by nonzero vacuum expectation values of the vectormultiplet scalars and the dual photon.

The Higgs branches of these theories can be understood as hyperk\"ahler quotients. In particular, the connection to Cherednik algebras was established much earlier in \cite{Gordonquiver}. Until recently, it was very challenging to work explicitly with Coulomb branches of these theories, since no mathematically precise description of them existed.  The naive expectation for the Coulomb branch is
\begin{equation}
	\cM_C \approx T^* \mathbb{T}^\vee / W
\end{equation}
where $\mathbb{T}^\vee$ is the abstract dual Cartan and $W$ the Weyl group of $G$. Quotienting by the Weyl group accounts for the symmetries of the maximal toral subalgebra. The above equation turns out not to be entirely correct (indicated by the $\approx$ sign) and is referred to as the classical Coulomb branch because the geometry of Coulomb branch receives quantum corrections.  

This issue was corrected in recent work of Braverman-Finkelberg-Nakajima \cite{BFN}, using the geometry of the affine Grassmannian.  As discussed in \cite[\S 1.4]{DGGH}, this can be interpreted physically as computing local operators in the A-twist of this theory at a junction of two trivial line defects.  By incorporating an $\Omega$-background, BFN also construct a non-commutative algebra quantizing the holomorphic functions (i.e. the chiral ring) on the Coulomb branch; we will call this the {\bf quantum Coulomb branch}.

This work has proven highly influential in giving new presentations of already understood algebras as quantum Coulomb branches such as $U(\mathfrak{gl}_n)$ and several of its variants such at finite W-algebras.  The most important example for us is the spherical Cherednik algebra of $G(\ell,1, n)$, as shown by Kodera-Nakajima \cite{KoNa} and expanded on by Braverman-Etingof-Finkelberg \cite{BEF} and the second author \cite{Webalt}.  




\begin{theorem}
The spherical Cherednik algebra $e_1\mathsf{H}_1 e_1$ is isomorphic to the BFN Coulomb branch of the 3d $\cN=4$ gauge theory with matter $N=\mathfrak{gl}_n\oplus \C^{n}\otimes \C^{\ell}$ (with trivial action on $\C^\ell$) and (complexified) gauge group $G=GL_n$.
\end{theorem}
This is a quiver gauge theory with quiver diagram:
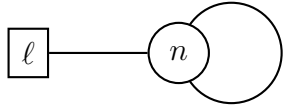
\begin{equation*}
\tikz{
    \draw[thick] (.7,0) circle (19pt);
	\node[draw, thick, circle, inner sep=5pt,fill=white] (a) at (0,0) {$n$};
	\node[draw, thick,inner sep=5pt] (b) at (-2,0){$\ell$};
		\draw[thick] (a) -- (b) ;
}
\end{equation*}

\subsection{Vortex line operators}

We wish to enrich this picture by considering line operators in the A-twist of our gauge theory.  This category should be interpreted mathematically as the category of D-modules on the stack quotient $N((t))/G((t))$ which is one version of the loop space of the quotient $N/G$.  The easiest way to construct such a D-module is to consider a subspace $L\subset V((t))$ invariant under a subgroup $\mathcal{G}_0\subset G((t))$, and push forward the functions of $L$ by the map $L/\mathcal{G}_0\to N((t))/G((t))$.  In the notation of \cite{DGGH}, we take the Lagrangian $\mathcal{L}_0$ to be the conormal of $L$.  

It will be convenient for us to use the identification $(\C^n\otimes \C^\ell)((t))\cong \C^n((t))$; if $e_0,\dots, e_{\ell-1}$ are the standard basis vectors of $\C^\ell$, this is induced by the isomorphism 
\begin{equation}
\C^\ell((t))\cong \C((t))\qquad  t^me_i\mapsto t^{m\ell+i}\label{eq:l-iso}
\end{equation}
Let $\tl\colon \C^\ell((t))\to \C^\ell((t))$ be the map intertwined with multiplication by $t$ with this isomorphism; we'll also use this for the induced automorphism of $(\C^n\otimes \C^\ell)((t))$.  This corresponds to multiplication of elements of $\C^\ell((t))$ by the matrix
\[ \tl=\begin{bmatrix}
0 & 0 & 0&\cdots & 0& 0 & t\\
1 & 0 & 0&\cdots & 0& 0 & 0\\
0 & 1 & 0&\cdots & 0& 0 & 0\\
\vdots & \vdots& \vdots & \ddots &\vdots &\vdots&\vdots\\
0 & 0 & 0&\cdots & 0& 0 & 0\\
0 & 0 & 0&\cdots & 1& 0 & 0\\
0 & 0 & 0&\cdots & 0& 1 & 0
\end{bmatrix}. \]
We let $\tn$ be the $n\times n$ version of this matrix.  Note that $\tn$ generates a copy of $\Z$ in $G((t))$ which normalizes the Iwahori; in fact, this can be identified with the length 0 elements of the extended affine Weyl group.  

Let us describe the cast of operators of interest to us. Let $I$ be the standard Iwahori given by matrices in $G[[t]]$ which are upper-triangular mod $t$, and $\mathfrak{I}$ its Lie algebra as a subspace of $\mathfrak{g}[[t]]\subset N[[t]]$.   Let $u_k=\tl^{k\ell/p}(\C^\ell[[t]])$ and $U_k=\C^n\otimes u_k$.
\begin{itemize}
    \item If $L=N[[t]]$ and $\mathcal{G}_0=G[[t]]$, then the result is the trivial line defect $\mathbbm{1}$.  The endomorphisms of this line defect are precisely the BFN Coulomb branch.
    \item If $L=\mathfrak{I}\oplus U_k$ and $\mathcal{G}_0=I$, the result is the {\bf Procesi lines} $\mathbbm{p}_k$.  We call these ``Procesi lines'' since they behave quite similarly to Procesi bundles on Hilbert schemes.
    \item If $L=\mathfrak{gl}_n[[t]]\oplus U_k$ and $\mathcal{G}_0=G[[t]]$, the result is the {\bf spherical Procesi lines} $\mathbbm{sp}_k$.  
    \item If $L=\{0\}$ and $\mathcal{G}_0=G((t))$, then this is the {\bf GKLO line} $\mathbbm{o}$.  This name arises from its relation to the so-called GKLO representations (originally defined for Yangians in \cite{GKLO}).
\end{itemize}
Local operators at junctions of these defects can be computed directly, using the geometry of fiber products; see \cite[\S 5]{DGGH} for a discussion from a physical perspective.  For mathematicians, it is hopefully satisfactory to blindly apply the theorem that if $X_1 \to Y \leftarrow X_2$ is a pair of maps, and $M_1$ and $M_2$ the pushforward of the function D-modules, then
\[\Ext^{\bullet}(M_1,M_2)\cong H^{BM}_{\bullet}(X_1\times_{Y} X_2).\] That is, for $(L_1,\mathcal{G}_1)$ and $(L_2,\mathcal{G}_2)$ defining vortex defects $\mathbb{L}_1,\mathbb{L}_2$, then 
\begin{multline}\label{eq:line-ext}
    \Ext^{\bullet}(\mathbb{L}_1,\mathbb{L}_2) = H^{BM}_{\bullet}\Big(\frac{L_1}{\mathcal{G}_1}\times_{\frac{N((t))}{G((t))}}\frac{L_2}{\mathcal{G}_2}\Big)\\
    =H^{BM, \mathcal{G}_1}_{\bullet}\Big(\big\{(g\mathcal{G}_2,v(t))\in \frac{G((t))}{\mathcal{G}_2}\times L_1\mid g^{-1}v(t)\in L_2\big\}\Big).
\end{multline}
Adding an $\Omega$-background deforms this algebra by adding in equivariance for a circle action; we denote the resulting Ext space by $\Ext^{\bullet}_{\hbar}(\mathbb{L}_1,\mathbb{L}_2)$, where $\hbar$ is the equivariant parameter for the circle action.  The result depends on the choice of mass parameters.  These come from an element of the flavor Lie algebra $\mathfrak{gl}_{\ell}\times \C=\End_{GL_n}(N)$, which encodes how the circle acts on the matter $N$.


We typically interpret $N((t))/G((t))$ as the space of trivializable principal $G$-bundles on a formal punctured disk with a choice of section of the associated $N$-bundle, and $N[[t]]/G[[t]]$ as the space of such bundles with section on an unpunctured disk.  
In the case where $I=\mathcal{G}_1=\mathcal{G}_2$, this means that $L_1/I$ is the space of principal $G$-bundles on the formal disk with a $B$-reduction at the origin, with a meromorphic section whose behavior near the origin matches an element of $L_1$.  Taking fiber product can be interpreted as considering the space of principal bundles on the ``raviolo'' obtained by gluing two formal disks away from the origins, together a reduction of structure group to $B$ at each origin, with a meromorphic section of the associated $N$-bundle whose singularities at the 2 origins are controlled by $L_1$ and $L_2$ respectively.

Alternatively, $G((t))/I$ can be interpreted as the space $\mathcal{F}\ell$ of affine flags:
\[\mathcal{F}\ell=\{\cdots \subset \Lambda_i \subset \Lambda_{i-1} \subset \cdots \subset \C^n((t)) \mid \Lambda_i\text{ a lattice, } \Lambda_{i-1} /\Lambda_i=\C, \Lambda_{i+n}=t\Lambda_i.\}\] by acting with $G((t))$ on the standard affine flag $\Delta_\bullet$ with $\Delta_0=\C^n[[t]]$ and $\Delta_i=\tn^i\Delta_0$  the unique $I$-invariant lattice containing or contained in $\Delta_0$ with the correct codimension.  
      Using this perspective, we can see that:
    \begin{multline}
  \label{eq:1}
  \Ext_{\hbar}^{\bullet}(\mathbbm{p}_k, \mathbbm{p}_m)\\=H_{\bullet}^{BM,I\ltimes \C^*}(\{(\Lambda_\bullet, x(t), y(t))\in \mathcal{F}\ell\times \mathfrak{I}\times U_k\mid x(t)\cdot \Lambda_{i}\subset \Lambda_{i+1}, y(t)\in \Lambda_ 0\otimes u_m\})
\end{multline}

One very useful construction is the GKLO representation, the functor from line operators to vector spaces which sends $\mathbb{L}$ to 
\[\Ext^\bullet(\mathbbm{o},\mathbb{L})\cong H^*(B\mathcal{G}_0)\qquad \Ext_{\hbar}^\bullet(\mathbbm{o},\mathbb{L})\cong H^*(B\mathcal{G}_0)[\hbar].\]
In particular, if $\mathcal{G}_0=I$ then this is a polynomial ring in the Chern classes of the tautological line bundles, and if $\mathcal{G}_0=G[[t]]$, then it is the symmetric polynomials in these classes, which is a polynomial ring in the Chern classes of the tautological rank $n$ bundle.

\subsection{Connection to \texorpdfstring{$G(\ell,p,n)$}{G(l,p,n)}}

While \eqref{eq:line-ext} gives a geometric description of the local operators in these theories, the approach of \cite{WebSD, Webalt} can be used to compute these operators explicitly.  We can match the mass parameters of the theory with the parameters of the Cherednik algebra as in \cite{Webalt}.  More explicitly, the space $\frac{L_1}{\mathcal{G}_1}\times_{\frac{N((t))}{G((t))}}\frac{L_2}{\mathcal{G}_2}$ has:
\begin{enumerate}
\item an action of $\C^*$ by loop
  rotation with equivariant parameter $\ell \hbar$;
\item an action of $GL_\ell$ on 
  $\C^n\otimes \C^{\ell}$ with the action on the first factor trivial;  we will identify the  Chern roots of the tautological bundle
with the parameters $-s_i+\ell\hbar $;
\item an action of $\C^*$ by scalar multiplication on
  $\mathfrak{gl}_n$ with equivariant parameter $k$.
\end{enumerate}
We call these parameters $p$-cyclic if they satisfy the relations of Definition \ref{def:p-cyclic}, that is, if  $s_m + \ell \hbar/p = s_{m + \ell/p}$ for all $m$.  In terms of the loop action on $(\C^n\otimes \C^\ell)[[t]]$, this means that $\tl^{\ell/p}$ has weight $1/p$ under the loop action.    Thus, we can interpret specializing $\hbar$ and $s_k$ to complex numbers as giving us a flavor homomorphism $\C\to \mathfrak{gl}_\ell\times \C$ sending
\[1\mapsto \big(\operatorname{diag}(
\ell\hbar-s_{\ell-1}, \ell\hbar-s_{\ell-2}, \cdots, \ell\hbar-s_{1}, \ell\hbar-s_{0}), \hbar \big)
\]
Note our reversal of the order of subscripts due to the fact that $s_i$ appears with a minus sign.  Using the $p$ cyclic property, we have that 
\begin{multline}\label{eq:flavor}
    \operatorname{diag}(
\ell\hbar-s_{\ell-1}, \ell\hbar-s_{\ell-2}, \cdots, \ell\hbar-s_{1}, \ell\hbar-s_{0})\\
=\operatorname{diag}(
\frac{\ell\hbar}{p}-s_{\lp-1}, \cdots, \frac{\ell\hbar}{p}-s_{0}, \frac{2\ell\hbar}{p}-s_{\lp-1}, \cdots, \frac{2\ell\hbar}{p}-s_{0},\cdots ,  \ell\hbar-s_{\lp-1}, \cdots, \ell\hbar-s_{0})
\end{multline}
\begin{theorem}\label{thm:procesi-ext}
  For any $p$-cyclic mass parameters, the space $\Ext_{\hbar}^{\bullet}(\mathbbm{p}_k, \mathbbm{p}_m)$ is isomorphic to the $m-k$ degree part of the partially spherical Cherednik algebra $\mathsf{H}_p^{\operatorname{psph}}$ with the corresponding parameters. 
\end{theorem}
\begin{proof}
First, we note that this Ext space only depends on the value of $k-m\pmod p$.  If $m-k=qp$ for $q\in \Z$, then we have that multiplication by $t^qI\in G((t))$ gives an isomorphism between the spaces $\mathfrak{I}\oplus U_k$ and $\mathfrak{I}\oplus U_m$, and the graph of this isomorphism gives an isomorphism between the corresponding line operators.  
  
  As before, everywhere we write an element $x\in \mathsf{H}_p$, it is to be understood that that we mean its image $e'xe'$.
  The proof proceeds, as in \cite[Lem. 4.2]{Webalt} by comparing the action of the generators of Proposition \ref{prop:psc_gen2} with natural homology classes.  These are easily read off from the action formulas of Proposition \ref{prop:action}, with the degree 0 elements follow the isomorphism of \cite[Lem. 4.2]{Webalt}.  
  \begin{enumerate}
  \item When $k=m$, the elements $u_i-k\ell\hbar/p$ are sent to the Chern classes $c(\Lambda_i/\Lambda_{i+1})$ of tautological line bundles $\Lambda_i/\Lambda_{i+1}$ on the affine flag variety.  Note that this is compatible with the use of the graph of $t^qI$ between different line operators, since this element does not commute with loop rotation.  
  \item When $k=m$, the elements $r_i-1$ are sent to the classes of the subspace
    \begin{equation} \{(\Lambda_\bullet, x(t), y(t))\in \mathcal{F}\ell\times \mathfrak{I}\times U_k\mid \Lambda_j = \Delta_j \text{ for }j\not\equiv i\hspace{-3mm}\pmod n, x(t)\cdot \Delta_{i-1}\subset \Delta_{i+1} \} \end{equation}
  \item 
The element $y_n^{\ell-1}\tau $ is sent to the
shift 
correspondence \begin{equation}\label{eq:3-correspondence}
\{(\Lambda_\bullet, x(t), y(t))\in \mathcal{F}\ell\times \mathfrak{I}\times U_k\mid
\Lambda_j=\Delta_{j-1}\text{ for all }j\}\end{equation} and 
$x_1^{\ell-1}\sigma$ is sent to the
correspondence \begin{equation}\label{eq:3x-correspondence}\{(\Lambda_\bullet, x(t), y(t))\in \mathcal{F}\ell\times \mathfrak{I}\times U_k\mid
  \Lambda_j=\Delta_{j+1}\text{ for all }j\}.\end{equation}
\item When $k=m+1$, then $U_k\subset U_m$ and $\sigma^{n\ell/p}$ is sent to the correspondence
  \begin{equation}\label{eq:4-correspondence}\{(\Lambda_\bullet, x(t), y(t))\in \mathcal{F}\ell\times \mathfrak{I}\times U_m\mid
    \Lambda_j=\Delta_{j}\text{ for all }j\}.\end{equation}  This looks suspiciously like the identity correspondence, but of course, it is between two different spaces, so this is not even a meaningful possibility.
\item If $k=m-1$, we simply switch the sides of the correspondence, and send $\tau^{n\ell/p}$ to  the class of \begin{equation}\label{eq:5-correspondence}\{(\Lambda_\bullet, x(t), y(t))\in \mathcal{F}\ell\times \mathfrak{I}\times U_k\mid
\Lambda_j=\Delta_{j}\text{ for all }j\}.\end{equation}
\item For $c\in [1,n-1]$ and any $k-m=d\in [1,p-1]$, we send $\left( x_1^{d\ell/p-1} \sigma \right)^c \left( y_{n-i}^{\ell-d\ell/p-1} (n-i,\dots,n) \tau \right)^{n-c} $ to the homology class of
  \begin{equation}\label{eq:6-correspondence}\{(\Lambda_\bullet, x(t), y(t))\in \mathcal{F}\ell\times \mathfrak{I}\times (U_k\cap \tn^{n-c}U_m)\mid
\Lambda_j=\Delta_{j-n+c}\text{ for all }j\}.\end{equation}
\end{enumerate}
In order to check that these satisfy the desired relations, we will confirm that the action of these classes in the GKLO representation \[\Ext_\hbar^{\bullet}(\mathbbm{o},\mathbbm{p}_k)\cong \C[c(\Lambda_0/\Lambda_{1}),\cdots, c(\Lambda_{n-1}/\Lambda_{n}),\hbar]\] 
matches the formulas of Proposition \ref{prop:action}.  For the elements of degree 0, this is already proven in \cite[Lem. 4.2]{Webalt}.  The remaining cases, our cycles consist on an affine space $V$ sitting above a single point in the affine flag variety corresponding to a length 0 element $w$ of the extended affine Weyl group, so their action is easy to check.  They must act by an affine transformation of the Chern classes  $c(\Lambda_i/\Lambda_{i+1})$ corresponding to $w^{-1}$, and multiplication by the Euler class of the normal bundle to $w^{-1}V$ inside $\mathfrak{I}\oplus U_m$.  As usual, calculating this normal class in the $I\times \C^*$-equivariant homology of a contractible space simply requires decomposing the normal bundle under of the action of the torus $T$ and loop $\C^*$ and taking the product of the weights that appear with multiplicity (considered as degree 2 cohomology classes on the classifying space).  
\begin{enumerate}\setcounter{enumi}{2}
    \item While this is covered in \cite[Lem. 4.2]{Webalt}, it will be a useful warm-up to consider these cases.  In \eqref{eq:3-correspondence}, we have $w=\tn$, and $\tn^{-1}V=\mathfrak{I}\oplus U_k$ since $ U_k\subset \tn^{-1}U_k$.  Thus the normal bundle is trivial, and we have the desired action.  On the other hand, in \eqref{eq:3x-correspondence}, we have $w=\tn^{-1}$ and $\tn V=\mathfrak{I}\oplus \tn U_k$, since $\tn U_k\subset U_k$.  Thus, we have to multiply by the Euler class of the quotient $U_k/\tn U_k\cong \C^{\ell}$.  Since the flavor group $GL_\ell$ acts on this as a copy of the defining representation, and $T\times \C^*$ by a single character $\mu$, we have that the Euler class is $ \prod_{m=0}^{\ell-1} (\mu+\ell\hbar-s_m)$.  The character $\mu$ is that of the action on $t^k$ times the first basis vector in $\C^n$, which is exactly the image of $u_1$ under our map.  This completes the proof.
    \item In this case, $w=1$ and $V=U_k$, since $U_k\subset U_m$.  Thus, the Euler class of interest is that of $U_m/U_k$.  This is a sum of $\ell/p$ copies of $\C^n$, which we can identify with $\tl^m(\C^{\ell/p})\otimes \C^n$ where $GL_n$ acts as usual on the right hand side and the flavor and loop tori act by the usual action on the left hand side.  By \eqref{eq:flavor}, the weights for the flavor and loop tori $\tl^m(\C^{\ell/p})$ for correspond to the monomials \[(m+1)\ell\hbar/p-s_0, (m+1)\ell\hbar/p-s_1, \dots, (m+1)\ell\hbar/p-s_{\ell/p}.\]  By the $p$-cyclicity of the parameters, we have that these values are enough to characterize the full set of parameters for the Cherednik algebra.  The weights of $GL_n$ correspond to the tautological bundles, and thus to 
    \[u_1-m\ell\hbar/p, \dots, u_n-m\ell\hbar/p.\]
    Taking the product of all ways taking sums of these elements gives the weights of the tensor product.  The shift of the variables $U_i\mapsto U_i+\ell \hbar/p$ follows from the fact that tautological classes map to tautological classes, but these differ from the $u_i$'s by a shift.  
    \item In this case, $w=1$ and $V=U_m$, since $U_m\subset U_k$.  Thus the Euler class is trivial, and we only have the shift that comes from identifying tautological classes.
    \item In this case, $w=\tn^{d}$; this, together with the shift in identifying tautological classes and the $u_i$'s, this gives the desired affine transformation.  Thus, we need only calculate the Euler class for $U_m/w^{-1}V$.  We have $w^{-1}V=U_m\cap \tn^{c-n}U_k$; for simplicity, consider the case of $m=0,k=d$.  In this case $U_0=(\C^n\otimes \C^\ell)[[t]]$ and $\tn^{c-n}U_k$ is the $\C[[t]]$ submodule spanned by 
    \begin{align*}
        t (e_i&\otimes e_j) && \text{ for }i\in [1,c], j\in [1,\nicefrac{d\ell}{p}]\\
         e_i&\otimes e_j && \,\,\parbox{3in}{ for $i\in [c+1,n], j\in [1,\nicefrac{d\ell}{p}]$\\
    \mbox{}\hspace{.8cm} or $i\in [1,c], j\in [\nicefrac{d\ell}{p}+1,\ell]$}\\ 
         t^{-1} (e_i&\otimes e_j) && \text{ for }i\in [c+1,n], j\in [\nicefrac{d\ell}{p}+1,\ell]
    \end{align*}   Thus, the quotient $U_m/w^{-1}V$ is again a tensor product, which is of the form $\tl^m(\C^{d\ell/p})\otimes \C^c$ and the weights of this tensor product give the desired product, using \eqref{eq:flavor} again.  \qedhere
\end{enumerate}
\end{proof}

\begin{corollary}
  For any $p$-cyclic mass parameters, the space $\Ext_{\hbar}^{\bullet}(\mathbbm{sp}_k, \mathbbm{sp}_m)$ is isomorphic to the $m-k$ degree part of the spherical Cherednik algebra $\mathsf{H}_p^{\operatorname{sph}}$ with the corresponding parameters. 
\end{corollary}
One peculiar seeming aspect of this description is that the space of morphisms only depends on the difference between $k$ and $m$;  this can be explained by the existence of an action of $\Z/p\Z$ on $V[[t]]/G[[t]]$ induced by the action of $\tl^{\lp}$; the action of $\tl^{\ell}=tI$ on this quotient is trivial since it is the action of an element of $G((t)).$ It's clear from the definition that we have isomorphisms
  \begin{equation*}
      (\tl^{\lp})^*\mathbbm{p}_k\cong \mathbbm{p}_{k-1}\qquad (\tl^{\lp})^*\mathbbm{sp}_k\cong \mathbbm{sp}_{k-1}
  \end{equation*}
and in particular $\mathbbm{p}_k\cong \mathbbm{p}_{k-p}$ and   $\mathbbm{sp}_k\cong \mathbbm{sp}_{k-p}$.  
\begin{remark}
The category of D-modules on $N((t))/G((t))$ has a factorization structure, which physicists would view as the operator product expansion of two line operators.  It seems that from this perspective $(\tl^{\lp})^*$ should be viewed as ``tensoring'' with $\mathbbm{sp}_{-1}$.  This operation fits into a larger physics context as we can view our 3d gauge theory as a boundary condition for 4d Yang-Mills for $GL_\ell$, and $\mathbbm{sp}_{-1}$ as the image in the boundary of a line operator in this 4d theory, which plays an important role in the physical interpretation of geometric Langlands.

It would be interesting to make this statement more precise, since it's closely related to the construction of resolutions of Coulomb branches in \cite{BFNline}, which uses a similar pullback construction.
\end{remark}

For $i,j\in [0,p-1]$, let 
\[\overline{i+j}=\begin{cases}i+j & i+j <p\\ i+j-p& i+j\geq p\end{cases}\]
\begin{theorem}\label{thm:lines-cherednik}
We have isomorphisms 
\[\mathsf{H}_p^{\operatorname{psph}}\cong \bigoplus_{i=0}^{p-1} \Ext_{\hbar}^{\bullet}(\mathbbm{p}_0,\mathbbm{p}_i)\qquad \mathsf{H}_p^{\operatorname{sph}}\cong\bigoplus_{i=0}^{p-1} \Ext_{\hbar}^{\bullet}(\mathbbm{sp}_0,\mathbbm{sp}_i)\]
with multiplication 
 $\Ext_{\hbar}^{\bullet}(\mathbbm{p}_0,\mathbbm{p}_i)\times \Ext_{\hbar}^{\bullet}(\mathbbm{p}_0,\mathbbm{p}_j)\to \Ext_{\hbar}^{\bullet}(\mathbbm{p}_0,\mathbbm{p}_{\overline{i+j}})$ given by composition after the identifications
 \begin{equation*}
      \Ext_{\hbar}^{\bullet}(\mathbbm{p}_0,\mathbbm{p}_j)\cong \Ext_{\hbar}^{\bullet}(\mathbbm{p}_i,\mathbbm{p}_{{i+j}})\qquad \mathbbm{p}_{\overline{i+j}}\cong \mathbbm{p}_{i+j}
 \end{equation*}
 induced by pullback (and similarly for $\mathbbm{sp}_*)$.
\end{theorem}
\begin{proof}
The only point that needs to be confirmed here beyond Theorem \ref{thm:procesi-ext} is that the isomorphism of that theorem is compatible with pullback by $\tl^\ell$.  This is clear because both act by shifting the variables $U_i$ in the polynomial representation by $\ell\hbar$; the faithfulness of the polynomial representation assures that this uniquely characterizes the isomorphism.  
\end{proof}
This result opens the door to a geometric approach to the representation theory of $\mathsf{H}_p$, since for any line operator $\mathbb{L}$, we have that $\bigoplus_{i=0}^{p-1} \Ext_{\hbar}^{\bullet}(\mathbb{L},\mathbbm{p}_i)$ is a $\mathsf{H}_p^{\operatorname{psph}}$-module.  It is thus natural to ask which $\mathsf{H}_p^{\operatorname{psph}}$-modules arise this way.  This type of geometric realization of representations of Cherednik algebras appeared in work of Oblomkov and Yun \cite{ObYu}, but could benefit from considerably more exploration.

\excise{
For the generalized case of Cherednik algebras associated with $G(\ell,p,n)$, we propose that the corresponding physical theory has a 1d line defect. This line defect is the vortex operator described in \cite{AG}. Essentially, the 3d theory is coupled to a 1d supersymmetric theory with trivial gauge group in such a way that $k\ell/p$ of the hypermultiplet scalars vanish on the line defect for some $k=1,\dots,p$. The corresponding quiver diagram is:
\begin{equation*}
\tikz[scale=1.1]{
	\node[above left] at (4.5,-3.3) {3d};
	\node[below left] at (4.5,-3.3) {1d};
	\draw (-.28,.28) to[out=135,in=180] (0,.7);
	\draw (0,.7) to[out=0,in=45] (.28,.28);
	\draw (0,0) circle (0.4) node {$n$};
	\draw (-3.3,-2.3) rectangle (-2.7,-1.7);
	\node at (-3,-2){$\ell/p$};
	\draw (-.4,0) -- (-3,-1.7);
	\draw[dashed] (-3,-2.3) -- (-.4,-4);
	\draw (-2.3,-2.3) rectangle (-1.7,-1.7);
	\node at (-2,-2){$\ell/p$};
	\draw (-.35,-.19) -- (-2,-1.7);
	\draw[dashed] (-2,-2.3) -- (-.35,-3.81);
	\node at (-1,-2){$\cdots$};
	\draw (-.3,-2.3) rectangle (.3,-1.7);
	\node at (0,-2){$\ell/p$};
	\draw (0,-.4) -- (0,-1.7);
	\draw[dashed] (0,-2.3) -- (-0,-3.6);
	\draw (.7,-2.3) rectangle (1.3,-1.7);
	\node at (1,-2){$\ell/p$};
	\draw (.19,-.35) -- (1,-1.7);
	\node at (2,-2){$\cdots$};
	\draw (2.7,-2.3) rectangle (3.3,-1.7);
	\node at (3,-2){$\ell/p$};
	\draw (.4,0) -- (3,-1.7);
	\draw (0,-4) circle (0.4) node {$0$};
	\draw[very thick] (-4.5,-3.3) -- (4.5,-3.3);
	\draw[decoration={brace,mirror,raise=5pt},decorate] (1,-2.3) -- node[below=6pt] {$p-k$} (3,-2.3);
	\draw[decoration={brace,mirror,raise=5pt},decorate] (-3,-2.3) -- node[below=6pt] {$k$} (0,-2.3);	
}
\end{equation*}

Recall that the spherical subalgebra of the Cherednik algebra associated with $G(\ell,1,n)$ is isomorphic to the chiral ring of a 3d $\cN=4$ theory with gauge group $U(n)$ and $\ell$ hypermultiplets. The basic idea in understanding the isomorphism is understanding how to match the generators of the algebra and operators in the physical theory. The generators $u_i$ of $eHe$ correspond to the scalar fields $\varphi_a$. The monopole operators shift the scalar fields similarly to how $\sigma$ and $\tau$ shift $u_i$, so the monopole operators correspond to an appropriate power of $\tau$ or $\sigma$. The mass parameters are related to the parameters $s_m$ in $eHe$. The simple case of $n=1$ is shown in Appendix \ref{app:matching}.}

\bibliography{gen}

\def\cprime{$'$} \def\cprime{$'$}
\providecommand{\bysame}{\leavevmode\hbox to3em{\hrulefill}\thinspace}
\providecommand{\MR}{\relax\ifhmode\unskip\space\fi MR }
\providecommand{\MRhref}[2]{%
  \href{http://www.ams.org/mathscinet-getitem?mr=#1}{#2}
}
\providecommand{\href}[2]{#2}
\begin{thebibliography}{GKLO05}

\bibitem[{\'A}DL03]{ADL}
Istv{\'a}n {\'A}goston, Vlastimil Dlab, and Erzs{\'e}bet Luk{\'a}cs,
  \emph{Quasi-hereditary extension algebras}, Algebr. Represent. Theory
  \textbf{6} (2003), no.~1, 97--117. \MR{MR1960515 (2004c:16010)}

\bibitem[BEF]{BEF}
Alexander Braverman, Pavel Etingof, and Michael Finkelberg, \emph{{Cyclotomic
  double affine Hecke algebras}}, \arxiv{1611.10216}.

\bibitem[BFNa]{BFNline}
Alexander Braverman, Michael Finkelberg, and Hiraku Nakajima, \emph{Line
  bundles over coulomb branches}, \arxiv{1805.11826}.

\bibitem[BFNb]{BFN}
\bysame, \emph{{Towards a mathematical definition of Coulomb branches of
  $3$-dimensional $\mathcal N=4$ gauge theories, II}}, \arxiv{1601.03586}.

\bibitem[BK09]{BKKL}
Jonathan Brundan and Alexander Kleshchev, \emph{{Blocks of cyclotomic Hecke
  algebras and Khovanov-Lauda algebras}}, Invent. Math. \textbf{178} (2009),
  451--484.

\bibitem[CM]{CM}
Joseph Chuang and Hyohe Miyachi, \emph{{Hidden {H}ecke algebras and {K}oszul
  duality}}, preprint.

\bibitem[DGGH]{DGGH}
Tudor Dimofte, Niklas Garner, Michael Geracie, and Justin Hilburn, \emph{Mirror
  symmetry and line operators}, \arxiv{1908.00013}.

\bibitem[FO10]{FOgalois}
Vyacheslav Futorny and Serge Ovsienko, \emph{Galois orders in skew monoid
  rings}, Journal of Algebra \textbf{324} (2010), no.~4, 598--630.

\bibitem[GKLO05]{GKLO}
A.~Gerasimov, S.~Kharchev, D.~Lebedev, and S.~Oblezin, \emph{On a class of
  representations of the {Y}angian and moduli space of monopoles}, Comm. Math.
  Phys. \textbf{260} (2005), no.~3, 511--525. \MR{2182434 (2006k:53038)}

\bibitem[Gor08]{Gordonquiver}
I.~G. Gordon, \emph{Quiver varieties, category {$\mathcal O$} for rational
  {C}herednik algebras, and {H}ecke algebras}, Int. Math. Res. Pap. IMRP
  (2008), no.~3, Art. ID rpn006, 69. \MR{2457847 (2010c:16032)}

\bibitem[Gri10a]{GrifJack}
Stephen Griffeth, \emph{Orthogonal functions generalizing {J}ack polynomials},
  Trans. Amer. Math. Soc. \textbf{362} (2010), no.~11, 6131--6157. \MR{2661511}

\bibitem[Gri10b]{Gri}
\bysame, \emph{{Towards a combinatorial representation theory for the rational
  Cherednik algebra of type $G (r, p, n)$}}, Proceedings of the Edinburgh
  Mathematical Society \textbf{53} (2010), no.~2, 419--445.

\bibitem[Gri18]{GrifU}
\bysame, \emph{Unitary representations of cyclotomic rational {C}herednik
  algebras}, J. Algebra \textbf{512} (2018), 310--356. \MR{3841527}

\bibitem[Har]{Hartwig}
Jonas~T. Hartwig, \emph{{Principal Galois orders and Gelfand-Zeitlin modules}},
  \arxiv{1710.04186}.

\bibitem[Jen14]{RolloJenkins}
Rollo Crozier~John Jenkins, \emph{Representations of rational {C}herednik
  algebras: {K}oszulness and localisation}, ProQuest LLC, Ann Arbor, MI, 2014,
  Thesis (Ph.D.)--The University of Edinburgh (United Kingdom). \MR{3389407}

\bibitem[KL09]{KLI}
Mikhail Khovanov and Aaron~D. Lauda, \emph{A diagrammatic approach to
  categorification of quantum groups. {I}}, Represent. Theory \textbf{13}
  (2009), 309--347.

\bibitem[KN18]{KoNa}
Ryosuke Kodera and Hiraku Nakajima, \emph{Quantized {C}oulomb branches of
  {J}ordan quiver gauge theories and cyclotomic rational {C}herednik algebras},
  String-{M}ath 2016, Proc. Sympos. Pure Math., vol.~98, Amer. Math. Soc.,
  Providence, RI, 2018, pp.~49--78. \MR{3821749}

\bibitem[KR11]{KlRa}
Alexander Kleshchev and Arun Ram, \emph{Representations of
  {K}hovanov-{L}auda-{R}ouquier algebras and combinatorics of {L}yndon words},
  Math. Ann. \textbf{349} (2011), no.~4, 943--975. \MR{2777040 (2012b:16078)}

\bibitem[Lau12]{Lauintro}
Aaron~D. Lauda, \emph{An introduction to diagrammatic algebra and categorified
  quantum {$\mathfrak{sl}_2$}}, Bull. Inst. Math. Acad. Sin. (N.S.) \textbf{7}
  (2012), no.~2, 165--270. \MR{3024893}

\bibitem[Lec04]{Lecshuf}
Bernard Leclerc, \emph{Dual canonical bases, quantum shuffles and
  {$q$}-characters}, Math. Z. \textbf{246} (2004), no.~4, 691--732. \MR{2045836
  (2005c:17019)}

\bibitem[LeP19]{LePageThesis}
Elise LePage, \emph{{Cherednik Algebras, Principal Galois Orders, and Coulomb
  Branches}}, Master's thesis, Perimeter Insitute/University of Waterloo,
  Waterloo, ON, 2019.

\bibitem[Liu17]{LiuClifford}
Shoumin Liu, \emph{Clifford theory on rational {C}herednik algebras of
  imprimitive groups}, Indag. Math. (N.S.) \textbf{28} (2017), no.~4, 736--748.
  \MR{3679739}

\bibitem[OY16]{ObYu}
Alexei Oblomkov and Zhiwei Yun, \emph{Geometric representations of graded and
  rational {C}herednik algebras}, Adv. Math. \textbf{292} (2016), 601--706.
  \MR{3464031}

\bibitem[Ros19]{Rostam}
Salim Rostam, \emph{Cyclotomic quiver {H}ecke algebras and {H}ecke algebra of
  {$G(r,p,n)$}}, Trans. Amer. Math. Soc. \textbf{371} (2019), no.~6,
  3877--3916. \MR{3917212}

\bibitem[Rou08]{RouqSchur}
Raphael Rouquier, \emph{{$q$}-{S}chur algebras and complex reflection groups},
  Mosc. Math. J. \textbf{8} (2008), no.~1, 119--158, 184. \MR{2422270
  (2010b:20081)}

\bibitem[RR03]{RamRam}
Arun Ram and Jacqui Ramagge, \emph{Affine {H}ecke algebras, cyclotomic {H}ecke
  algebras and {C}lifford theory}, A tribute to {C}. {S}. {S}eshadri
  ({C}hennai, 2002), Trends Math., Birkh\"{a}user, Basel, 2003, pp.~428--466.
  \MR{2017596}

\bibitem[RSVV16]{RSVV}
Rapha\"el Rouquier, Peng Shan, Michela Varagnolo, and Eric Vasserot,
  \emph{Categorifications and cyclotomic rational double affine {H}ecke
  algebras}, Invent. Math. \textbf{204} (2016), no.~3, 671--786. \MR{3502064}

\bibitem[Suz03]{SuzukiCherednik}
Takeshi Suzuki, \emph{Classification of simple modules over degenerate double
  affine {H}ecke algebras of type {$A$}}, Int. Math. Res. Not. (2003), no.~43,
  2313--2339. \MR{1998469}

\bibitem[SVV14]{SVV}
P.~Shan, M.~Varagnolo, and E.~Vasserot, \emph{Koszul duality of affine
  {K}ac-{M}oody algebras and cyclotomic rational double affine {H}ecke
  algebras}, Adv. Math. \textbf{262} (2014), 370--435. \MR{3228432}

\bibitem[Weba]{WebGT}
Ben Webster, \emph{{Gelfand-Tsetlin modules in the Coulomb context}},
  \arxiv{1904.05415}.

\bibitem[Webb]{WebSD}
\bysame, \emph{{Koszul duality between Higgs and Coulomb categories
  $\mathcal{O}$}}, \arxiv{1611.06541}.

\bibitem[Webc]{WebBKnote}
\bysame, \emph{On graded presentations of {H}ecke algebras and their
  generalizations}, {to appear in Algebraic Combinatorics, \arxiv{1305.0599}}.

\bibitem[Web15]{WebCB}
\bysame, \emph{Canonical bases and higher representation theory}, Compos. Math.
  \textbf{151} (2015), no.~1, 121--166.

\bibitem[Web17]{WebRou}
\bysame, \emph{Rouquier's conjecture and diagrammatic algebra}, Forum Math.
  Sigma \textbf{5} (2017), e27, 71. \MR{3732238}

\bibitem[Web19a]{Webalt}
\bysame, \emph{Representation theory of the cyclotomic {C}herednik algebra via
  the {D}unkl-{O}pdam subalgebra}, New York J. Math. \textbf{25} (2019),
  1017--1047. \MR{4017214}

\bibitem[Web19b]{WebwKLR}
\bysame, \emph{Weighted {K}hovanov-{L}auda-{R}ouquier algebras}, Doc. Math.
  \textbf{24} (2019), 209--250. \MR{3946709}

\end{thebibliography}
\bibliographystyle{amsalpha}
\end{document}